\documentclass[a4paper,12pt]{amsart}
\usepackage{amsmath}
\usepackage{array}
\usepackage{amsthm}
\usepackage{amssymb}
\usepackage[all]{xy}

\usepackage{hyperref}

\theoremstyle{plain}
\newtheorem*{thm*}{Theorem}
\newtheorem*{prop*}{Proposition}
\newtheorem*{rem*}{Remark}

\newtheorem{thm}{Theorem}[section]
\newtheorem{cor}[thm]{Corollary}
\newtheorem{defi}[thm]{Definition}
\newtheorem{prop}[thm]{Proposition}
\newtheorem{lm}[thm]{Lemma}

\newtheorem{claim*}{Claim}
\newtheorem{exple}[thm]{Example}

\numberwithin{equation}{thm}

\theoremstyle{remark}
\newtheorem{rem}[thm]{Remark}

\newcommand{\Z}{\mathbb Z}
\newcommand{\N}{\mathbb N}

\newcommand{\ul}{\underline}

\newcommand{\ti}{{\text{-}}}

\newcommand{\A}{{\mathcal{A}}}
\newcommand{\C}{{\mathcal{C}}}
\newcommand{\D}{{\mathcal{D}}}
\newcommand{\F}{{\mathcal{F}}}
\newcommand{\E}{{\mathcal{E}}}

\newcommand{\M}{{\mathcal{M}}}
\newcommand{\PP}{{\mathcal{P}}}

\newcommand{\Common}{{com \ti mon}}

\usepackage{multicol}

\makeatletter
  {\end{multicols}\if@restonecol\onecolumn\else\clearpage\fi}
\makeatother
	\newdir{>>}{{}*!/3.5pt/:(1,-.2)@^{>}*!/3.5pt/:(1,+.2)@_{>}*!/7pt/:(1,-.2)@^{>}*!/7pt/:(1,+.2)@_{>}}
\newdir{ >>}{{}*!/8pt/@{|}*!/3.5pt/:(1,-.2)@^{>}*!/3.5pt/:(1,+.2)@_{>}}
\newdir{ |>}{{}*!/-3.5pt/@{|}*!/-8pt/:(1,-.2)@^{>}*!/-8pt/:(1,+.2)@_{>}}
\newdir{ >}{{}*!/-8pt/@{>}}
\newdir{>}{{}*:(1,-.2)@^{>}*:(1,+.2)@_{>}}
\newdir{<}{{}*:(1,+.2)@^{<}*:(1,-.2)@_{<}}

\makeindex

\begin{document}

	 \title[Polynomial functors from algebras over a set-operad]{Polynomial functors from algebras over a set-operad and non-linear Mackey functors}

		      \author{Manfred Hartl \&Teimuraz Pirashvili \& Christine Vespa}
		      
		      \address{Universit\'e Lille Nord de France, F-59044 Lille, France
UVHC, LAMAV and FR CNRS 2956, F-59313 Valenciennes, France}
\email{manfred.hartl@univ-valenciennes.fr}

\address{Department of Mathematics, University of Leicester, University Road, Leicester, LE1 7RH, UK}
\email{tp59@leicester.ac.uk}

\address{ Universit\'e de Strasbourg, Institut de Recherche
  Math\'ematique Avanc\'ee, Strasbourg, France. }                        
\email{vespa@math.unistra.fr}

\date{\today}
\maketitle

\begin{flushright}
		      \textit{Dedicated to the memory of Jean-Louis Loday for his generosity and benevolence.}
		      \end{flushright}

\begin{abstract}
\vspace{.3cm}
In this paper, we give a description of polynomial functors from (finitely generated free) groups to abelian groups in terms of non-linear Mackey functors generalizing those given in a paper of Baues-Dreckmann-Franjou-Pirashvili published in $2001$. This description is a consequence of our two main results: a description of functors from (finitely generated free) $\PP$-algebras (for $\PP$ a set-operad) to abelian groups in terms of non-linear Mackey functors and the isomorphism between polynomial functors on (finitely generated free) monoids and those on (finitely generated free) groups. Polynomial functors from (finitely generated free) $\PP$-algebras to abelian groups and from (finitely generated free) groups to abelian groups are described explicitely by their cross-effects and maps relating them which satisfy a list of relations.

\textit{Mathematics Subject Classification:} 18D; 18A25; 55U
\vspace{.3cm}

\textit{Keywords}: polynomial functors; non-linear Mackey functors; set-operads\end{abstract}
\begin{footnote}{ 
The third author is partially  supported by project ANR blanc BLAN08-2\_338236, HGRT and ANR blanc HOGT. The authors gratefully acknowledge the hospitality of the Max Planck Institut f\"ur Mathematik in Bonn where this project started and where this paper was written.
}
\end{footnote}

Polynomial functors
play a prominent role in the representation theory of algebraic groups,
algebraic $K$-theory as well as in the theory of modules over the Steenrod algebra (see, for instance, \cite{Pan-S}). In particular it is a main computational tool for
computing the stable cohomology of classical groups (considered as
discrete groups) with twisted coefficients (see \cite{Betley}, \cite{FFSS} and  \cite{Dja-V}). 

The study of polynomial functors, in their own right, has a long history starting with the work of Schur in $1901$ in \cite{Schur}, even before the notion of a functor was defined by Eilenberg and MacLane. In fact, without using the language of functors, Schur proved that over a field of characteristic zero any polynomial functor is a
direct sum of homogeneous functors and the category of homogeneous
functors of degree $d$ is equivalent to the category of
representations of the symmetric group on $d$ letters. There were many attempts to generalize Schur's theorem for general rings. For polynomial functors from (finitely generated free) abelian groups to abelian groups a satisfactory answer is given in \cite{BDFP} where the authors obtained a description of those polynomial functors in terms of non linear Mackey functors. 

The principal aim of this paper is to generalize this result. The two main results of this paper are a description of functors from (finitely generated free) $\PP$-algebras (for $\PP$ a set-operad) to abelian groups and a description of polynomial functors from (finitely generated free) groups to abelian groups, obtained from the first result and an isomorphism between polynomial functors on (finitely generated free) monoids and those on (finitely generated free) groups.

More precisely, for $\PP$ a set-operad, $Free(\PP)$ the category of (finitely generated free) $\PP$-algebras and $Ab$ the category of abelian groups, we consider the category $Func(Free(\PP), Ab)$ of functors $F: Free(\PP) \to Ab$. The first principal result of this paper is the following:
\begin{thm} \label{Thm-1}
There is a natural equivalence of categories:
$$Func(Free(\PP), Ab) \simeq PMack(\Omega(\PP), Ab)$$
where $PMack(\Omega(\PP), Ab)$ is the category of pseudo-Mackey functors defined in Definition \ref{PMack}.
\end{thm}
This theorem is the generalization of Theorem $0.2$ in \cite{BDFP} giving an equivalence of categories between the functors from (finitely generated free) commutative monoids and pseudo-Mackey functors. In fact, applying Theorem \ref{Thm-1} to the operad $\mathcal{C}om$ we recover Theorem $0.2$ in \cite{BDFP} since $Free(\mathcal{C}om)$ is the category of (finitely generated free) commutative monoids (see Example \ref{ex-alg}).

In \cite{BDFP}, in order to prove Theorem $0.2$ the authors consider the category of finite pointed sets $\Gamma$ and the category of sets where the morphisms are the surjections $\Omega$. An essential tool of the proof is the  Dold-Kan  type theorem proved in \cite{Pira-Dold}, giving a Morita equivalence between the categories  $\Gamma$ and $\Omega$. In this paper, we replace the categories 
$\Gamma$ and $\Omega$ by categories associated to the operad $\mathcal{P}$ denoted by  $\Gamma(\mathcal{P})$ and $\Omega(\mathcal{P})$ (see Definitions \ref{gammaP} and \ref{omegaP}) having the same objects as $\Gamma$ resp. $\Omega$ and whose morphisms are set maps decorated by certain elements depending on $\mathcal{P}$. The extension of the Dold-Kan type theorem to this context (see Theorem \ref{DK}) is an application of the general result of S{\l}omi{\'n}ska in \cite{Slom} describing general conditions which imply Morita equivalences of functor categories. In this paper, however, we give another proof of this result which  has the advantage to provide an explicit description of the functors constituting the equivalence. Theorem \ref{Thm-1}, proved in section \ref{section-4}, then is obtained by combining the following equivalences of categories:

\begin{itemize}
\item by Proposition \ref{Free(P)-span} the category $Free(\mathcal{P})$ is equivalent to the category of Spans on the double  category $\mathcal{S}(\PP)_2$ having as objects finite sets and where horizontal maps are set maps and vertical maps are set maps decorated by certain elements depending on $\mathcal{P}$ (see Definition \ref{double});
\item in Proposition \ref{eq-M-functor} we obtain an equivalence of categories between the functors on this category of Spans and the $\mathcal{M}$-functors on $\Gamma(\mathcal{P})$ defined in Definition \ref{M-functor};
\item in Proposition \ref{eq-Omega} we prove that the Morita equivalence between $\Gamma(\mathcal{P})$ and $\Omega(\mathcal{P})$ provides an equivalence of categories between the $\mathcal{M}$-functors on $\Gamma(\mathcal{P})$ and the pseudo-Mackey functors on $\Omega(\mathcal{P})$ to $Ab$.
\end{itemize}
Applying Theorem \ref{Thm-1} to the operad $\mathcal{A}s$, since $Free(\mathcal{A}s)$ is the category of (finitely generated free) monoids (see Example \ref{ex-alg}), we obtain a description of functors from  (finitely generated free) monoids to abelian groups in terms of pseudo-Mackey functors. In order to obtain a description of polynomial functors from (finitely generated free) groups to abelian groups we would like to obtain a relationship between polynomial functors from (finitely generated free) monoids to abelian groups and those from (finitely generated free) groups to abelian groups. This relation is  the second main result of this paper (see Proposition \ref{pol-gr-mon}). Recall that a reduced functor is a functor such that $F(0)=0$.
\begin{thm} \label{Thm-2}
There is an isomorphism of categories between the reduced polynomial functors of degree $n$ from (finitely generated free) monoids to abelian groups and those from (finitely generated free) groups to abelian groups.
\end{thm}
This theorem is the generalization of Proposition $8.3$ in \cite{BDFP} giving the equivalence of categories between the reduced polynomial functors of degree $n$ from (finitely generated free) commutative monoids to abelian groups and those from (finitely generated free) abelian groups to abelian groups. In \cite{BDFP} the sketched proof of Proposition $8.3$ use, in an essential way, the fact that the composition:
$$Hom_{\Common}(B,C) \times Hom_{\Common}(A,B)  \to Hom_{\Common}(A,C) $$
is a bilinear map. In the non-commutative case, this map is linear in the 
second variable (i.e. $f(g+h)=fg+fh$) but no more linear in the first variable. So we can't adapt directly the proof proposed in \cite{BDFP} in the non-commutative setting. Consequently, to prove Theorem \ref{Thm-2} we have to introduce, in section \ref{section5}, several new tools of independent interest. In particular, in Definition \ref{polmap-setmor} we extend the notion of polynomial maps in the sense of Passi (from monoids
to abelian groups) to maps from a set of morphisms in a (suitable) category
to abelian groups. Using this notion we can characterize, in Theorem \ref{thm:polfunc-polmap}, polynomial functors in terms of their effect on morphism sets, like in the classical case of polynomial
functors between abelian categories. This leads us to introduce, in Proposition \ref{TnZ}, the categories $T_n\overline{\Z}[\C]$ extending the categories $P_n(\A)$ for $\A$ an additive category introduced in \cite{Pira}. These categories have the following important property (see Corollary \ref{Poln(C)=Add(PnC)}):
polynomial functors of degree $n$ from $\C$ to $Ab$ are equivalent to additive functors
from $T_n\overline{\Z}[\C]$  to $Ab$. We then compute, in Corollary \ref{Pn(Free(P))}, the category $T_n\overline{\Z}[\C]$ for $\C$ being the category of $\PP$-algebras (for $\PP$ a set-operad having a binary
operation for which $0_{\PP} \in \PP(0)$ is a unit).  The desired equivalence between
polynomial functors from monoids and those from groups then follows from a
canonical isomorphism $T_n\overline{\Z}[mon] \simeq T_n\overline{\Z}[gr]$   obtained in Theorem \ref{Pnmon=Pngr} from the computation of these categories given in Corollary  \ref{Pn(Free(P))}. 

We point out that these methods also allow to establish a similar isomorphism between polynomial functors from (finitely generated free) loops (i.e. quasigroups with unit) and polynomial functors from (finitely generated free) $\mathcal{M}ag$-algebras where $\mathcal{M}ag$ is the set-operad encoding magmas with unit; the latter can be computed using the results of this paper. This will be considered elsewhere.

Combining Theorems \ref{Thm-1} (for $\PP=\mathcal{A}s$) and \ref{Thm-2} we obtain in section \ref{section-5.5}:

\begin{thm}
There is an equivalence of categories between the reduced polynomial functors of degree $n$ from (finitely generated free) groups to abelian groups
and the category 
$ (PMack(\Omega(\mathcal{A} s), Ab))_{\leq n}$
where $(PMack(\Omega(\mathcal{A} s), Ab))_{\leq n}$ is the full subcategory of $PMack(\Omega(\mathcal{A} s), Ab)$ having as objects the functors which vanish on the sets $X$ of cardinality greater than $n$ and on $\underline{0}$.
\end{thm}

Applying our results to $\PP=\mathcal{C}om$ we recover the result of \cite{BDFP}.

In section \ref{section6} we give a presentation of reduced polynomial functors generalizing the one given in \cite{BDFP}. In particular we obtain an explanation of a phenomenon which could, at first glance seem mysterious in \cite{BDFP}. In fact, in the presentation of reduced polynomial functors given in \cite{BDFP} the longest relation curiously has only $8$ terms,
independantly of the degree of the functor under consideration. In section \ref{section6} we see that this property arises from the fact that the operad $\C om$ is generated by a binary
operation. Since $\A s$ is generated also by a binary
operation the corresponding longest relation also has only $8$ terms. But, for an
operad which has genuine higher than binary operations the corresponding  relation has
a number of terms growing exponentially with the arity of the generating operations.

In the case of reduced polynomial functors of degree two from (finitely generated free) groups to abelian groups  we recover the corresponding result obtained in \cite{Baues-Pira}.

\tableofcontents

\begin{minipage}{\textwidth}
\printindex
\end{minipage}

\textbf{Notations}: We denote by $Gr$ the category of groups, $Ab$ the category of abelian groups, $Mon$ the category of monoids and by $ComMon$ the category of commutative monoids. We denote by $gr$ (resp. $ab$, $mon$ and $common$) the fullsubcategory of $Gr$ (resp. $Ab$, $Mon$ and $ComMon$) having as objects finitely generated free objects.
\section{Recollections on set-operads, double categories and Mackey functors}
\marginpar{Intro}

\subsection{Symmetric set-operads}

Let $\mathcal{S}$ be the skeleton of the category of finite sets with objects $\underline{n}=\{ 1, \ldots, n \}$ for $n \geq 0$. We have $\underline{0}=\emptyset$ which is the initial object of $\mathcal{S}$. 

The category $\mathcal{S}$ is a symmetric monoidal category with product the cartesian product $\times$ and unit object $\underline{1}$.

A set-operad $\PP$ consists of sets $\PP(j)$ with right actions by the symmetric groups $\mathfrak{S}_j$ for $j \geq 0$ (with $\mathfrak{S}_0=\{Id\}$), a unit element $1_{\PP} \in \PP(1)$ and product maps:
$$\gamma_{(n;m_1, \ldots, m_n)}: \PP(n) \times \PP(m_1) \times \ldots \times \PP(m_n) \to \PP(m_1+ \ldots + m_n)$$
that are suitably associative, unital and equivariant. For brevity, we will often write $\gamma$ instead of 
$\gamma_{(n;m_1, \ldots, m_n)}$. (For further information about operads, see for instance \cite{May-Kriz}).

We will assume that the operad $\PP$ is unitary that is $\PP(0)=\underline{1}$. We denote by $0_{\PP}$ the unique element of $\PP(0)$.

\begin{exple} \label{ex-1.1}
\begin{enumerate}
\item For  a set $X$, we have an associated set-operad $\mathcal{E}_X$ such that $\mathcal{E}_X(n)=Hom_{\mathcal{S}}(X^n,X)$.
\item The unitary set-operad $\mathcal{I}$ is defined by $\mathcal{I}(0)=\mathcal{I}(1)=\underline{1}$ and $\mathcal{I}(n)=\emptyset$ for $n>1$. It is the initial object in the category of unitary set-operads.
\item The unitary commutative set-operad $\mathcal{C}om$ is defined by $\forall n \in \N \ \mathcal{C}om(n)=\underline{1}$ equipped with the trivial $\mathfrak{S}_n$-action. It is the terminal object in the category of unitary set-operads. In fact, for $\PP$ a unitary set operad the constant maps:
$$\forall j \in \N, \quad  \PP(j) \rightarrow \mathcal{C}om(j)=\underline{1}$$
form a morphism of operads: $\PP \to \mathcal{C}om$.
\item The unitary associative set-operad $\mathcal{A}s$ is defined by $\forall n \in \N,\ \mathcal{A}s(n)=\underline{\mid \mathfrak{S}_n \mid}$, where $\mid X \mid$ denotes the cardinality of the set $X$, equipped with the action of $\mathfrak{S}_n$ given by right translations.
\end{enumerate}
\end{exple}

Let $\PP$ be a set-operad. A \textit{$\PP$-algebra} is a set $X$ and a morphism of operads $\PP \to \mathcal{E}_X$. We denote by $\PP \ti Alg$ the category of $\PP$-algebras. More explicitely, a $\PP$-algebra is a set $X$ together with a family of maps:
$$ \quad \mu_j:\PP(j) \times_{\mathfrak{S}_j} X^j \to X, \quad j \in \N $$
that is associative and unital. We denote by $1_X$ the element of $X$ given by $\mu_0(0_{\PP}, *)$ where $* \in X^0=\underline{1}$.

The forgetful functor $\PP \ti Alg \to Set$ has a left adjoint functor $\mathcal{F}_{\PP}: Set \to \PP \ti Alg$, where $Set$ is the category of all sets.
The category of finitely generated free $\PP$-algebras $Free(\PP)$ is the full subcategory of $\PP \ti Alg$ whose objects are $\mathcal{F}_{\PP}(\underline{n})$ for $n \geq 0$. Recall that $\mathcal{F}_{\PP}(X)=\underset{n\geq 0}{\amalg} \PP(n) \underset{\mathfrak{S}_n}{\times} X^n$.

\begin{rem}  \label{skelett-Free(P)}
There is an isomorphism of categories between $Free(\PP)$ and the category $\F(\PP)$ whose objects are the sets $\underline{n}$ and whose morphisms are given by $\F(\PP)(\underline{n}, \underline{m})=Free(\PP)(\F_\PP(\underline{n}), \F_\PP(\underline{m}))$.
\end{rem}

\begin{exple} \label{ex-alg}
\begin{enumerate}
\item We have $Free(\mathcal{I})=\Gamma$ where $\Gamma$ is the skeleton of the category of finite pointed sets having as  objects $[n]=\{0, 1, \ldots, n \}$ with  $0$ as basepoint and morphisms the pointed set maps $f: [n] \to [m]$. 
\item We have $Free(\mathcal{C}om)=common$ where $common$ is the category of finitely generated free commutative monoids with unit.
\item We have $Free(\mathcal{A}s)=mon$ where $mon$ is the category of finitely generated free monoids with unit.
\end{enumerate}
\end{exple}

\begin{rem} \label{remarque-section1}
Recall that a bijection $h:\underline{n} \to \underline{n}$ induces a bijection $\hat{h}: \PP(n) \to \PP(n)$ given by $\hat{h}(\omega)=\omega . h^{-1}$.
\end{rem}

\subsection{May-Thomason category $\mathcal{S}(\PP)$ and the categories $\Gamma(\PP)$ and $\Omega(\PP)$}

The following construction is due to May and Thomason in \cite{May-T}.

\begin{defi} \label{May-T} \index{$\mathcal{S}(\PP)$}
Let $\PP$ be a unitary set-operad. The \textbf{May-Thomason category} associated to $\PP$ (or the category of operators): $\mathcal{S}(\PP)$, has as objects the finite sets $\underline{n}=\{1, \ldots, n\}$, and morphisms from $\underline{n}$ to $\underline{m}$ in $\mathcal{S}(\PP)$ are the pairs $(f, \omega^f)$ where $f \in \mathcal{S}(\underline{n},\underline{m})$, $\omega^f \in \Pi_{y \in Y} \PP(\mid f^{-1}(y)\mid)$. For $(f, \omega^f): \underline{n} \to \underline{m}$ and $(g, \omega^g): \underline{m} \to \underline{l}$ morphisms in $\mathcal{S}(\PP)$, the composition $(g, \omega^g)\circ (f, \omega^f)$ in $\mathcal{S}(\PP)$ is the pair $(g \circ f, \omega^{g \circ f})$ where, for $i \in \{1, \ldots l \}$ we have:
$$\omega_i^{g \circ f}=\gamma(\omega_i^g; \omega^f_{j_1}, \ldots, \omega^f_{j_s}).\sigma_i$$
where $g^{-1}(i)=\{j_1, \ldots, j_s \}$ with $j_1<\ldots <j_s$, for $h: \underline{p} \to \underline{q}$ and $\alpha \in \underline{q}$, $\omega^h_\alpha \in \PP(\mid h^{-1}(\alpha) \mid)$ is the component of $\omega^h$ and $\sigma_i$ is the permutation of $\mid (g \circ f)^{-1}(i)\mid$ defined in the following way: let $A_i=(x_1, \ldots, x_{n_j})$ be the natural ordering of $(g \circ f)^{-1}(i)$ and let $B_i=(y_1, \ldots, y_{n_j})$ be its ordering obtained  by regarding it as $\amalg_{g(j)=i} f^{-1}(j)$ so ordered that: 
\begin{itemize}
\item if $j<j'$, all elements of $f^{-1}(j)$ precede all elements of $f^{-1}(j')$;
\item each $f^{-1}(j)$ has its natural ordering as a subset of \underline{n};
\end{itemize}
$\sigma_i$ then is defined by: $\forall k \in \{1, \ldots n_j\}$, $\sigma_i(k)=l$ such that $y_l=x_k$.

\end{defi}

\begin{rem}
For $\underline{n} \in \mathcal{S}(\PP)$ the identity morphism of $\underline{n}$ is $(Id_{\underline{n}}, (1_{\PP}, \ldots, 1_{\PP}) \in \PP(1)^{n})$ and the associativity of composition follows from the definition of operads.
\end{rem}

\begin{exple}[Example of composition in $\mathcal{S}(\PP)$]
Let $(f, \omega_1 \in \PP(1), \omega_2 \in \PP(2)): \underline{3} \to \underline{2}$ be the morphism in $\mathcal{S}(\PP)$ defined by $f(1)=f(3)=2$ and $f(2)=1$ and $(g,\alpha_1 \in \PP(2), \alpha_2 \in \PP(1), \alpha_3 \in \PP(3)): \underline{6} \to \underline{3}$ defined by $g(1)=g(5)=1$, $g(2)=2$ and $g(3)=g(4)=g(6)=3$. We have:
$$(f, \omega_1, \omega_2)(g,\alpha_1, \alpha_2, \alpha_3 )=(fg, \gamma(\omega_1, \alpha_2) \in \PP(1), \gamma(\omega_2, \alpha_1, \alpha_3).\sigma \in \PP(5))$$
where $\sigma=\begin{pmatrix} 1&2&3&4&5\\ 1&3&4&2&5 \end{pmatrix}$ (in this example $A_2=(1,3,4,5,6)$ and $B_2=(1,5,3,4,6)$).
\end{exple}

\begin{exple} $\ $
\begin{itemize}
\item For $\PP=\mathcal{I}$, $\mathcal{S}(\mathcal{I})$ is the subcategory of $\mathcal{S}$ having as morphisms the injections.
\item For $\PP=\C om$. Since $\forall n \geq 0 \quad \C om(n)=\underline{1}$, a morphism in $\mathcal{S}(\PP)$ is a pair $(f, \omega^f)$ where $f \in Hom_{\mathcal{S}}(X,Y)$ and $\omega^f \in  \Pi_{y \in Y}  \C om(\mid f^{-1}(y)\mid)=  \Pi_{y \in Y}  \underline{1}$. Consequently, $\mathcal{S}(\C om)=\mathcal{S}$.

\item For $\PP=\A s$, $\mathcal{S}(\A s)$ is the category of non-commutative sets (see \cite{Pira-PROP}) having as objects the sets $\underline{n}$ and for   $\underline{n}$  and $\underline{m}$ objects in  $\mathcal{S}(\A s)$, an element of  $Hom_{\mathcal{S}(\A s)}(\underline{n},\underline{m})$ is a set map $f: \underline{n} \to \underline{m}$ together with a total ordering of the fibers $f^{-1}(j)$ for all $j \in \underline{m}$. The morphisms of $\mathcal{S}(\A s)$ are called non-commutative maps. For $f: \underline{n} \to \underline{m}$ and $g: \underline{m} \to \underline{k}$ two non-commutative maps, the composition $g \circ f \in Hom_{\mathcal{S}(\A s)}(\underline{n},\underline{k})$ is the set map $g \circ f: \underline{n} \to \underline{k}$ and the total ordering in $(g \circ f)^{-1}(i)$ for $i \in \underline{n}$ is given by the ordered union of ordered sets:
$$(g \circ f)^{-1}(i)=\amalg_{j \in g^{-1}(i)} f^{-1}(j).$$
\end{itemize}
\end{exple}

In general, the disjoint union $\amalg$ is not a coproduct in the category $\mathcal{S}(\PP)$ but it defines a symmetric monoidal category structure on $\mathcal{S}(\PP)$. The following straightforward lemma will be useful in the sequel.
\begin{lm} \label{1.6}
Let $(f, \omega^f): X \to S_1 \amalg S_2$ be a morphism in $\mathcal{S}(\PP)$ then:
$$(f, \omega^f)=(f_1, \omega^{f_1}) \amalg (f_2, \omega^{f_2})$$
where $f_i$ is the restriction of $f$ on $f^{-1}(S_i)$ for $i \in \{1, 2 \}$ and $\omega^{f_1}$ and $\omega^{f_2}$ are defined by:
$$\omega^f=(\omega^{f_1}, \omega^{f_2}) \in \prod_{y \in S_1 \amalg S_2} \PP(\mid f^{-1}(y) \mid)=\prod_{y_1 \in S_1} \PP(\mid f^{-1}(y_1) \mid) \times \prod_{y_2 \in S_2} \PP(\mid f^{-1}(y_2) \mid)$$
$$=\prod_{y_1 \in S_1} \PP(\mid f_1^{-1}(y_1) \mid) \times \prod_{y_2 \in S_2} \PP(\mid f_2^{-1}(y_2) \mid).$$
\end{lm}

In the sequel, we adapt the definition of May-Thomason category to the categories $\Gamma$ and $\Omega$.

\begin{defi}   \label{gammaP} \index{$\Gamma(\PP)$}
Let $\PP$ be a unitary set-operad. Let $\Gamma(\PP)$ be the category of $\PP$-pointed sets having as objects the pointed sets $[n]$ and an element of  $Hom_{\Gamma(\PP)}([n],[m])$ is a pair $(f, \omega^f)$ where $f$ is a pointed set map $f: [n] \to [m]$ and $\omega^f \in \Pi_{y \in [m], y \not= 0} \PP(\mid f^{-1}(y)\mid)$. The composition is defined in the same way as in the category $\mathcal{S}(\PP)$.
\end{defi}

\begin{exple} $\ $
\begin{itemize}
\item For $\PP=\mathcal{I}$, $\Gamma(\mathcal{I})$ is the subcategory of $\Gamma$ having as morphisms the pointed injections.
\item For $\PP=\C om$, we have $\Gamma(\C om)=\Gamma$.
\item For $\PP=\A s$, we have $\Gamma(\PP)=\Gamma(\A s)$ where $\Gamma(\A s)$ is the category   of non-commutative pointed sets having as objects the sets $[n]$ and for   $[n]$  and $[m]$ objects in  $\Gamma(as)$, an element of  $Hom_{\Gamma(as)}([n],[m])$ is a set map $f: [n] \to [m]$ such that $f(0)=0$ together with a total ordering of the fibers $f^{-1}(j)$ for all $j=1, \ldots, m$. We remark that this definition differs from the one given in \cite{Pira-Richter}: here, we do not suppose to have an ordering of $f^{-1}(0)$.
\end{itemize}
\end{exple}

Let $\Omega$ be the category having as objects the sets $\underline{n}$, for $n \geq 0$  and as morphisms the surjective maps.

\begin{defi} \label{omegaP}
Let $\PP$ be a unitary set-operad.  Let $\Omega(\PP)$ be the subcategory of $\mathcal{S}(\PP)$ having as morphisms the morphisms of $\mathcal{S}(\PP)$:  $(f, \omega^f)$ such that $f$ is a surjective map.
\end{defi}

\begin{exple} $\ $ \label{Omega}
\begin{itemize}
\item For $\PP=\mathcal{I}$, $\Omega(\mathcal{I})$ is the subcategory of $\mathcal{S}$ having as morphisms the bijections.
\item For $\PP=\C om$, we have $\Omega(\C om)=\Omega$.
\item For $\PP=\A s$,  $\Omega(\A s)$ is the category having as objects the non-commutative sets and as morphisms the surjective maps.
\end{itemize}
\end{exple}

The following proposition is a direct consequence of definitions.
\begin{prop}
A morphism of unitary set-operad $\phi: \PP \to \mathcal{Q}$ gives rise to functors $\mathcal{S}(\phi): \mathcal{S}(\PP) \to \mathcal{S}(\mathcal{Q})$, $\Omega(\phi): \Omega(\PP) \to \Omega(\mathcal{Q})$ and $\Gamma(\phi): \Gamma(\PP) \to \Gamma(\mathcal{Q})$.
\end{prop}

For $\phi: \PP \to \mathcal{Q}$ a morphism of set-operad, the categories $\Gamma(\PP)$, $\mathcal{S}(\PP)$, $\Omega(\PP)$, $\Gamma(\mathcal{Q})$, $\mathcal{S}(\mathcal{Q})$, $\Omega(\mathcal{Q})$ are connected by the functors in the following diagram:
\begin{equation}\label{diagram-1}
\xymatrix{\Omega(\PP)  \ar@{^{(}->}[r]^{\iota} \ar[d]_{\Omega(\phi)} &   \mathcal{S}(\PP) \ar[r]^j \ar[d]_{\mathcal{S}(\phi)}&\Gamma(\PP) \ar[d]_{\Gamma(\phi)} \\
\Omega(\mathcal{Q})  \ar@{^{(}->}[r]_{\iota} &   \mathcal{S}(\mathcal{Q}) \ar[r]_j&\Gamma(\mathcal{Q})  } 
\end{equation}
where
\begin{itemize}
\item $\iota$ is the obvious faithful functor;
\item $j$ is the functor which adds a basepoint to a set, that is: $j(\underline{n})=[n]$ and for an arrow $(g, \omega^g) \in Hom_{\mathcal{S}(\PP)}(\underline{n}, \underline{m})$  we define   $j(g, \omega^g)=(g_+, \omega^{g_+})$ where the map $g_+ \in Hom_\Gamma([n], [m])$ is given by  $g_+(0)=0$ and ${g_+}_{\mid \{1, \ldots, n\}}=g$ and  
$$\omega^{g_+}=\omega^g \in  \prod_{x \in [m], x \neq 0}\PP(\mid g_+^{-1}(x) \mid )=  \prod_{x \in \underline{m}} \PP(\mid g_+^{-1}(x) \mid)= \prod_{x \in \underline{m}} \PP( \mid g^{-1}(x) \mid).$$
\end{itemize}


\subsection{Double categories, Mackey and pseudo-Mackey functors}
First we recall the notion of double category due to Ehresmann \cite{Ehres}. We refer the reader to \cite{Fied-Loday} for a complete definition of a double category.
\begin{defi}
A double category $\mathbb{D}$ consists of a set of objects, a set of horizontal arrows $A \to B$, a set of vertical arrows $\vcenter{\xymatrix{A \ar@{-->}[d]\\
C}}
$ and a set of squares:
$$\vcenter{\xymatrix{ A \ar@{-->}[d] \ar[r] &B \ar@{-->}[d]\\
C \ar[r] &D
}} $$ satisfying natural conditions. The objects and the horizontal arrows form a category denoted by $\mathbb{D}^h$. The objects and vertical arrows form a category denoted by $\mathbb{D}^v$.
\end{defi}
The following examples of double category will be useful in the sequel.

\begin{defi} \label{double}
For the category $\mathcal{S}(\PP)$ (resp. $\Gamma(\PP)$, resp. $\Omega(\PP)$) we define the  double category $\mathcal{S}(\PP)_2$  (resp. $\Gamma(\PP)_2$, resp. $\Omega(\PP)_2$)  whose objects are objects of $\mathcal{S}$ (resp. $\Gamma$, resp. $\Omega$), horizontal arrows are morphisms in $\mathcal{S}$ (resp. $\Gamma$, resp. $\Omega$), and vertical arrows are morphisms of $\mathcal{S}(\PP)$ (resp. $\Gamma(\PP)$, resp. $\Omega(\PP)$). Squares are diagrams
$$\mathcal{D}= \quad \vcenter{
\xymatrix{ A \ar@{-->}[d]_ -{(\phi, \omega^{\phi})} \ar[r]^-f &B \ar@{-->}[d]^-{(\psi, \omega^{\psi})}\\
C \ar[r]_-g &D
} }$$ 
where $f$ and $g$ are horizontal maps and ${(\phi, \omega^{\phi})}$ and ${(\psi, \omega^{\psi})}$ are vertical maps of $\mathcal{S}(\PP)_2$  (resp. $\Gamma(\PP)_2$, resp. $\Omega(\PP)_2$) , satisfying the following conditions:
\begin{enumerate}
\item the image of $\mathcal{D}$ in $\mathcal{S}$ is a pullback diagram of sets,
\item for all $c \in C$ the bijection $g_*: \PP( \mid\phi^{-1}(c)\mid ) \to  \PP ( \mid \psi^{-1}(g(c))\mid )$  induced by the bijection $\phi^{-1}(c) \to  \psi^{-1}(g(c))$ (see Remark \ref{remarque-section1}) satisfies 
$$g_*(\omega^{\phi}_{c})=\omega^{\psi}_{g(c)}.$$ 
\end{enumerate}
\end{defi}

\begin{exple}
For $\PP=\A s$, $\mathcal{S}(\PP)_2=\mathcal{S}(\A s)_2$ where $\mathcal{S}(\A s)_2$ is the double category defined in \cite{Pira-PROP}. The condition $(2)$ in the previous definition becomes: for all $c \in C$ the induced map $g_*: \phi^{-1}(c) \to \psi^{-1}(g(c))$ is an isomorphism of ordered sets. 
\end{exple}

The following diagram generalizes the diagram (\ref{diagram-1}):

\begin{equation}\label{diagram-1}
\xymatrix{\Omega(\PP)_2  \ar@{^{(}->}[r]^{\iota} \ar[d] &   \mathcal{S}(\PP)_2 \ar[r]^j \ar[d]&\Gamma(\PP)_2 \ar[d] \\
\Omega(\mathcal{Q})_2  \ar@{^{(}->}[r]_{\iota} &   \mathcal{S}(\mathcal{Q})_2 \ar[r]_j&\Gamma(\mathcal{Q})_2  } 
\end{equation}

Remark that a vertical arrow of $\Gamma(\PP)_2$: $(f, \omega^f) \in Hom_{\Gamma(\PP)}([n], [m])$ is of the form $j(\tilde{f}, \omega^{\tilde{f}})$
if and only if $f^{-1}(0)=0$.

For an operad $\PP$ such that $\PP(1)=\{ 1_{\PP}\}$, $(\mathcal{S}(\PP)_2)^h$ and $(\mathcal{S}(\PP)_2)^v$ have the same class of isomorphisms but, in general, the class of isomorphisms of $(\mathcal{S}(\PP)_2)^v$ is bigger than that of $(\mathcal{S}(\PP)_2)^h$. So we can not apply directly the construction considered in section $5$ in \cite{Pira-PROP}. Nevertheless, we can adapt the construction of the Span category in our situation considering the notion of double isomorphisms in a double category.

Let $\mathbb{D}$ be a double category and $A,B$ be two objects of $\mathbb{D}$. Consider a pair $(f,u)$ where:
\begin{itemize}
\item $f: A \to B$ is a horizontal isomorphism of $\mathbb{D}$;
\item 
$u:\xymatrix{ A \ar@{-->}[r] &B }$ is a vertical isomorphism of 
$\mathbb{D}$;
\end{itemize}
provided with a square $\mathcal{D}$ in $\mathbb{D}$:
$$\mathcal{D}= \quad \vcenter{
\xymatrix{ A \ar@{-->}[d]_ -{u} \ar[r]^-f &B \ar@{-->}[d]^-{Id}\\
B \ar[r]_-{Id} &B.
} }$$
Equivalently, the pair $(f,u)$ is provided with the square:
$$\mathcal{D}^*= \quad \vcenter{
\xymatrix{ A \ar@{-->}[d]_ -{Id} \ar[r]^-{Id} &A \ar@{-->}[d]^-{u}\\
A \ar[r]_-{f} &B
} }$$
since we have:
$$\mathcal{D}= \quad \vcenter{
\xymatrix{ A \ar@{-->}[d]_ -{u} \ar[r]^-{Id} &A \ar@{-->}[d]^-{Id} \ar[r]^-{f} & B \ar@{-->}[d]^-{Id} \\
B \ar[r]_-{f^{-1}} &A \ar[r]_-{f}& B
} }$$
where the left-hand square of $\mathbb{D}$  is the horizontal inverse of $\mathcal{D}^*$ and the right-hand square is the square of $\mathbb{D}$ associated to the horizontal map $f$.

\begin{defi}\cite{Grandis-Pare} Let $\mathbb{D}$ be a double category and $A,B$ be two objects of $\mathbb{D}$. A double isomorphism from $A$ to $B$ is a pair $(f,u)$ where:
\begin{itemize}
\item $f: A \to B$ is a horizontal isomorphism of $\mathbb{D}$;
\item 
$u:\xymatrix{ A \ar@{-->}[r] &B }$ is a vertical isomorphism of 
$\mathbb{D}$.
\end{itemize}
equipped with a square $\mathcal{D}$ (or equivalently with $\mathcal{D}^*$) in $\mathbb{D}$:
$$\mathcal{D}= \quad \vcenter{
\xymatrix{ A \ar@{-->}[d]_ -{u} \ar[r]^-f &B \ar@{-->}[d]^-{Id}\\
B \ar[r]_-{Id} &B.
} } \quad (\mathrm{equivalently}\quad \mathcal{D}^*= \quad \vcenter{
\xymatrix{ A \ar@{-->}[d]_ -{Id} \ar[r]^-{Id} &A \ar@{-->}[d]^-{u}\\
A \ar[r]_-{f} &B
} })$$
\end{defi}

\begin{defi}
Two squares of $\mathbb{D}$
$$
\xymatrix{ A \ar@{-->}[d]_ -{a} \ar[r]^-{\alpha} &B \ar@{-->}[d]^-{b}& \mathrm{and} & A' \ar@{-->}[d]_ -{a'} \ar[r]^-{\alpha'} &B\ar@{-->}[d]^-{b'}\\
C \ar[r]_-{\beta} &D&&C \ar[r]_-{\beta'} &D'
}$$
are isomorphic in $\mathbb{D}$ if there exists a double isomorphism $(f,u)$ from $A'$ to $A$ and a double isomorphism $(f',u')$ from $D$ to $D'$ 

such that $\alpha'=\alpha f$, $a'=au$, $\beta'=f' \beta$ and $b'=u'b$.

\end{defi}

The following lemma which describes double isomorphisms in the double categories $\mathcal{S}(\PP)_2$, $\Omega(\PP)_2$ and $\Gamma(\PP)_2$ is a straightforward consequence of the definition of squares in these categories.
\begin{lm} \label{dble-iso-MT}
Double isomorphisms in $\mathcal{S}(\PP)_2$ (resp. $\Omega(\PP)_2$ resp. $\Gamma(\PP)_2$) are pairs $(f,(f,1_{\PP}^{\times \mid A\mid}))$ where $f:A\to A$ is a bijection. 
\end{lm}
\begin{proof}
Let $(f,(g,\omega))$ be a double isomorphism in  $\mathcal{S}(\PP)_2$ (resp. $\Omega(\PP)_2$ resp. $\Gamma(\PP)_2$) from $A$ to $A'$ provided with the square:
$$\mathcal{D}= \quad \vcenter{
\xymatrix{ A \ar@{-->}[d]_ -{(g,\omega)} \ar[r]^-f &A' \ar@{-->}[d]^-{(Id,1_{\PP}^{\times \mid A'\mid})}\\
A' \ar[r]_-{Id} &A'.
} }$$
By definition, $f$ is an isomorphism in  $\mathcal{S}$ (resp. $\Omega$ resp. $\Gamma$) from $A$ to $A'$ so $f$ is a bijection  and $A=A'$ since we consider skeleton of categories. Since the image of $\mathcal{D}$ in $\mathcal{S}$ (resp. $\Omega$ resp. $\Gamma$) is a pullback diagram of sets we have $g=f$. Condition $(2)$ in the definition of $\mathcal{S}(\PP)_2$ (resp. $\Omega(\PP)_2$ resp. $\Gamma(\PP)_2$) implies that $\omega=1_{\PP}^{\times \mid A\mid}$.
\end{proof}

In order to define Mackey functors from $\mathcal{S}(\PP)_2$  and pseudo-Mackey functors from $\Omega(\PP)_2$ we need the following lemma.
\begin{lm} \label{uniq-square}
Let $g: C \to D$ be a horizontal arrow of $\mathcal{S}(\PP)_2$ (resp. $\Omega(\PP)_2$) and $(\psi, \omega^{\psi}): B \to D$ be a vertical arrow of $\mathcal{S}(\PP)_2$ (resp. $\Omega(\PP)_2$), there exists a unique square in $\mathcal{S}(\PP)_2$ (resp. $\Omega(\PP)_2$) up to isomorphism, of the form:
$$\mathcal{D}= \quad \vcenter{
\xymatrix{ A \ar@{-->}[d]_ -{(\phi, \omega^{\phi})} \ar[r]^-f &B \ar@{-->}[d]^-{(\psi, \omega^{\psi})}\\
C \ar[r]_-g &D.
} }$$ 
\end{lm}
\begin{proof}
For $\mathcal{S}(\PP)_2$,  since the image of $\mathcal{D}$ in $\mathcal{S}$ is a pullback diagram of sets, $A$ is the pullback of sets
$$\vcenter{
\xymatrix{ A \ar[d]_ -{\phi} \ar[r]^-f &B \ar[d]^-{\psi}\\
C \ar[r]_-g &D
} }$$
defined up to a bijection $\alpha: A' \to A$.
Then we lift the set map $\phi$ into $\mathcal{S}(\PP)_2$ according to property $(2)$. Indeed, $\omega^{\phi} \in \prod_{c \in C} \PP( \mid \phi^{-1}(c) \mid)$ and for $c \in C$, $\omega_c^{\phi} \in \PP(\mid \phi^{-1}(c) \mid)$ such that $g_*(\omega^{\phi}_c)=\omega_{g(c)}^\psi$, so the lifting $(\phi, \omega^\phi)$ of $\phi$ in $\mathcal{S}(\PP)_2$ exists and is unique. The square 
$$\vcenter{
\xymatrix{ A \ar[d]_ -{(\phi, \omega^{\phi})} \ar[r]^-f &B \ar[d]^-{(\psi, \omega^{\psi})}\\
C \ar[r]_-g &D
} }$$
is defined up to the double isomorphism $(\alpha, (\alpha, 1_{\PP}^{\times \mid A'\mid}))$ from $A'$ to $A$.

For $\Omega(\PP)_2$, in the pullback diagram of sets, since $g$ and $\psi$ are surjective, $f$ and $\phi$ are also surjective. We lift the set map $\phi$ into $\Omega(\PP)_2$ as previously. 
\end{proof}

\begin{rem}
If we replace the category $\mathcal{S}(\PP)_2$ by $\Gamma(\PP)_2$ in Lemma \ref{uniq-square}, the statement is no more true. If fact, for $g: C \to D$ a horizontal arrow of $\Gamma(\PP)_2$ and $(\psi, \omega^{\psi}): B \to D$ a vertical arrow of $\Gamma(\PP)_2$, for $c \in C \setminus\{ 0\}$ such that $g(c)=0$, there is not a unique way to define $\omega_c^{\phi}$. However we have the following analogue of Lemma \ref{uniq-square} for $\Gamma(\PP)_2$.
\end{rem}

\begin{lm} 
Let $g: C \to D$ be a horizontal arrow of $\Gamma(\PP)_2$ such that $g^{-1}(0)=0$ and $(\psi, \omega^{\psi}): B \to D$ be a vertical arrow of $\Gamma(\PP)_2$ there exists a unique square in $\Gamma(\PP)_2$  up to isomorphism, of the form:
$$\mathcal{D}= \quad \vcenter{
\xymatrix{ A \ar@{-->}[d]_ -{(\phi, \omega^{\phi})} \ar[r]^-f &B \ar@{-->}[d]^-{(\psi, \omega^{\psi})}\\
C \ar[r]_-g &D.
} }$$ 
\end{lm}

\begin{proof}
Since $g^{-1}(0)=0$ the lifting $(\phi, \omega^{\phi})$ of $\phi$ into $\Gamma(\PP)_2$ exists and is unique.
\end{proof}

\begin{defi} $\ $ Let $\A$ be a category.
\begin{itemize}
\item A Janus functor $M$ from a double category $\mathbb{D}$ to $\A$ is:
\begin{enumerate}
\item a covariant functor $M_*: \mathbb{D}^h \to \A$
\item a contravariant functor $M^*:  (\mathbb{D}^v)^{op} \to \A$
\end{enumerate}
such that, for each object $A \in \mathbb{D}$ we have $M_*(A)=M^*(A)=M(A)$.
\item A Mackey functor $M=(M_*, M^*)$ from a double category $\mathbb{D}$ to $\A$ is a Janus functor $M$ from  $\mathbb{D}$ to $\A$ such that for each square in $\mathbb{D}$, 
$$\vcenter{
\xymatrix{ A \ar@{-->}[d]_ -\phi \ar[r]^-f &B \ar@{-->}[d]^-\psi\\
C \ar[r]_-g &D
} }$$ we have:
$$M^*(\psi) M_*(g)=M_*(f) M^*(\phi).$$
The category of Mackey functors from $\mathbb{D}$ to $\A$ is denoted by $Mack(\mathbb{D}, \A)$.
\end{itemize}
\end{defi}

We give the construction of the category of spans of a double category: 
\begin{defi} \label{span}
Let $\mathbb{D}$ be a double category, such that for $f \in Hom_{\mathbb{D}^h}(X,B)$ and $\psi \in Hom_{\mathbb{D}^v}(Y,B)$ there exists a unique square of $\mathbb{D}$, up to isomorphism, of the form 
$$ \mathcal{D}=\quad \vcenter{
\xymatrix{ Z \ar@{-->}[d]_ -{\phi_1} \ar[r]^-{f_1} &Y \ar@{-->}[d]^-\psi\\
X \ar[r]_-f &B.
} }$$  The category of spans of $\mathbb{D}$, denoted by $Span(\mathbb{D})$ is defined in the following way:
\begin{enumerate}
\item the objects of $Span(\mathbb{D})$ are those of $\mathbb{D}$;
\item for $A$ and $B$ objects of $\mathbb{D}$, $Hom_{Span(\mathbb{D})}$ is the set of equivalence classes of diagrams 
$$\vcenter{\xymatrix{ X \ar@{-->}[d]_ -\phi \ar[r]^-f &B \\A }}$$
 where $\phi \in Hom_{\mathbb{D}^v}(X,A)$ and $f \in Hom_{\mathbb{D}^h}(X,B)$, for the equivalence relation which identifies the two diagrams  $(\xymatrix{ A &X \ar@{-->}[l]_ -\phi \ar[r]^-f &B })$ and $(\xymatrix{ A &X' \ar@{-->}[l]_ -{\phi'} \ar[r]^-{f'} &B })$ if there exists a double isomorphism $(h,u)$ from $X'$ to $X$ in $\mathbb{D}$ such that:
 $$(\xymatrix{ A &X' \ar@{-->}[l]_ -{\phi'} \ar[r]^-{f'} &B }) =(\xymatrix{ A &X' \ar@{-->}[l]_ -{\phi u} \ar[r]^-{f h} &B }).$$
We denote by $[\xymatrix{ A &X \ar@{-->}[l]_ -\phi \ar[r]^-f &B }]$ the map of $Hom_{Span(\mathbb{D})}(A,B)$ represented by the diagram $\xymatrix{ A &X \ar@{-->}[l]_ -\phi \ar[r]^-f &B }$.
\item the composition of  $[\xymatrix{ A &X \ar@{-->}[l]_ -\phi \ar[r]^-f &B }]$ and $[\xymatrix{ B &Y\ar@{-->}[l]_ -\psi \ar[r]^-g &C }]$ is the class of the diagram $\xymatrix{ A &Z \ar@{-->}[l]_ -{\phi \phi_1} \ar[r]^-{g f_1} &C }$ where 
$$ \mathcal{D}=\quad  \vcenter{
\xymatrix{ Z \ar@{-->}[d]_ -{\phi_1} \ar[r]^-{f_1} &Y \ar@{-->}[d]^-\psi\\
X \ar[r]_-f &B
} }$$
is a square in $\mathbb{D}$.
\end{enumerate}
\end{defi}
The following lemma is a generalization of a result of Lindner \cite{Lindner} in the context of classical Mackey functors.
\begin{lm} \label{Lindner}
Let $\mathbb{D}$ be a double category satisfying the hypothesis of Definition \ref{span}, there is a natural equivalence:
$$Mack(\mathbb{D},\A) \simeq Func(Span(\mathbb{D}), \A).$$
\end{lm}
According to Lemma \ref{uniq-square} the category $Span(\mathcal{S}(\PP)_2)$ exists. We have the following result.
\begin{prop} \label{coprod}
The disjoint union $\amalg$ is a coproduct in $Span(\mathcal{S}(\PP)_2)$.
\end{prop}

\begin{proof}
Let $S_1$, $S_2$ and $T$ be objects of $\mathcal{S}(\PP)_2$ and  $\beta=[\xymatrix{ S_1 &A \ar@{-->}[l]_ -{(f, \omega^f)} \ar[r]^-{f'} &T]}$, $\gamma=[\xymatrix{ S_2 &B \ar@{-->}[l]_ -{(g, \omega^g)} \ar[r]^-{g'} &T]}$,  $I_1=[\xymatrix{ S_1 &S_1 \ar@{-->}[l]_ -{(Id, 1_{\PP}^{\times \mid S_1 \mid})} \ar[r]^-{i_1} & S_1 \amalg S_2 ]}$ and $I_2=[\xymatrix{ S_2 &S_2 \ar@{-->}[l]_ -{(Id,1_{\PP}^{\times \mid S_2 \mid})} \ar[r]^-{i_2} & S_1 \amalg S_2 ]}$ morphisms of $Span(\mathcal{S}(\PP)_2)$.

For $\alpha=[\xymatrix{ S_1\amalg S_2 &A \amalg B \ar@{-->}[l]_ -{(f, \omega^{f}) \amalg (g, \omega^{g})} \ar[r]^-{f' \amalg g'} & T ]}$ we have:
$$\alpha \circ I_1=\beta \mathrm{\quad and \quad} \alpha \circ I_2=\gamma.$$
The uniqueness of $\alpha$ is obtained according to Lemma \ref{1.6}.
\end{proof} 

Although the notion of Mackey functors from $\Omega(\PP)_2$ is well defined, we need to introduce the notion of pseudo-Mackey functors from $\Omega(\PP)_2$ in order to obtain a description of functors from finitely generated free $\PP$-algebras to $Ab$. Let 
$$\mathcal{D}=\quad \vcenter{
\xymatrix{ A \ar@{-->}[d]_ -{(\phi, \omega^{\phi})} \ar[r]^-f &B \ar@{-->}[d]^-{(\psi, \omega^{\psi})}\\
C \ar[r]_-g &D
} }$$ be a square in $\Omega(\PP)_2$. (For fixed morphisms $g$ and $(\psi, \omega^{\psi})$, $\mathcal{D}$ is unique up to isomorphism according to Lemma \ref{uniq-square}).  For any subset $A'$ of $A$ we let $f_{A'}$ and $(\phi, \omega^{\phi})_{A'}$ be the restrictions of $f$ and $(\phi, \omega^{\phi})$ on $A'$ (i.e., if $i: A' \to A$ is the inclusion, we have $(i, \omega^i) \in Hom_{\mathcal{S}(\PP)}(A',A)$ where $\omega^i=(1_{\PP}^{\times \mid A' \mid}, 0_{\PP}^{\times \mid A'^c\mid}) \in \Pi_{y \in i(A')} \PP(\mid i^{-1}(y)\mid) \times \Pi_{y \in i(A')^c}\PP(0)$, so $(\phi, \omega^{\phi})_{A'}=(\phi, \omega^{\phi}) \circ (i, \omega^i)$.) Following \cite{BDFP}, we call $A'$ a $\mathcal{D}$-\textit{admissible subset} of $A$ if  $f_{A'}$ and $\phi_{A'}=\phi \circ i$ are surjections.  We denote by $Adm(\mathcal{D})$ the set of all $\mathcal{D}$-admissible subsets of $A$. 
\begin{defi} \label{PMack}
A pseudo-Mackey functor $M: \Omega(\PP)_2 \to Ab$ is a Janus functor such that:
\begin{enumerate}
\item For  a double isomorphism $(f,u)$ of $\Omega(\PP)_2$ we have: 
$$M_*(f) M^*(u)=Id.$$
\item For a pair of horizontal and vertical morphisms of $\Omega(\PP)_2$ having the same target: $g: C \to D$ and $(\psi, \omega^{\psi}): \xymatrix{ B \ar@{-->}[r] &D}$ and  the unique square,  up to isomorphism, in $\Omega(\PP)_2$
$$\mathcal{D}=\quad \vcenter{
\xymatrix{ A \ar@{-->}[d]_ -{(\phi, \omega^{\phi})} \ar[r]^-f &B \ar@{-->}[d]^-{(\psi, \omega^{\psi})}\\
C \ar[r]_-g &D
} }$$ 
we have
$$M^*(\psi, \omega^\psi) \circ M_*(g)= \sum_{A' \in Adm(A)}M_*(f_{A'}) \circ M^*(\phi, \omega^{\phi})_{A'}.$$
\end{enumerate}
\end{defi}

Note that the notion of pseudo-Mackey functor is not weaker than that of Mackey functor. There is no obvious link between the categories of pseudo-Mackey functors  and of Mackey functors of $\Omega(\PP)_2$.

\section{Polynomial functors}
In this section we generalize the definition of polynomial functors introduced by Eilenberg and MacLane \cite{EML} for covariant and contravariant functors from a pointed monoidal category.

Let $\C$ be a pointed category (i.e. having a null object denoted by $0$) and equipped with a monoidal structure $\vee$ , whose neutral object is $0$. In other words, there is a bifunctor $\vee: \C \times \C \to \C$ together with isomorphisms $X \vee 0 \simeq X$ and $0 \vee X \simeq X$ such that usual properties (associativity, etc) hold.
\begin{exple}
\begin{enumerate}
\item Any pointed category with finite coproducts is an example of such category. In particular, the categories $Gr$, $Ab$, $Mon$ and $ComMon$ are pointed categories with finite coproducts.
\item Although $\Gamma(\PP)$ has no coproduct, it is an example of such a category. In fact, the null object of $\Gamma(\PP)$ is $[0]$ and the symmetric monoidal structure is given by the wedge operation of pointed sets: $[n] \vee [m]=[n+m]$.
\end{enumerate}
\end{exple}
In the sequel, we need the following notations: for $X_1, X_2, \ldots X_n \in \C$, let
$$i^n_{\hat{k}}: X_1 \vee \ldots \vee \hat{X_k} \vee \ldots \vee X_n \to X_1 \vee \ldots \vee {X_k} \vee \ldots \vee X_n$$
be the composition:
$$X_1 \vee \ldots \vee \hat{X_k} \vee \ldots \vee X_n \simeq X_1 \vee \ldots \vee 0 \vee \ldots \vee X_n \to X_1 \vee \ldots \vee {X_k} \vee \ldots \vee X_n$$
where the second map is induced by the unique map $0 \to X_k$. 
Similarly one obtains morphisms:
$$r^n_{\hat{k}}=(1_{X_1}, \ldots, 1_{X_1}, 0, 1_{X_1},  \ldots, 1_{X_n}) : X_1 \vee \ldots \vee {X_k} \vee \ldots \vee X_n \to X_1 \vee \ldots \vee \hat{X_k} \vee \ldots \vee X_n.$$
We have the relation:
$$r^n_{\hat{k}}\ i^n_{\hat{k}}=1_{X_1 \vee \ldots \vee \hat{X_k} \vee \ldots \vee X_n}.$$

In the following,  $\D$ is an abelian category.

\subsection{Cross-effects of covariant functors}
Let $F: \C \to \D$ be a functor.
\begin{defi} \label{cross-eff}
The $n$-th cross-effect of $F$ is a functor $cr_nF: \C^{\times n} \to \D$ (or a multi-functor) defined inductively by
$$cr_1F(X) = ker(F(0) : F(X) \to F(0))$$
$$cr_2F(X_1,X_2) = ker((F(r^2_{\hat{2}}),F(r^2_{\hat{1}}))^t : F(X_1\vee X_2) \to F(X_1)\oplus F(X_2))$$
and, for $n \geq 3$, by
$$ cr_nF(X_1, \ldots, X_n) = cr_2(cr_{n-1}F(-,X_3,\ldots,X_n))(X_1,X_2) .$$
\end{defi}
We write also $F(X_1\mid \ldots \mid X_n)$ instead of $cr_nF(X_1, \ldots, X_n)$.

Recall that a \textit{reduced} functor $F: \C \to \D$ is a functor satisfying $F(0)=0$. We denote by $Func_*(\C, {\mathcal{D}})$ the category of reduced functors $F: \C \to {\mathcal{D}}$. 

There is an alternative description of cross-effects.

\begin{prop} \label{cr_n}
Let $F:\C \to \D$ be a functor. Then the  $n$-th cross-effect \linebreak $cr_nF(X_1, \ldots, X_n)$ is equal to the kernel of the natural homomorphism
$$\hat{r}^F\hspace{1mm} : \hspace{1mm} F(X_1 \vee \ldots \vee X_n) \hspace{1mm} \xrightarrow{\phantom{aaaa}} \hspace{1mm} \bigoplus_{k=1}
^n F(
X_1 \vee \ldots \vee \hat{X_k} \vee \ldots \vee X_n). $$
where $\hat{r}^F$ is the map $(F(r^n_{\hat{1}}), \ldots, F(r^n_{\hat{n}}))^t$.
\end{prop}
The next proposition gives the cross-effect decomposition of $F(X_1 \vee \ldots \vee X_n)$.

\begin{prop} \label{ce-prop} Let $F: \C \to \D$ be a reduced functor.
Then there is a natural decomposition 
$$F(X_1 \vee \ldots \vee X_n) \simeq \bigoplus_{k=1}^{n} \bigoplus_{1 \leq i_1 < \ldots < i_k \leq n} cr_kF(X_{i_1}, \ldots, X_{i_k}) .$$
\end{prop}

	The cross-effects have the following crucial property.
\begin{prop} \label{exact}
The functor $cr_n: Func(\mathcal{C}, \D) \to Func(\mathcal{C}^{\times n} , \D)$ is exact for all $n\ge 1$.
\end{prop}

\begin{proof} For $n=1$ it is a consequence of the natural decomposition $F(X) \simeq cr_1F(X) \oplus F(0)$. For $n=2$ this follows from the snake-lemma. For higher $n$ use induction.
\end{proof}

\begin{defi} \label{poly} 
A functor $F: \C \to \D$  is said to be \textit{polynomial of degree lower or equal to $n$} if $cr_{n+1}F=0$. A reduced functor $F$ is called \textit{linear} if $n=1$ and is called \textit{quadratic} if $n=2$. 
\end{defi}
We denote by $Pol_n(\mathcal{C},\D)$ the full subcategory of $Func_{*}(\mathcal{C},\D)$ consisting of reduced polynomial functors of degree lower or equal to $n$ from $\mathcal{C}$ to $\D$, $Lin(\mathcal{C},\D)$ the subcategory of linear functors and $Quad(\mathcal{C},\D)$ those of quadratic functors.

If $\C$ is a preadditive category, the linear functors are precisely the additive functors. In this case, we will denote $Add(\mathcal{C},\D)$ instead of $Lin(\mathcal{C},\D)$.

The category $Pol_n(\mathcal{C}, \mathcal{D})$ has the following fundamental property which is an immediate consequence of Proposition \ref{exact}.
	\begin{prop} \label{thick}
The category $Pol_n(\mathcal{C}, \D)$ is  thick i.e. closed under quotients, subobjects and extensions.
\end{prop} 

 For $X$ a fixed object of $\C$ we define the universal functor $U_X^{\mathcal{C}}: \C \to Ab$ as follows. For a set $S$,  let $\mathbb{Z}[S]$ denote the free abelian group with basis $S$. 
Since for all $Y \in\C$, $Hom_{\mathcal{C}}(X,Y)$ is pointed with basepoint the zero map, we can define a subfunctor $\mathbb{Z}[0]$ of $\mathbb{Z}[Hom_{\mathcal{C}}(X,-)]$ by $\mathbb{Z}[0](Y)=\mathbb{Z}[\{X \xrightarrow{0} Y \}]$ for $Y \in \C$. 
\begin{defi} \label{UXC}
The universal functor $U_X^{\mathcal{C}}: \C \to Ab$ relative to $X$ is the quotient of $\mathbb{Z}[Hom_{\mathcal{C}}(X,-)]$ by the subfunctor $\mathbb{Z}[0]$.
\end{defi}
Note that $U_X^{\mathcal{C}}$ is the reduced standard projective functor associated with $X$.

To keep notation simple we write $f$ also for the equivalence class in $U_X^{\mathcal{C}}(Y)$ of an element $f$ of $Hom_{\mathcal{C}}(X,Y)$.

\subsection{Cross-effects of contravariant functors}
Let $G$ be a contravariant functor from $\C$ to $\D$.

\begin{defi}
The $n$-th cross-effect of $G$ is a functor $\tilde{cr}_nG: ({\C^{op}})^{\times n} \to \D$ (or a multi-functor) defined inductively by
$$\tilde{cr}_1G(X) = ker(G(0) : G(X) \to G(0))$$
$$\tilde{cr}_2G(X_1,X_2) = ker((G(i^2_{\hat{2}}),G(i^2_{\hat{1}}))^t : G(X_1\vee X_2) \to G(X_1)\oplus G(X_2))$$
and, for $n \geq 3$, by
$$ \tilde{cr}_nG(X_1, \ldots, X_n) = \tilde{cr}_2(\tilde{cr}_{n-1}G(-,X_3,\ldots,X_n))(X_1,X_2) .$$
\end{defi}

\begin{prop}
Let $G$ be a contravariant functor from $\C$ to $\D$. Then the  $n$-th cross-effect \linebreak $\tilde{cr}_nG(X_1, \ldots, X_n)$ is equal to the kernel of the natural homomorphism
$$ \hat{i}^G\hspace{1mm} : \hspace{1mm} G(X_1 \vee \ldots \vee X_n) \hspace{1mm} \xrightarrow{\phantom{aaaa}} \hspace{1mm} \bigoplus_{k=1}
^n G(
X_1 \vee \ldots \vee \hat{X_k} \vee \ldots \vee X_n) $$
where $ \hat{i}^G$ is the map $(G(i^n_{\hat{1}}), \ldots, G(i^n_{\hat{k}}))^t$.
\end{prop}
\begin{prop} \label{contra-coker}
If $G$ is a contravariant functor from $\C$ to $\D$ then for all $n \in \N$:
$$\tilde{cr}_nG(X_1, \ldots, X_n) \simeq coker(\bigoplus_{k=1}^n G(r^n_{\hat{k}})\hspace{1mm} : \hspace{1mm} \bigoplus_{k=1}
^n G(
X_1 \vee \ldots \vee \hat{X_k} \vee \ldots \vee X_n) \hspace{1mm} \xrightarrow{\phantom{aaaa}} \hspace{1mm}G(X_1 \vee \ldots \vee X_n)).$$
\end{prop}

\begin{defi} \label{poly} 
A contravariant functor $G$ from $\C$ to $\D$  is said to be \textit{polynomial of degree lower or equal to $n$} if $\tilde{cr}_{n+1}G=0$. A reduced functor is called \textit{linear} if $n=1$ and is called \textit{quadratic} if $n=2$. 
\end{defi}

Let $G$ be a contravariant functor from $\C$ to $\D$. We can consider the covariant functor $G^{op}: \C \to \D^{op}$ associated to $G$. 
\begin{prop}
Let $G$ be a contravariant functor from $\C$ to $\D$ and $G^{op}: \C \to \D^{op}$ be the covariant functor associated to $G$.
If $\D$ is an abelian category, then:
$$cr_n(G^{op})  \simeq \tilde{cr}_n(G) \quad \forall n \in \N.$$
\end{prop}
\begin{proof}
In an abelian category, finite products and coproducts coincide.
$$\begin{array}{l}
cr_n(G^{op})(X_1, \ldots, X_n)\\
=Ker^{Ab^{op}}(\hat{r}^{G^{op}}\hspace{1mm} : \hspace{1mm} G^{op}(X_1 \vee \ldots \vee X_n) \hspace{1mm} \xrightarrow{\phantom{aaaa}} \hspace{1mm} \bigoplus_{k=1}
^n G^{op}(
X_1 \vee \ldots \vee \hat{X_k} \vee \ldots \vee X_n))\\
=coker^{Ab}(\hat{r}^{G}\hspace{1mm} : \hspace{1mm} G(X_1 \vee \ldots \vee X_n) \hspace{1mm} \xleftarrow{\phantom{aaaa}} \hspace{1mm} \bigoplus_{k=1}
^n G(
X_1 \vee \ldots \vee \hat{X_k} \vee \ldots \vee X_n))\\
\simeq \tilde{cr}_nG(X_1, \ldots, X_n). 
\end{array}$$
Where the latter isomorphism is given by Proposition \ref{contra-coker}.
\end{proof}

\begin{rem}
When we consider polynomial functors with values in a suitable non-abelian category $\D$ (for example, the category of groups or, more generally, a semi-abelian category (see \cite{BB})) and suppose that $ \D^{op}$ has products and kernels, $cr_n(G^{op})$ is defined but in general $cr_n(G^{op}) \not \simeq \tilde{cr}_n(G)$. 
\end{rem}

\subsection{Polynomial functors on a pointed category having finite coproducts}
	
	In section \ref{section5} we will consider polynomial functors on a pointed category having finite coproducts $\C$. In this setting the existence of the folding map $\nabla^n: \vee_{i=1}^{n} X \to X$ gives us an explicit description of the $n$-Taylorisation functor (see Definition \ref{Tn-coker}). Then, using this description, we give several results on polynomial functors in this setting which will be useful in section \ref{section5}.

	We denote by $\Delta^n_{\C}: \C \to \C^{\times n}$ the diagonal functor. For $n=2$ we write $\Delta_{\C}$ instead of $\Delta^2_{\C}$.

\begin{defi} \label{+}
For $F \in Func(\mathcal{C},\D)$ and $X \in \mathcal{C}$, we denote by $S_{n}^F$ the natural transformation $S_{n}^F: (cr_{n}F) \Delta^{n}_{\C} \to F$ given by the composition
$$cr_{n}F (X , \ldots , X) \xrightarrow{inc} F( \vee_{i=1}^{n} X) \xrightarrow{F(\nabla^n)} F(X)$$
where $\nabla^n: \vee_{i=1}^{n} X \to X$ is the folding map.

\end{defi}
In the sequel we denote by $Func(\mathcal{C}, \D)_{\leq n} $ the full subcategory of $Func(\mathcal{C}, \D)$ consisting of polynomial functors of degree lower or equal to $n$.

\begin{defi} \label{Tn-coker}
The $n-$Taylorisation functor $T_n :  Func(\mathcal{C}, \D) \to  Func(\mathcal{C}, \D)_{\leq n} $ is defined by: $T_nF=coker((cr_{n+1}F) \Delta^{n+1}_{\C}  \xrightarrow{S_{n+1}^F} F)$  for $F \in Func(\mathcal{C}, \D)$. We call $T_1$ the linearization functor and   $T_2$   the quadratization functor.
\end{defi}

Let $U_n: Func(\mathcal{C}, \D)_{\leq n} \to Func(\mathcal{C}, \D)$ denote the forgetful (i.e.\ inclusion)      functor.

\begin{prop} \label{Tn-prop}
The $n-$Taylorisation functor $T_n :  Func(\mathcal{C}, \D) \to  (Func(\mathcal{C}, \D))_{\leq n}$ is a left adjoint to $U_n$. The unit of the adjunction is the natural epimorphism $t_n^F: F \to T_nF$ which is an isomorphism if $F$ is polynomial of degree $\le n$.

\end{prop}

Thus, we obtain the diagram:

$$\xymatrix{
     &&F  \ar[dl]_{t_{n+1}^F}   \ar[d]_{t_n^F}  \ar[dr]^{t_{n-1}^F} &&& \\
    \ldots   \ar[r] & T_{n+1}F \ar[r]^{\pi_{n+1}}&  T_{n}F \ar[r]^{\pi_{n}}  &T_{n-1}F \ar[r]^{\pi_{n-1}}& \ldots \ar[r]&T_1F \ar[r] & T_0F=0.
      }$$

To keep notation simple we write $t_n$ instead of $t_n^F$ when the functor $F$ is understood.


	\begin{lm}\label{Tn(GoF)} Let $\C,\D$ be pointed categories with finite sums,   $\E$ be an abelian category and $\xymatrix{\C\ar[r]^F &\D\ar[r]^G &\E}$ be functors where $F$ is reduced. Then the natural transformation
	\[T_n(F^*t_n^G)\colon T_n(G\circ F) \longrightarrow T_n(T_nG\circ F)\] is an isomorphism.
	\end{lm}
	
	\begin{proof}
	To construct an inverse of $T_n(F^*t_n^G)$ we first show that the epimorphism\linebreak $\xymatrix{t_n^{GF}\colon G\circ F \ar@{->>}[r]& T_n(G\circ F)}$ factors through the  epimorphism $F^*t_n^G$.

	Consider the following diagram.

	$$\xymatrix{
	cr_{n+1}G(FX,\ldots,FX) \ar@/_6pc/[dd]_{(S_{n+1}^{G})_{FX}}
	\ar@{.>}[d] \ar@{{ >}->}[r]^-{\iota^G} & G(F(X)^{\vee n+1})\ar[d]_{G((Fi_1,\ldots,Fi_{n+1}))}
	\ar[r]^-{\hat{r}^G} & \prod_{k=1}^{n+1} G(F(X)^{\vee n})\ar[d]_{\prod G((Fi_1,\ldots,Fi_{n}))}\\
	cr_{n+1}(G\circ F)(X,\ldots,X)\ar[d]^{(S_{n+1}^{GF})_X}
	\ar@{{ >}->}[r]^-{\iota^{GF}}  & (G\circ F)(X^{\vee n+1}) \ar[dl]^{GF(\nabla^{n+1}_X)}
	\ar[r]^-{\hat{r}^{GF}}&\prod_{k=1}^{n+1} G(F(X)^{\vee n})\\
	GF(X) & &
	}$$
	The right-hand square commutes: for $1\le k\le n+1$, denoting by $pr_k$ the projection to the $k$-th factor, we have
	\begin{eqnarray*}
pr_k \circ \hat{r}^{GF}  \circ G((Fi_1,\ldots,Fi_{n+1})) &=& GF(r^{n+1}_{\hat{k}}) G((Fi_1,\ldots,Fi_{n+1})) \\
&=& G((F(r^{n+1}_{\hat{k}}i_1),\ldots,F(r^{n+1}_{\hat{k}}i_{n+1}))\\
&=&G((Fi_1,\ldots,Fi_{n}) r^{n+1}_{\hat{k}}) \quad\mbox{since $F$ is reduced}\\
&=&G((Fi_1,\ldots,Fi_{n})) G(r^{n+1}_{\hat{k}})\\
&=&pr_k \circ \mbox{$\prod $} G((Fi_1,\ldots,Fi_{n}))  \circ \hat{r}^{G}
\end{eqnarray*}
 whence the dotted arrow exists making the left-hand square commute. By definition of the map $S_{n+1}$ the right-hand triangle commutes , and so does the (graphically degenerate) left-hand triangle since $GF(\nabla^{n+1}_X)G((Fi_1,\ldots,Fi_{n+1})) = G(\nabla^{n+1}_{FX})$. Thus $S_{n+1}^G$ factors through $S_{n+1}^{GF}$ which means that $t_n^{GF} = {\rm coker}(S_{n+1}^{GF})$ factors through $F^*t_n^{G}={\rm coker}(t_n^G)_{FX}$, as
$\xymatrix{	t_n^{GF}\colon (G\circ F)(X) \ar@{->>}^-{F^*t_n^{G}}[r] & (T_nG)\circ F(X) \ar@{->>}[r]^-{\overline{t_n^{GF}}} & T_n(G\circ F)(X)}$\hspace{-1mm}. As $T_n(G\circ F)$ is polynomial of degree $\le n$, the map $\overline{t_n^{GF}}$ further factors as a composition $$\xymatrix{
\overline{t_n^{GF}}\colon (T_nG)\circ F 
	 \ar@{->>}[r]^-{t_n^{H}} &T_n(T_nG \circ F)  
	 \ar@{->>}[r]^-{\overline{\overline{t_n^{GF}}}} & T_n(G\circ F)}$$
	where $H=(T_nG)\circ F$. It remains to check that $\overline{\overline{t_n^{GF}}}$ is an inverse of $T_n(F^*t_n^G)$:
	\begin{eqnarray*}
	\left( \overline{\overline{t_n^{GF}}} \circ T_n(F^*t_n^G)\right) t_n^{GF}&=& \overline{\overline{t_n^{GF}}}\circ t_n^H \circ F^*t_n^G\\&=&
	\overline{t_n^{GF}} \circ F^*t_n^G\\&=&  t_n^{GF}
	\end{eqnarray*}
	whence $\overline{\overline{t_n^{GF}}} \circ T_n(F^*t_n^G)=1$. As $T_n(F^*t_n^G)$ is an epimorphism this suffices to conclude.
	\end{proof}

\begin{lm}\label{TnMDelm} Let $\C$ be a pointed category with finite sums and $\D$ be abelian. Let $M\colon \C^m\to \D$ be a multireduced multifunctor (i.e. such that 
$M(X_1,\ldots,X_n)=0$ if $X_k=0$ for some $1\le k\le m$). Denote by 
$\Delta^m\colon\C \to \C^m$ the diagonal functor. Then for $1\le n<m$ one has $T_n(M\Delta^m)=0$ (in other words, $M\Delta^m$ is cohomogenous of degree $\le n$).
\end{lm}

\begin{proof}
We may assume that $n=m-1$, thanks to the natural factorization of the morphism
$\xymatrix{t_n\colon
M\Delta^m \ar@{-{>>}}[r]^-{t_{m-1}} & T_{m-1}(M\Delta^{m}) \ar@{-{>>}}[r]^-{\overline{t_{n}}} & T_n(M\Delta^{m})
}$, where $\overline{t_n}=\pi_{n+1} \ldots \pi_{m-1}$ (see the diagram after the Proposition \ref{Tn-prop}). So it suffices to show that the map $S_m^{M\Delta^{m}}\colon cr_m(M\Delta^{m})\circ \Delta^{m} \to M\Delta^{m}$ is a pointwise epimorphism. For $X\in \C$, $$\xymatrix{(S_m^{M\Delta^{m}})_X\colon  cr_m(M\Delta^{m})(X,\ldots,X)\ar@{{ >}->}[r]^-{\iota_{X,\ldots,X}^{M\Delta^{m}} } & M(X^{\vee m},\ldots, X^{\vee m}) \ar[rr]^-{M(\nabla^m,\ldots,\nabla^m)} & &M(X,\ldots,X). }$$ But $M(i_1,\ldots,i_m)\colon 
M(X,\ldots,X)\to M(X^{\vee m},\ldots, X^{\vee m})$ is a section of $M(\nabla^m,\ldots,\nabla^m)$ which takes values in $cr_m(M\Delta^{m})(X,\ldots,X)$ since for $1\le k\le m$, $M\Delta^{m}(r^m_{\hat{k}}) M(i_1,\ldots,i_m)$ $ =0$ since $r^m_{\hat{k}}i_k=0$ and $M$ is multireduced. Hence $(\iota_{X,\ldots,X}^{M\Delta^{m}})^{-1}M(i_1,\ldots,i_m)$ is a section of $(S_m^{M\Delta^{m}})_X$, whence the latter  indeed is an epimorphism.
\end{proof}

The following proposition will be useful in section \ref{section5}.
	\begin{prop}\label{polopol} [Theorem $1.9$ in  \cite{PiraSS} or Proposition $1.6$  in  \cite{JMcotriples}]
	Let $\xymatrix{\C\ar[r]^F & \A\ar[r]^G& \mathcal{B}}$ be functors where $\A$ and $\mathcal{B}$ are abelian categories. If $F$ is polynomial of degree $\leq n$ and $G$ is polynomial of degree $\leq m$ then $G \circ F$ is polynomial of degree $\leq nm$.
	\end{prop}


\section{Dold-Kan type theorem for $\Gamma(\PP)$}
In \cite{Pira-Dold} it is proved that the categories $\Gamma$ and $\Omega$ are Morita equivalent i.e. we have an equivalence of categories $Func(\Gamma, Ab) \simeq Func(\Omega, Ab)$. In \cite{BDFP}, the authors use this result and its contravariant version. The aim of this section is to generalize this Morita equivalence to the category $\Gamma(\PP)$. 
More precisely, we have the following result.

\begin{thm} \label{DK}
Let $\PP$ be a unitary set operad, there are natural equivalences:
$$\xymatrix{
Func(\Gamma(\PP), Ab)  \ar@<-.5ex>[r]_-{cr}  & Func(\Omega(\PP), Ab) \ar@<-.5ex>[l]_-{i_!}}$$
$$\xymatrix{
Func(\Gamma(\PP)^{op}, Ab)  \ar@<-.5ex>[r]_-{\tilde{cr}}  & Func(\Omega(\PP)^{op}, Ab) \ar@<-.5ex>[l]_-{i_*}}.$$
\end{thm}
These equivalences are already established in \cite{Slom}. In this paper, S{\l}omi{\'n}ska describes general conditions which imply Morita equivalences of functor categories. The previous theorem is an application, given in section $2.7$, of her general result. However, we here give another proof of this theorem. We obtain the following explicit description of the natural equivalences appearing in the previous theorem.

\begin{prop} \label{DK-explicit}
For $F_* \in Func(\Gamma(\PP), Ab)$ 
$$cr(F_*)(\underline{n})=cr_nF_*([1], \ldots , [1])$$
where $cr_n$ is the $n$-th cross-effect of a covariant functor, and for $(f, \omega^f) \in Hom_{\Omega(\PP)}(\underline{n}, \underline{m})$ we denote by $(f_+, \omega^{f_+})$ the unique extension of $(f, \omega^f)$ in $\Gamma(\PP)$, $cr(F_*)(f, \omega^f)$ is the morphism induced by restriction of $F(f_+, \omega^{f_+})$.

For $F^* \in Func(\Gamma(\PP)^{op}, Ab)$)
$$\tilde{cr}(F^*)(\underline{n})=\tilde{cr_n}F^*([1], \ldots , [1])$$
where $\tilde{cr_n}$ is the $n$-th cross-effect of a contravariant functor, and for $(f, \omega^f) \in Hom_{\Omega(\PP)}(\underline{n}, \underline{m})$ $\tilde{cr}(F^*)(f, \omega^f)$ is defined as previously.

For $G_* \in Func(\Omega(\PP),Ab)$, 
$$i_!(G_*)([n])=\bigoplus_{\mu \subset \underline{n}} G_*(\underline{ \mid \mu \mid}).$$
For $(g,\omega^g) \in Hom_{\Gamma(\PP)}([n], [m])$ the restriction of the morphism $i_!(G_*)(g,\omega_g)$ to the summand $G_*( \underline{ \mid \mu \mid})$ of $i_!(G_*)([n])$ is $0$ if $0 \in g(\mu)$ and is the morphism:
$$G_*(g_{|\mu}, \omega^{g_{|\mu}}): G_*(\underline{ \mid \mu \mid}) \to G_*(\underline{ \mid g(\mu) \mid})$$
followed by the injection into $i_!(G_*)([m])$ otherwise.

For $G^* \in Func(\Omega(\PP)^{op},Ab)$,
$$i_*(G^*)([n])=\bigoplus_{\mu \subset \underline{n}} G^*(\underline{ \mid \mu \mid}).$$

For $(h,\omega^h) \in Hom_{\Gamma(\PP)^{op}}([n], [m])$ the restriction of the morphism $i_*(G^*)(h,\omega_h)$ to the summand $G^*( \underline{ \mid \mu \mid})$ of $i_*(G^*)([n])$ is $0$ to the factors  $G^*( \underline{ \mid \nu \mid})$ such that $h(\nu) \neq \mu$ and is the morphism:
$$G^*(h_{|\nu}, \omega^{h_{|\nu }}): G^*(\underline{ \mid \mu \mid}) \to G^*(\underline{ \mid \nu \mid})$$
to the factors  $G^*( \underline{ \mid \nu \mid})$ such that  $h(\nu)= \mu$.
\end{prop} 

The previous explicit description of natural equivalences let us to prove the following result.
\begin{prop}
The natural equivalences of Theorem \ref{DK}
 are natural in $\PP$ (i.e. for a morphism of set-operad $\alpha: \PP \to \mathcal{Q}$ we have commutative diagrams
$$\xymatrix{
Func(\Gamma(\PP), Ab)  \ar@<-.5ex>[r]_-{cr}  & Func(\Omega(\PP), Ab) \ar@<-.5ex>[l]_-{i_{!}}\\
Func(\Gamma(\mathcal{Q}), Ab)  \ar@<-.5ex>[r]_-{cr} \ar[u] & Func(\Omega(\mathcal{Q}), Ab) \ar@<-.5ex>[l]_-{i_{!}}  \ar[u] } \quad 
\xymatrix{
Func(\Gamma(\PP)^{op}, Ab)  \ar@<-.5ex>[r]_-{\tilde{cr}}  & Func(\Omega(\PP)^{op}, Ab) \ar@<-.5ex>[l]_-{i_{ *}}\\
Func(\Gamma(\mathcal{Q})^{op}, Ab)  \ar@<-.5ex>[r]_-{\tilde{cr}} \ar[u] & Func(\Omega(\mathcal{Q})^{op}, Ab) \ar@<-.5ex>[l]_-{i_{*}}  \ar[u] }
$$

where the vertical maps are induced by precomposition by the functors $\Gamma(\PP) \to \Gamma(\mathcal{Q})$ and $\Omega(\PP) \to \Omega(\mathcal{Q})$ induced by $\alpha$.
\end{prop}

\begin{rem} For $\PP$ a unitary set-operad, we have a unique morphism $\PP \to \C om$ since $\C$om is the terminal object in the category of unitary set-operads (see Example \ref{ex-1.1} (2)). So, we obtain that the natural equivalences $cr, i_!,\tilde{cr}, i_*$ extend the Morita equivalence  between $\Gamma$ and $\Omega$ in \cite{Pira-Dold}, i.e. we have commutative diagrams
$$\xymatrix{
Func(\Gamma(\PP), Ab)  \ar@<-.5ex>[r]_-{cr}  & Func(\Omega(\PP), Ab) \ar@<-.5ex>[l]_-{i_{!}}\\
Func(\Gamma, Ab)  \ar@<-.5ex>[r]_-{cr} \ar@{^{(}->}[u] & Func(\Omega, Ab) \ar@<-.5ex>[l]_-{i_{!}}  \ar@{^{(}->}[u] } \quad 
\xymatrix{
Func(\Gamma(\PP)^{op}, Ab)  \ar@<-.5ex>[r]_-{\tilde{cr}}  & Func(\Omega(\PP)^{op}, Ab) \ar@<-.5ex>[l]_-{i_{ *}}\\
Func(\Gamma^{op}, Ab)  \ar@<-.5ex>[r]_-{\tilde{cr}} \ar@{^{(}->}[u] & Func(\Omega^{op}, Ab) \ar@<-.5ex>[l]_-{i_{*}}  \ar@{^{(}->}[u] }
$$
where the vertical functors are induced by precomposition by the  functors $ \Omega(\PP) \to \Omega(\C om)= \Omega$ and $ \Gamma(\PP) \to \Gamma(\C om)=\Gamma$ induced by the unique morphism $\PP \to \C om$.
\end{rem}

To prove Proposition \ref{DK-explicit}, we use an argument combining the Moebius inversion and the theorem of Gabriel-Popescu. This argument has been used in \cite{Vespa1} to obtain a description of the category $Func(Span(\E^q_{deg}(\mathbb{F}_2)), \E)$ in terms of modules over orthogonal groups where $\E^q_{deg}(\mathbb{F}_2)$ is the category of $\mathbb{F}_2$-quadratic spaces and $\E$ is the category of $\mathbb{F}_2$-vector spaces and in the appendix of \cite{Dja-V} to give an alternative proof of the Dold-Kan type theorem given in  \cite{Pira-Dold} which is very similar to the proof of Theorem \ref{DK} presented here. 

First, we recall the two principal ingredients.
\begin{thm}[Theorem 3.9.2 \cite{Stanley}] \label{poset}
Let $(X, \leq)$ be a finite poset, in which every pair $\{x,y\}$ has a greatest lower bound $x \wedge y$. Let $0$ denote the smallest element in $X$. Suppose that $X$ has a greatest element $1$. Let $R$ be a ring (with identity element $1_R$), and suppose that $\alpha \mapsto e_{\alpha}$ is a map from $X$ to $R$ with the properties that $e_{\alpha}e_{\beta}=e_{\alpha \wedge \beta}$ for any $\alpha, \beta \in X$, and $e_1=1_R$. For $\alpha \in X$, define
$$f_{\alpha}=\sum_{\beta \leq \alpha}\mu_X(\beta, \alpha) e_{\beta},$$
where $\mu_X$ is the M\"{o}bius function of $X$. Then the elements $f_{\alpha}$, for $\alpha \in X$, are orthogonal idempotents of $R$, whose sum is equal to $1_R$. Furthermore, we have $e_{\alpha}=\sum_{\beta \leq \alpha} f_{\beta}$.
\end{thm}
\begin{thm}[Gabriel-Popescu theorem; \cite{Popescu} Corollaire 6.4 p 103] \label{Popescu}
For any abelian category $\C$ the following assertions are equivalent.
\begin{enumerate}
\item The category $\C$ has arbitrary direct sums and a set  $\{ P_i\}_{i \in I}$  of  small projective generators of $\C$.
\item The category $\C$ is equivalent to the subcategory  $Func^{add}(\mathcal{Q}^{op}, Ab)$ of $Func(\mathcal{Q}^{op}, Ab)$ whose objects are additive functors  (i.e. functors satisfying $F(f+g)=F(f)+F(g)$ where $f$ and $g$ are morphisms of $Hom_{\mathcal{Q}^{op}}(V,W)$) and $\mathcal{Q}$ is the full subcategory of $\C$ whose set of objects is $\{ P_i \mid i \in I \}$.
\end{enumerate}
\end{thm}
The proof of Proposition \ref{DK-explicit} relies on the following lemmas.

\begin{lm} \label{DK-lm1}
Let $R$ be the ring $\Z[End_{\Gamma(\PP)}([n])]$ (the identity element of $R$ is $(Id_{[n]}, 1_{\PP}^{\times n})$), $A$ be a subset  of $[n]$ containing $0$ and $(e_A, \omega^{e_A})\in End_{\Gamma(\PP)}([n])$ given by $e_A(j)=j$ if $j \in A$ and $e_A(j)=0$ else, and 
$$\omega^{e_A}=(1_{\PP}, \ldots 1_{\PP}, 0_{\PP}, \ldots, 0_{\PP}) \in  \prod_{y \in [n], y \not= 0} \PP(\mid e_A^{-1}(y) \mid )=\prod_{a \in A, a \not= 0} \PP(1) \times \prod_{a' \in A^c} \PP(0) $$
 where $1_{\PP} \in \PP(1)$ and $0_{\PP} \in \PP(0)$. Then the elements $f_A$ defined by:
$$f_A=\sum_{B \subset A} \mu_{\PP os([n])}(B,A) (e_B, \omega^{e_B})=\sum_{B \subset A} (-1)^{\mid A \setminus B \mid} (e_B, \omega^{e_B})$$
are orthogonal idempotents of $R$ such that $\sum_{A \subset [n]} f_A=(Id_{[n]}, (1_{\PP}, \ldots, 1_{\PP}) \in \PP(1)^{\times n})$, for $\PP os([n])$ the poset of subsets of $[n]$ containing $0$ ordered by the inclusion.
\end{lm}
\begin{proof}
Let $[n]$ be an object in $\Gamma(\PP)$. The pointed set $[n]$ is a greatest element of $\PP os([n])$. Every pair $\{A, B\}$ of elements of $\PP os([n])$ admits a greatest lower bound given by the intersection $A \cap B$. Let $R$ be the ring $\Z[End_{\Gamma(\PP)}([n])]$. To an element $A$ of $\PP os([n])$ we associate the endomorphism of $[n]$ in $\Gamma(\PP)$ $(e_A, \omega^{e_A}) $ defines in the statement.
We have, 
$$(e_{[n]}, \omega^{e_{[n]}})=(Id_{[n]}, (1_{\PP}, \ldots, 1_{\PP}) \in \PP(1)^{\times n})  \mathrm{\quad and \quad}(e_A, \omega^{e_A}) (e_B, \omega^{e_B})=(e_{A \cap B}, \omega^{e_{A \cap B}}).$$ Consequently, by Theorem \ref{poset} the elements $f_A$ defined in the statement
are orthogonal idempotents of $R$ such that $ \sum_{A \subset [n]} f_A=(Id_{[n]}, (1_{\PP}, \ldots, 1_{\PP}))$. 
\end{proof}

In the following proposition we give another description of the idempotents $f_A$ which will be useful in the sequel.

\begin{prop} \label{idemp-explic}
Let $A$ be a subset of $[n]$. We have the following identity:
$$f_A=(e_A, \omega^{e_A})\underset{B\subsetneq A}{\prod}((Id_{[n]}, 1_{\PP}^{\times n})-(e_B, \omega^{e_B})).$$
\end{prop}

\begin{proof}
The proof of this proposition is the same that the proof of Proposition $4.39$ in \cite{Vespa1} mutatis mutandis.
\end{proof}

Let $P_{[n]}^{\Gamma(\PP)}=\Z[Hom_{\Gamma(\PP)}([n], -)]$ be the standard projective objects of $Func(\Gamma(\PP), Ab)$.

\begin{lm} \label{DK-lm2}
Let $A$ be  a subset of $[n]$ and $B$ be a subset of $[m]$, we have:
$$Hom_{Func(\Gamma(\PP), Ab)}(P_{[n]}^{\Gamma(\PP)} f_A, P_{[m]}^{\Gamma(\PP)} f_B) \simeq \Z[Hom_{\Omega(\PP)}(\underline{\mid B \setminus\{0 \} \mid }, \underline{\mid A \setminus\{0 \} \mid })].$$
\end{lm}

\begin{proof}
By the Yoneda lemma, we have 
\begin{eqnarray*}
Hom_{Func(\Gamma(\PP), Ab)}(P_{[n]}^{\Gamma(\PP)} f_A,  P_{[m]}^{\Gamma(\PP)} f_B) &\simeq &f_A P_{[m]}^{\Gamma(\PP)}([n]) f_B\\
&=&f_A \Z[Hom_{\Gamma(\PP)}([m], [n])] f_B
\end{eqnarray*}

Consider the linear map
$$\kappa: \Z[Hom_{\Omega(\PP)}(\underline{\mid B \setminus \{ 0 \} \mid}, \underline{\mid A \setminus \{ 0 \} \mid} )] \to f_A \Z[Hom_{\Gamma(\PP)}([m], [n])] f_B$$
given by 
$$\kappa(\alpha, \omega^{\alpha})=f_A (\tilde{\alpha}, \omega^{\tilde{\alpha}}) f_B$$
where $\tilde{\alpha}$ is the unique extension of $\alpha$ such that $\tilde{\alpha}=0$ on  $(B \setminus\{ 0 \})^c$ and $\omega^{\tilde{\alpha}}=(\omega^\alpha, 0_\PP, \ldots, 0_\PP) \in \Pi_{y \in [n], y \neq 0} \PP(\mid \tilde{\alpha}^{-1}(y) \mid) =\Pi_{y \in \underline{A \setminus \{0 \}}} \PP(\mid {\alpha}^{-1}(y) \mid) \times \Pi_{y \in A^c} \PP(0) $. We will prove that $\kappa$ is an isomorphism. First, it is straightforward to see that, for $(t, \omega^t) \in Hom_{\Gamma(\PP)}([m],[n])$:
\begin{enumerate}
\item $(e_A, \omega^{e_A})(t, \omega^t)=(t, \omega^t)(e_{t^{-1}(A)}, \omega^{t^{-1}(A)})$;
\item $(t, \omega^t)=(t, \omega^t) (e_E, \omega^{e_E})$ for $E=\{0\} \cup (t^{-1}(0))^c$;
\item $(e_C, \omega^{e_C}) f_D= \left \lbrace 
\begin{array}{ll}
f_D \quad & \mathrm{if}\ D \subset C\\
0 & \mathrm{otherwise.}
\end{array}
\right.$
\end{enumerate}
We compute 
\begin{eqnarray*}
f_A (t, \omega^t) f_B &=& \sum_{C \subset A} \mu_{\PP os}(C,A) (e_C, \omega^{e_C})  (t, \omega^t) f_B \mathrm{\ by\ definition\ of\ } f_A\\
&=&(t, \omega^t) \sum_{C \subset A} \mu_{\PP os}(C,A) (e_{t^{-1}C}, \omega^{e_{t^{-1}C}}) f_B \mathrm{\ by\ } (1)\\
&=&(t, \omega^t) \sum_{t(B) \subset C \subset A} \mu_{\PP os}(C,A) f_B \mathrm{\ by\ } (3)\\
&=&\left \lbrace 
\begin{array}{ll}
(t, \omega^t) f_B \quad & \mathrm{if}\ t(B)=A\\
0 & \mathrm{otherwise}
\end{array}
\right.\\
&=&\left \lbrace 
\begin{array}{ll}
(t, \omega^t)(e_E, \omega^{e_E})f_B \quad & \mathrm{if}\ t(B)=A,\ \mathrm{for}\ E=\{0\} \cup (t^{-1}(0))^{c}\\
0 & \mathrm{otherwise} \quad \mathrm{by}\ (2)
\end{array}
\right.\\
&=&\left \lbrace 
\begin{array}{ll}
(t, \omega^t)f_B \quad & \mathrm{if}\ t(B)=A\ \mathrm{and}\ B \subset E=\{0\} \cup (t^{-1}(0))^{c}\\
0 & \mathrm{otherwise} \quad \mathrm{by}\ (3).
\end{array}
\right.\\
\end{eqnarray*}
Since $(t, \omega^t)f_B$ depends only of the restriction of $t$ on $B$, we deduce from the previous calculation that a non-zero element of the form $f_A (t, \omega^t) f_B $ induces a  map  in $\Omega(\PP)$ from $B \setminus \{0 \}$ to $A \setminus \{0 \}$: $(t, \omega^t)(e_B, \omega^{e_B})_{\mid B\setminus \{0 \}}$. Furthermore 
$$\kappa((t, \omega^t)(e_B, \omega^{e_B})_{\mid B \setminus \{0 \}})=f_A (t, \omega^t) (e_B, \omega^{e_B}) f_B=f_A (t, \omega^t) f_B$$
where the last equality is given by $(3)$.
We deduce that $\kappa$ is surjective.

To prove that $\kappa$ is bijective, we use a rank argument. On the one hand we have the isomorphism:
$$\bigoplus_{A \subset [n]; B \subset [m]} f_A \Z[Hom_{\Gamma(\PP)}([m], [n])]f_B \simeq  \Z[Hom_{\Gamma(\PP)}([m], [n])]$$
and, on the other hand, the bijection
$$Hom_{\Gamma(\PP)}([m],[n]) \simeq \coprod_{A \subset [n]; B \subset [m]} Hom_{\Omega(\PP)}(B \setminus \{0\}, A \setminus \{0\})$$
which associates to $(t, \omega^{t}) \in Hom_{\Gamma(\PP)}([m],[n])$ the elements $B=\{0\} \cup (t^{-1}(0))^{c}    $ , $A=t([m])$ and the map in $Hom_{\Omega(\PP)}(B \setminus \{0\}, A \setminus \{0\})$ induced by $(t, \omega^{t})$.

\end{proof}

\begin{lm} \label{Q^op-eq}
Let $Q$ be the full subcategory of $Func(\Gamma(\PP), Ab)$ whose set of objects is $\{ P_{[n]}^{\Gamma(\PP)} f_A \}_{n, A \subsetneq [n]}$, there is a natural equivalence of categories:
$$Q^{op} \simeq \Z[\Omega(\PP)]$$
where $ \Z[\Omega(\PP)]$ is the linearization of the category $\Omega(\PP)$ (i.e. $Obj( \Z[\Omega(\PP)])=Obj(\Omega(\PP))$ and $Hom_{ \Z[\Omega(\PP)]}(A,B)=\Z[Hom_{\Omega(\PP)}(A,B)]$).
\end{lm}
\begin{proof}
Let $\Phi: Q^{op} \to \Z[\Omega(\PP)]$ be the functor defined by:
$$\Phi(P_{[n]}^{\Gamma(\PP)} f_A)=\underline{\mid A \setminus \{ 0 \} \mid }$$
and for $f \in Hom_{Q_{[n]}^{op} }( P_{[n]}^{\Gamma(\PP)} f_A, P_{[n]}^{\Gamma(\PP)} f_B)$ we define $F(f)$ using the morphism $\kappa^{-1}$ introduced in the proof of Lemma \ref{DK-lm2}. To prove that $\Phi$ is a functor we have to verify that $\kappa$ is compatible with the composition. 
Let $c_\C$ be the composition map in the category $\C$ and consider $(\alpha, \omega^{\alpha}) \in Hom_{\Omega(\PP)}(B \setminus \{ 0 \}, A \setminus \{ 0 \} )$ and  $(\beta, \omega^{\beta}) \in Hom_{\Omega(\PP)}(C \setminus \{ 0 \}, B \setminus \{ 0 \} )$. In one hand we have:
$$\kappa c_{\mathbb{Z}[\Omega(\PP)]}((\alpha, \omega^{\alpha}), (\beta, \omega^{\beta}))=f_A(\widetilde{\alpha \circ \beta}, \omega^{\widetilde{\alpha \circ \beta}}) f_C$$
and in the other hand we have:
$$c_{Q^{op}}( \kappa(\alpha, \omega^{\alpha}),  \kappa(\beta, \omega^{\beta}))=f_A(\tilde{\alpha}, \omega^{\tilde{\alpha}}) f_B f_B (\tilde{\beta}, \omega^{\tilde{\beta}}) f_C.$$
We have:
\begin{align*}
&f_A(\tilde{\alpha}, \omega^{\tilde{\alpha}}) f_B f_B (\tilde{\beta}, \omega^{\tilde{\beta}}) f_C= f_A(\tilde{\alpha}, \omega^{\tilde{\alpha}}) f_B(\tilde{\beta}, \omega^{\tilde{\beta}}) f_C \mathrm{\quad since\ } f_B\mathrm{\ is\ an\ idempotent}\\
&= f_A(\tilde{\alpha}, \omega^{\tilde{\alpha}}) \underset{D \subset B}{\sum}\mu_\PP(D,B)(e_D, \omega^{e_D}) (\tilde{\beta}, \omega^{\tilde{\beta}}) f_C \mathrm{\quad by\ definition\ of\ } f_B\\
&= f_A(\tilde{\alpha}, \omega^{\tilde{\alpha}}) \underset{D \subset B}{\sum}\mu_\PP(D,B)(\tilde{\beta}, \omega^{\tilde{\beta}}) (e_{\tilde{\beta}^{-1}(D)}, \omega^{e_{\tilde{\beta}^{-1}(D)}}) f_C \mathrm{\quad by\ } (1)\mathrm{\ in \ the\  proof\ of\ Lemma\ \ref{DK-lm2}}\\
&= f_A(\tilde{\alpha}, \omega^{\tilde{\alpha}}) \underset{\tilde{\beta}(C) \subset D \subset B}{\sum}\mu_\PP(D,B)(\tilde{\beta}, \omega^{\tilde{\beta}})  f_C \mathrm{\quad by\ } (3)\mathrm{\ in \ the\  proof\ of\ Lemma\ \ref{DK-lm2}}\\
&=\left \lbrace 
\begin{array}{ll}
 f_A(\tilde{\alpha}, \omega^{\tilde{\alpha}}) (\tilde{\beta}, \omega^{\tilde{\beta}})  f_C &\mathrm{if\ } \tilde{\beta}(C)=B\\
0  &\mathrm{otherwise}
\end{array}
\right.
\end{align*}
Since $(\beta, \omega^{\beta}) \in Hom_{\Omega(\PP)}(C \setminus \{ 0 \}, B \setminus \{ 0 \} )$, $\beta$ is a surjective map and $\tilde{\beta}(C)=B$ so:
\begin{align*}
f_A&(\tilde{\alpha}, \omega^{\tilde{\alpha}}) f_B f_B (\tilde{\beta}, \omega^{\tilde{\beta}}) f_C= f_A(\tilde{\alpha}, \omega^{\tilde{\alpha}}) (\tilde{\beta}, \omega^{\tilde{\beta}})  f_C= f_A(\widetilde{\alpha \circ \beta}, \omega^{\widetilde{\alpha \circ \beta}})  f_C
\end{align*}
and $\Phi$ is a functor.

Since the morphism $\kappa$ is bijective, $\Phi$ is fully faithful. For $\underline{k} \in \Omega(\PP)$, the relation $\Phi(P_{[n]}^{\Gamma(\PP)}f_{[k]})=\underline{k}$ proves that $\Phi$ is essentially surjective.
\end{proof}

\begin{proof}[Proof of Theorem \ref{DK}]
By Lemma \ref{DK-lm1}, we have the decomposition 
$$P_{[n]}^{\Gamma(\PP)}=\bigoplus_{A \subset [n]} P_{[n]}^{\Gamma(\PP)} f_A.$$
Applying Theorem \ref{Popescu} to $\C=Func(\Gamma(\PP), Ab)$ and $Q$ being the full subcategory of $Func(\Gamma(\PP), Ab)$ whose set of objects is $\{ P_{[n]}^{\Gamma(\PP)} f_A\}_{n,A\subset[n]}$ we obtain an equivalence of categories:
$$\sigma: Func(\Gamma(\PP)_n, Ab) \to Func^{add}(Q_n^{op}, Ab))$$
given by: $\sigma(G)( P_{[n]}^{\Gamma(\PP)} f_A)=Hom_{Func(\Gamma(\PP)_n, Ab)}(P_{[n]}^{\Gamma(\PP)} f_A, G)$ for $G \in Func(\Gamma(\PP)_n, Ab)$.
So, we have the following sequence of natural equivalences:
\begin{eqnarray*}
Func(\Gamma(\PP), Ab)& \simeq& Func^{add}(Q^{op}, Ab)\\
& \simeq & Func^{add}(\Z[\Omega(\PP)], Ab)  \mathrm{\quad  \ by\  Lemma \   \ref{Q^op-eq}}\\
& \simeq & Func(\Omega(\PP), Ab) 
\end{eqnarray*}
Now we prove the explicit formulae of the natural equivalences given in the statement.

To obtain the explicit description of $cr$, recall that $cr$ is given by the composition:
$$cr: Func(\Gamma(\PP), Ab) \xrightarrow{\simeq} Func^{add}(Q^{op}, Ab) \xrightarrow{(\hat{\Phi})^*}  Func(\Z[\Omega(\PP)], Ab) \to Func(\Omega(\PP), Ab)$$
where $\hat{\Phi}: \Z[\Omega(\PP)] \to Q^{op}$ is an inverse of $\Phi$ given by:
$$\hat{\Phi}(\underline{n})=P^{\Gamma(\PP)}_{[n]}f_{[n]}$$
and by $\kappa$ on morphisms. For a functor $F:\mathcal{C} \to Ab$ let $\hat{F}:\Z[\mathcal{C}] \to Ab$ denote the $\Z$-linear extension of $F$. Then:
$$cr(F)(\underline{n})=Hom_{Func(\Gamma(\PP), Ab)}(P_{[n]}^{\Gamma(\PP)} f_{[n]}, F)\simeq \hat{F}(f_{[n]})F([n]) \simeq Im(\hat{F}(f_{[n]})).$$
By Proposition \ref{idemp-explic} we have: 
$$f_{[n]}=\underset{B\subsetneq [n]}{\prod}((Id_{[n]}, 1_{\PP}^{\times n})-(e_B, \omega^{e_B}))$$
so
\begin{equation} \label{eq-idem}
\hat{F}(f_{[n]})=\underset{B\subsetneq [n]}{\prod}\hat{F}((Id_{[n]}, 1_{\PP}^{\times n})-(e_B, \omega^{e_B}))=\underset{B\subsetneq [n]}{\prod}(1-F(e_B, \omega^{e_B})).
\end{equation}
Recall that for two commuting idempotents $E_A, E_B$ we have: 
$$Im(E_A E_B)=Im(E_A) \cap Im(E_B).$$
So, by (\ref{eq-idem}) we obtain:
$$Im(F(f_{[n]}))=\underset{B\subsetneq [n]}{\bigcap} Im(1-F(e_B, \omega^{e_B}))=\underset{B\subsetneq [n]}{\bigcap} Ker(F(e_B, \omega^{e_B})).$$
For $B \in \PP os([n])$ such that $B \neq [n]$, let $p_B: [n] \to B$ (resp. $i_B: B \to [n]$) be the epimorphism (resp. monomorphism) such that:
$$e_B=i_B \circ p_B.$$
We have $(e_B, \omega^{e_B})=(i_B, \omega^{e_B}) \circ (p_B, 1_\PP^{\times \mid B \mid}).$
Since $(Id_B, 1_\PP^{\times \mid B \mid})=(p_B, 1_\PP^{\times \mid B \mid})\circ (i_B, \omega^{e_B}) $ we have:
$$Ker(F(e_B, \omega^{e_B}))=Ker(F(p_B, 1_\PP^{\times \mid B \mid})).$$
Therefore:
$$Im(F(f_{[n]}))=\underset{B\subsetneq [n]}{\bigcap} Ker(F(p_B, 1_\PP^{\times \mid B \mid}))=\underset{B\subsetneq [n], \mid B \mid=n}{\bigcap} Ker(F(p_B, 1_\PP^{\times \mid B \mid}))$$
since for $B \subsetneq [n]$ such that $\mid B \mid <n$ there exists $B' \subsetneq[n]$ with $\mid B' \mid =n$ such that $p_B$ factors through $p_{B'}$. We deduce that:
$$cr(F)(\underline{n})=cr_nF([1], \ldots, [n])$$
by Proposition \ref{cr_n}.

Now we show that for $G \in Func(\Omega(\PP), Ab)$, the functor $i_!(G): \Gamma(\PP) \to Ab$ corresponds to $G$ under the above equivalences.  We must check that:
$$\Phi^*\hat{G}=Hom(-, i_!(G)): Q^{op} \to Ab.$$
We have $\Phi^*\hat{G}(P_{[n]}^{\Gamma(\PP)} f_A)=G(\underline{\mid A \setminus \{0\} \mid})$ while
$$Hom(P_{[n]}^{\Gamma(\PP)} f_A, i_!(G)) \simeq Im \left(i_!(G)(f_A) \right)$$
$$=Im\left(\underset{B \subset A}{\sum} \mu_{\PP os[n]} (B,A) i_!(G)(e_B, \omega^{e_B}): i_!(G)([n]) \to  i_!(G)([n])\right).$$
Let $\mu \subset \underline{n}$. Then $i_!(G)(e_B, \omega^{e_B})_{\mid G(\underline{\mid \mu \mid})}$ is $0$ if $\mu \not \subset B$ and is the injection 
$$I_{\mu}: G(\underline{\mid \mu \mid}) \hookrightarrow i_!(G)([n])$$
if $\mu \subset B$.
In particular, it is $0$ if $\mu \not \subset A$. For $\mu \subset A$, we have:
$$\underset{B \subset A}{\sum}\mu_{\PP os[n]} (B,A) i_!(G)(e_B, \omega^{e_B})_{\mid G(\mid \mu \mid)}=\underset{\mu \subset B \subset A}{\sum}(-1)^{\mid A \setminus B\mid} I_\mu=\left \lbrace 
\begin{array}{ll}
I_\mu\quad  \mathrm{if}\ \mu=A \setminus\{0\}\\
0  \quad \mathrm{otherwise}
\end{array}
\right.\\$$
since $B$ contains $0$ while $\mu$ does not. It follows that $Hom(P_{[n]}^{\Gamma(\PP)}.f_A, i_!(G)) \simeq G(\underline{\mid A \setminus \{0 \} \mid })$, as desired.
\end{proof}


\section{Equivalence between functors from $\PP \ti alg$ and pseudo Mackey functors} \label{section-4}
The aim of this section is to prove the following theorem.
\begin{thm} \label{thm-pal}
There is a natural equivalence of categories
$$Func(Free(\PP), Ab) \simeq PMack(\Omega(\PP), Ab)$$
where $Free(\PP)$ is the category of finitely generated free $\PP$-algebras and $\Omega(\PP)$ is the category defined in \ref{omegaP}.
\end{thm}
The proof is divided into three steps. In the first one, we prove the equivalence of categories $Free(\PP)\simeq Span(\mathcal{S}(\PP)_2)$ where $Span(\mathcal{S}(\PP)_2)$ is the category of Spans given in Definition \ref{span}. Then, we define in Definition \ref{M-functor} the notion of $\mathcal{M}$-functor from $\PP$-pointed sets. In the second part, we prove in Proposition \ref{eq-M-functor} the equivalence of categories $Mack(\mathcal{S}(\PP)_2, Ab)\simeq \mathcal{M} \ti func(\Gamma(\PP), Ab)$. Recall that $Mack(\mathcal{S}(\PP)_2, Ab)\simeq Func(Span(\mathcal{S}(\PP)_2), Ab)$ by Lemma \ref{Lindner}. Finally, we prove in Proposition \ref{eq-Omega} the equivalence $\mathcal{M} \ti func(\Gamma(\PP), Ab) \simeq PMack(\Omega(\PP), Ab)$. 

In the following proposition we extend an analogous result already present in \cite{Pira-PROP} and corresponding to the particular case of a set-operad satisfying $\PP(1)=\{1_{\PP}\}$. 
\begin{prop} \label{Free(P)-span}
For $\PP$ a unitary set-operad, there is an equivalence of categories:
$$Free(\PP) \simeq Span(\mathcal{S}(\PP)_2)$$ where $Span(\mathcal{S}(\PP)_2)$ is the category of Spans defined in Definition \ref{span}. 
\end{prop}
The proof of this result relies on the following lemma.
\begin{lm} \label{lm-equi}

\begin{enumerate}
\item
Let $f,f'\in Hom_{\mathcal{S}}(\underline{n}, \underline{m})$ and $\omega, \omega' \in \PP(n)$. 
The following assertions are equivalent in the double category $\mathcal{S}(\PP)_2$:
\begin{enumerate}
\item The diagrams $\xymatrix{ \underline{1} &\underline{n}  \ar@{-->}[l]_ -{(s,\omega)} \ar[r]^-f &\underline{m}  }$ and $\xymatrix{ \underline{1} &\underline{n}  \ar@{-->}[l]_ -{(s,\omega')} \ar[r]^-{f'} &\underline{m}  }$ are equivalent.
\item 
$(\omega ,(f(1), \ldots, f(n)))=(\omega',(f'(1), \ldots, f'(n))) \in \PP(n) \underset{\mathfrak{S}_n}{\times} \underline{m}^n.$
\end{enumerate}
\item
We have: $Hom_{Span(\mathcal{S}(\PP)_2) }(\underline{1}, \underline{m}) \simeq \amalg_n \PP(n) \underset{\mathfrak{S}_n}{\times} \underline{m}^n$.
\end{enumerate}
\end{lm}
\begin{proof}
\begin{enumerate}
\item
If the two diagrams in the statement are equivalent, by definition, there exists a double isomorphism in $\mathcal{S}(\PP)_2$: $(h, (h, 1_\PP^{\times n}))$ where $h:  \underline{n} \to \underline{n}$ is a bijection such that: $f'=fh$ and $(s', \omega')=(s, \omega)(h, 1_\PP^{\times n})=(s,\omega h)$. So, we have:
$$h.(\omega', (f'(1), \ldots, f'(n)))=(\omega'.h^{-1}, h.(f'(1), \ldots, f'(n)))=(\omega'.h^{-1}, (f'(h^{-1}(1)), \ldots, f'(h^{-1}(n))))$$
$$=(\omega, (f(1), \ldots, f(n)).$$
Therefore: $(\omega ,(f(1), \ldots, f(n)))=(\omega',(f'(1), \ldots, f'(n))) \in \PP(n) \underset{\mathfrak{S}_n}{\times} \underline{m}^n$. \\
Conversely, if $(\omega ,(f(1), \ldots, f(n)))=(\omega',(f'(1), \ldots, f'(n))) \in \PP(n) \underset{\mathfrak{S}_n}{\times} \underline{m}^n$, there exists $\sigma \in \mathfrak{S}_n$ such that 
$$(\omega' ,(f'(1), \ldots, f'(n)))=\sigma.(\omega,(f(1), \ldots, f(n)))= (\omega. \sigma^{-1},\sigma.(f(1), \ldots, f(n)))$$
$$=(\omega. \sigma^{-1},(f(\sigma^{-1}(1)), \ldots, f(\sigma^{-1}(n)))$$
So the double isomorphism $(\sigma^{-1}, (\sigma^{-1}, 1_{\PP}^{\times n}))$ provides an isomorphism between the two diagrams.
\item
Consider the map : $\phi: Hom_{Span(\mathcal{S}(\PP)_2) }(\underline{1}, \underline{m}) \to \amalg_n \PP(n) \underset{\mathfrak{S}_n}{\times} \underline{m}^n$ given by:
$$\phi([\xymatrix{ \underline{1} &\underline{n}  \ar@{-->}[l]_ -{(s,\omega)} \ar[r]^-f &\underline{m}  }])=(\omega ,(f(1), \ldots, f(n))).$$
This map is well-defined and is an injection by $(1)$ and obviously is a surjection.
\end{enumerate}

\end{proof}

\begin{proof}[Proof of Proposition \ref{Free(P)-span}]
We define a functor $F:Span(\mathcal{S}(\PP)_2) \to Free(\PP)$. On objects we take
$F(\underline{n})=\mathcal{F}_{\PP}(\underline{n})$.
Proposition \ref{coprod} gives the existence of coproducts in $Span(\mathcal{S}(\PP)_2) $ so $Hom_{Span(\mathcal{S}(\PP)_2) }(\underline{k}, \underline{m})= \amalg  Hom_{Span(\mathcal{S}(\PP)_2) }(\underline{1}, \underline{m})$. Consequently, it is sufficient to define $F$ on morphisms of the type $[\xymatrix{ \underline{1} &  \underline{n}\ar@{-->}[l]_ -{(s,\omega)} \ar[r]^-f & \underline{m} }]$. Let $D=(\xymatrix{ \underline{1} & \underline{n} \ar@{-->}[l]_ -{(s,\omega)} \ar[r]^-f & \underline{m}})$ be a diagram. Since $\mathcal{F}_{\PP}$ is the left adjoint to the forgetful functor $U: \PP \ti Alg \to  Set$, to define a morphism in $\PP \ti Alg$: $\mathcal{F}_{\PP}(\underline{1}) \to \mathcal{F}_{\PP}(\underline{m})$ is equivalent to define a set morphism $\underline{1} \to U(\mathcal{F}_{\PP}(\underline{m}))=\amalg_n \PP(n) \underset{\mathfrak{S}_n}{\times} \underline{m}^n$ i.e. an element of $U(\mathcal{F}_{\PP}(\underline{m}))=\amalg_n \PP(n) \underset{\mathfrak{S}_n}{\
 times} \underline{m}^n$. So  let
$$F(\xymatrix{ \underline{1} &\underline{n} \ar@{-->}[l]_ -{(s,\omega)} \ar[r]^-f &\underline{m} })=(\omega,(f(1), \ldots, f(n))) \in \PP(n) \underset{\mathfrak{S}_n}{\times} \underline{m}^n.$$
Note that $F(\xymatrix{ \underline{1} &\underline{0} \ar@{-->}[l]_ -{(s,\omega)} \ar[r]^-f &\underline{m} })=(0_\PP, *).$
Let $D'=(\xymatrix{ \underline{1} & \underline{n} \ar@{-->}[l]_ -{(s,\omega')} \ar[r]^-{f'} & \underline{m}})$ be another diagram such that $D$ and $D'$ are equivalent in $Span(\mathcal{S}(\PP)_2)$. By Lemma \ref{lm-equi}
$$(\omega,(f(1), \ldots, f(n)))=(\omega', (f'(1), \ldots, f'(n)))  \in \PP(n) \underset{\mathfrak{S}_n}{\times} \underline{m}^n.$$
So we can define: 
$$F([\xymatrix{ \underline{1} & \underline{n} \ar@{-->}[l]_ -{(s,\omega)} \ar[r]^-{f} & \underline{m}}])=(\omega,(f(1), \ldots, f(n))) \in \PP(n) \underset{\mathfrak{S}_n}{\times} \underline{m}^n.$$
One verifies that $F$ preserves composition.
Furthermore $F$ is essentially surjective and fully faithful since:
$$Hom_{Free(\PP)}(\mathcal{F}_{\PP}(\underline{p}), \mathcal{F}_{\PP}(\underline{m})\simeq Hom_{Set}(\underline{p}, U(\mathcal{F}_{\PP}(\underline{m})))\simeq \amalg_p\  Hom_{Set}(\underline{1}, U(\mathcal{F}_{\PP}(\underline{m})))$$
$$\simeq  \amalg_p\  (\amalg_k \PP(k) \underset{\mathfrak{S}_n}{\times} \underline{m}^k) \simeq \amalg_p\  Hom_{Span(\mathcal{S}(\PP)_2) }(\underline{1}, \underline{m}) \simeq Hom_{Span(\mathcal{S}(\PP)_2) }(\underline{p}, \underline{m}).$$
\end{proof}

\begin{defi} \label{M-functor}
A Janus functor $M: \Gamma(\PP)_2 \to Ab$ is a $\M \ti$functor if:
\begin{enumerate}
\item for a standard inclusion $i: [n] \to [n] \vee [m]$ and a standard retraction $r: [n] \vee [m] \to [n]$ one has 
$$M_*(r)=M^*(i, \omega^i): M([n] \vee [m] ) \to M([n])$$
$$M^*(r, \omega^r)=M_*(i): M([n]) \to M([n] \vee [m] )$$
where $\omega^i=(1_{\PP}^{\times n}, 0_{\PP}^{\times m})$ and $\omega^r=1_{\PP}^{\times n}$;
\item for any square in $\Gamma(\PP)_2$
$$ \mathcal{D}=\quad \vcenter{
\xymatrix{ [k] \ar@{-->}[d]_ -{(\phi, \omega^{\phi})} \ar[r]^-{f_1} &[l] \ar@{-->}[d]^-{(\psi, \omega^{\psi})}\\
[m] \ar[r]_-f &[n]
} }$$
with $f^{-1}(0)=0$ we have:
$$M^*(\psi, \omega^{\psi}) M_*(f)=M_*(f_1) M^*(\phi, \omega^{\phi}).$$
\end{enumerate}
The category of $\M \ti$functors is denoted by $\M \ti func(\Gamma(\PP), Ab)$.
\end{defi}

\begin{rem} \label{rem-iso}
\begin{enumerate}
\item For $M$ a $\mathcal{M}$-functor, by condition $(2)$ we have, in particular, that for an isomorphism $f: [n] \to [n]$ of $\Gamma$: $M_*(f)=M^*(f^{-1}, 1_{\PP}^{\times n})$.
\item Condition $(1)$ of the proposition also holds for all injections $i': [n] \to [n] \vee [m]$ in $\Gamma(\PP)$ and their dual retraction $r': [n] \vee [m] \to [n]$. To see this choose a permutation $\sigma$ of $[n] \vee [m]$ s.t. $i'=\sigma i$ and $r'=r \sigma$ and use point (1).
\end{enumerate}
\end{rem}

\begin{prop} \label{eq-M-functor} 
There is an equivalence of categories:
$$Mack(\mathcal{S}(\PP)_2, Ab) \simeq \mathcal{M} \ti func(\Gamma(\PP), Ab).$$
\end{prop}
In order to prove this proposition, we need the following construction and result. For a horizontal morphism $h: X \to Y$ in $\Gamma(\PP)_2$ we consider the restriction of $h$ on $X \setminus h^{-1}(0)$: $h_-: X \setminus h^{-1}(0) \to Y \setminus \{0\}$
and the inclusion: $i_h: X \setminus h^{-1}(0) \to X \setminus \{ 0 \}$. Similarly, for a vertical map
$(\phi, \omega^{\phi}): X \to Y$ in $\Gamma(\PP)_2$ where 
$$ \omega^{\phi} \in \prod_{y \in Y \setminus \{0\}}  \PP(\mid \phi^{-1}(y) \mid)$$
we define the restriction of $(\phi, \omega^{\phi})$ on $X \setminus \phi^{-1}(0)$: $(\phi_-, \omega^{\phi}): X \setminus \phi^{-1}(0) \to Y \setminus \{ 0\}$ and the inclusion: $(i_\phi, \omega^{i_\phi}): X \setminus \phi^{-1}(0) \to X \setminus \{0 \}$ where $\omega^{i_\phi}=(1_{\PP},  \ldots, 1_{\PP}, 0_{\PP},  \ldots, 0_{\PP}) \in \prod_{x \in X \setminus \{0 \}} \PP(\mid i_{\phi}^{-1}(x)\mid)$.

\begin{lm} \label{square-in-MT}
If
$$ \mathcal{D}=\quad \vcenter{
\xymatrix{ [k] \ar@{-->}[d]_ -{(\phi, \omega^{\phi})} \ar[r]^-{g} &[l] \ar@{-->}[d]^-{(\psi, \omega^{\psi})}\\
[m] \ar[r]_-f &[n]
} }$$
is a square in $\Gamma(\PP)_2$ such that $f^{-1}(0)=0$ then 
$$ \mathcal{D}'= \quad  \vcenter{
\xymatrix{ [k]\setminus \phi^{-1}(0) \ar@{-->}[d]_ -{(\phi_-, \omega^{\phi})} \ar[rr]^-{g_{\mid [k] \setminus  \phi^{-1}(0)}} &&[l] \setminus \psi^{-1}(0) \ar@{-->}[d]^-{(\psi_-, \omega^{\psi})}\\
\underline{m} \ar[rr]_-{f_-} && \underline{n}
} }$$
and 
$$ \mathcal{D}''= \quad  \vcenter{
\xymatrix{ [k]\setminus \phi^{-1}(0) \ar@{-->}[d]_ -{(i_\phi, \omega^{i_\phi})} \ar[rr]^-{g_{\mid [k] \setminus  \phi^{-1}(0)}} &&[l] \setminus \psi^{-1}(0) \ar@{-->}[d]^-{(i_\psi, \omega^{i_\psi})}\\
\underline{k} \ar[rr]_-{g_-} && \underline{l}
} }$$
are squares in $\mathcal{S}(\PP)_2$.
\end{lm}
\begin{proof}
Since the image of $\mathcal{D}$ in $\Gamma$ is a pullback diagram we have bijections:
$$g^{-1}(0) \simeq f^{-1}(0)=\{0 \} \quad \mathrm{and} \quad \psi^{-1}(0) \simeq \phi^{-1}(0).$$
We compute the pullback of 
$\vcenter{
\xymatrix{  & &[l] \setminus \psi^{-1}(0) \ar@{-->}[d]^-{(\psi_-, \omega^{\psi})}\\
\underline{m} \ar[rr]_-{f_-} && \underline{n}
} }$ in $\mathcal{S}$.
\begin{eqnarray*}
\underline{m} \times_{\underline{n}} ([l] \setminus \psi^{-1}(0))&=& \{(c,e) \in \underline{m}  \times  ([l] \setminus \psi^{-1}(0)) \mid f_-(c)=\psi_-(e) \}\\
&=&\{(c,e) \in [m] \times  [l] \mid f(c)=\psi(e) \} \setminus \{(c,e) \in [m] \times  [l] \mid f(c)=\psi(e)=0 \} \\
&=& [m] \times_{[n]} [l] \setminus (f^{-1}(0) \times \psi^{-1}(0))\\
&\simeq& [k]\setminus \phi^{-1}(0).
\end{eqnarray*}
So, the image of $\mathcal{D}'$ in $\mathcal{S}$ is a pullback diagram. Furthermore, for $y \in \underline{m}$ we have $(\phi_-)^{-1}(y)=\phi^{-1}(y)$ and $(\psi_-)^{-1}(y)=\psi^{-1}(y)$ so the second condition to be a square in $\Gamma(\PP)_2$ is satisfied by $\mathcal{D}'$ since it is satisfied by $\mathcal{D}$.

We compute the pullback of 
$ \vcenter{
\xymatrix{ &&[l] \setminus \psi^{-1}(0) \ar@{-->}[d]^-{(i_\psi, \omega^{i_\psi})}\\
\underline{k} \ar[rr]_-{g_-} && \underline{l}
} }$ in $\mathcal{S}$.
\begin{eqnarray*}
\underline{k} \times_{\underline{l}} ([l] \setminus \psi^{-1}(0))&=& \{(c,e) \in \underline{k}  \times  ([l] \setminus \psi^{-1}(0)) \mid g_-(c)=e \}\\
&\simeq&\{c \in  \underline{k}  \mid  g_-(c) \in  [l] \setminus \psi^{-1}(0) \} \\
&=&\{c \in  \underline{k}  \mid  c \in  [k] \setminus \phi^{-1}(0)  \} \mathrm{\ since\ }   g_-(c) \in  \psi^{-1}(0) \mathrm{\ iff\ } c \in  \phi^{-1}(0)\\
&=& [k] \setminus \phi^{-1}(0).
\end{eqnarray*}
So, the image of $\mathcal{D}''$ in $\mathcal{S}$ is a pullback diagram. Furthermore, for $y \in \underline{k}$ since $\mid (i_{\phi})^{-1}(y) \mid =1= \mid (i_{\psi})^{-1}(g(y)) \mid$, we have a bijection $ (i_{\phi})^{-1}(y) \to (i_{\psi})^{-1}(g(y))$ such that the induced bijection $\PP(1) \to \PP(1)$ is the identity and satisfies $\omega^{i_\phi}_y=1_{\PP} \mapsto \omega^{i_\psi}_{g(y)}=1_{\PP}$. 
\end{proof}

\begin{proof}[Proof of Proposition \ref{eq-M-functor}]
Let $M: \Gamma(\PP)_2 \to Ab$ be a $\mathcal{M}$-functor. We define a Janus functor $F(M)=(F(M)_*, F(M)^*): \mathcal{S}(\PP)_2 \xrightarrow{j} \Gamma(\PP)_2  \xrightarrow{M=(M_*, M^*)} Ab$ by precomposition of $M$ with the functor of double categories $j$. Let 
$$ \mathcal{D}=\quad \vcenter{
\xymatrix{ \underline{k} \ar@{-->}[d]_ -{(\phi, \omega^{\phi})} \ar[r]^-{f_1} &\underline{l} \ar@{-->}[d]^-{(\psi, \omega^{\psi})}\\
\underline{m} \ar[r]_-f &\underline{n}
} }$$ be a square in $\mathcal{S}(\PP)_2$, 
$$ j(\mathcal{D})=\quad \vcenter{
\xymatrix{ [k] \ar@{-->}[d]_ -{(\phi_+, \omega^{\phi})} \ar[r]^-{{f_1}_+} &[l] \ar@{-->}[d]^-{(\psi_+, \omega^{\psi})}\\
[m] \ar[r]_-{f_+} &[n]
} }$$ is a square in $\Gamma(\PP)_2 $ such that $f_+^{-1}(0)=0$. Since $M$ is a $\mathcal{M}$-functor we have:
$$M^*(\psi_+, \omega^{\psi}) \circ M_*(f_+)=M_*({f_{1}}_+) \circ M^*(\phi_+, \omega^{\phi})$$
i.e.
$$F(M)^*(\psi, \omega^{\psi}) \circ F(M)_*(f)=F(M)_*({f_{1}}) \circ F(M)^*(\phi, \omega^{\phi})$$
so, $F(M)$ is a Mackey functor and we defined a functor $F:  \mathcal{M} \ti func(\Gamma(\PP)_2, Ab) \to Mack(\mathcal{S}(\PP)_2, Ab).$

Let $N: \mathcal{S}(\PP)_2 \to Ab$ be a Mackey functor. We define $G(N)=(G(N)_*, G(N)^*): \Gamma(\PP)_2 \to Ab$ by:
$$G(N)_*([k])=G(N)^*([k])=N(\underline{k})$$
and, for $h: X \to Y \in \Gamma(\PP)_2^h$
$$G(N)_*(h)=N_*(h_-) \circ N^*(i_h, \omega^{i_h}) $$
and for $(\phi, \omega^{\phi}): X \to Y  \in \Gamma(\PP)_2^v$
$$G(N)^*(\phi, \omega^{\phi})=N_*(i_\phi) \circ N^*(\phi_-, \omega^{\phi}).$$

Let $i: [n] \to [n] \vee [m]$ be a standard inclusion and $r: [n] \vee [m] \to [n]$ be a standard retraction. We have:
$$r_-=i_i=Id_{[n] \setminus \{ 0 \}}: [n]  \setminus \{ 0 \} \to [n]  \setminus \{ 0 \}$$
$$i_r=i_-: ([n] \vee [m]) \setminus (\{0 \} \vee [m])=[n]  \setminus \{ 0 \} \to ( [n] \vee [m]) \setminus\{ 0 \}.$$
So, we have
$$G(N)_*(r)=G(N)^*(i, \omega^i) \mathrm{\ and\ } G(N)^*(r, \omega^r)=G(N)_*(i)$$
since
$$G(N)_*(r)=N_*(r_-) \circ N^*(i_r, \omega^{i_r})=N^*(i_r, \omega^{i_r}) \mathrm{\ and\ }G(N)^*(i, \omega^i)=N_*(i_i) \circ N^*(i_-, \omega^i)=N^*(i_-, \omega^i)$$
and 
$$G(N)^*(r, \omega^r)=N_*(i_r) \circ N^*(r_-, \omega^r)=N_*(i_r) \mathrm{\ and\ }G(N)_*(i)=N_*(i_-) \circ N^*(i_i, \omega^i)=N_*(i_-).$$
Let $$ \mathcal{D}=\quad  \vcenter{
\xymatrix{ [k] \ar@{-->}[d]_ -{(\phi, \omega^{\phi})} \ar[r]^-{g} &[l] \ar@{-->}[d]^-{(\psi, \omega^{\psi})}\\
[m] \ar[r]_-f &[n]
} }$$
be a square in $\Gamma(\PP)_2$ such that $f^{-1}(0)=\{0\}$, according to Lemma \ref{square-in-MT} the squares $\mathcal{D}'$ and $\mathcal{D}''$ in this lemma are squares of $\mathcal{S}(\PP)_2$. Since $N$ is a Mackey functor we deduce that:
$$ N^*(\psi_-, \omega^{\psi})  \circ N_*(f_-)=N_*(g_{\mid [k] \setminus  \phi^{-1}(0)}) \circ N^*(\phi_-, \omega^{\phi})$$
and 
$$N^*(i_\psi, \omega^{i_\psi})  \circ N_*(g_-)=N_*(g_{\mid [k] \setminus  \phi^{-1}(0)}) \circ N^*(i_\phi, \omega^{i_\phi}).$$
Furthermore, since $f^{-1}(0)=g^{-1}(0)=\{0\}$ we have $i_f=Id_{\underline{m}}$ and $i_g=Id_{\underline{k}}$.
We compute on the one hand:
\begin{eqnarray*}
G(N)^*(\psi, \omega^\psi) \circ G(N)_*(f)&=& N_*(i_\psi) \circ N^*(\psi_-, \omega^{\psi}) \circ N_*(f_-) \circ N^*(i_f, \omega^{i_f})\\
&=&N_*(i_\psi) \circ N_*(g_{\mid [k] \setminus  \phi^{-1}(0)}) \circ N^*(\phi_-, \omega^{\phi})\circ N^*(i_f, \omega^{i_f})\\
&=& N_*(i_\psi) \circ N_*(g_{\mid [k] \setminus  \phi^{-1}(0)}) \circ N^*(\phi_-, \omega^{\phi})\\
&=& N_*(i_\psi \circ g_{\mid [k] \setminus  \phi^{-1}(0)}) \circ N^*(\phi_-, \omega^{\phi})
\end{eqnarray*}
and on the other hand:
\begin{eqnarray*}
G(N)_*(g)  \circ G(N)^*(\phi, \omega^\phi)&=&N_*(g_-) \circ N^*(i_g, \omega^{i_g}) \circ N_*(i_\phi) \circ N^*(\phi_-, \omega^{\phi})\\
&=&N_*(g_-) \circ N_*(i_\phi) \circ N^*(\phi_-, \omega^{\phi})\\
&=&N_*(g_- \circ i_\phi) \circ N^*(\phi_-, \omega^{\phi}).
\end{eqnarray*}
Since $i_\psi \circ g_{\mid [k] \setminus  \phi^{-1}(0)}=g_- \circ i_\phi$ we conclude that $G(N)$ is a $\mathcal{M}$-functor.

We verify easily that the functors $F$ and $G$ define an equivalence of categories.
\end{proof}

\begin{prop} \label{eq-Omega} 
There is an equivalence of categories:
$$\mathcal{M} \ti func(\Gamma(\PP), Ab) \simeq PMack(\Omega(\PP), Ab).$$
\end{prop}
\begin{proof}

\begin{enumerate}
\item Let $M=(M_*, M^*)$ be a $\mathcal{M}$-functor, we prove that $(cr(M_*), \tilde{cr}(M^*))$ is a pseudo-Mackey functor. 

First, we prove that $(cr(M_*), \tilde{cr}(M^*))$ is a Janus functor. Let $ \underline{n} \in Ob(\Omega(\PP)_2)$ we have:
\begin{eqnarray*}
cr(M_*)( \underline{n})&=&cr_{n}M_*([1], \ldots, [1])\\
&=&Ker(\prod_{k=1}^{n}M_*(r^{n}_{\hat{k}}))\\
&=&Ker(\prod_{k=1}^{n}M^*(i^{n}_{\hat{k}}, \omega^{i^{n}_{\hat{k}}})) \mathrm{\quad by\ Remark\ }\ref{rem-iso}\ (2)\\
&=&\tilde{cr}_{n}M^*([1], \ldots, [1])\\
&=&\tilde{cr}(M^*)(\underline{n}).
\end{eqnarray*}

Let $(f, (f, 1_\PP^{\times \mid B \mid}))$ be a double isomorphism from $A$ to $B$ in   $\Omega(\PP)_2$. We consider the double isomorphism $(j(f), (j(f), 1_\PP^{\times \mid B \mid}))$ from $j(A)$ to $j(B)$ of $\Gamma(\PP)_2$, where $j$ is the functor which adds a basepoint. By Remark \ref{rem-iso} (1) we have:
$$M_*(j(f))=M^*(j(f^{-1}), 1_\PP^{\times \mid B \mid}) \mathrm{\quad and \quad} M^*(j(f), 1_\PP^{\times \mid B \mid})=M_*(j(f^{-1})).$$
In the sequel, for simplicity, for $f: A \to B$ a bijection, we denote by $f$ also the vertical morphism $(f, 1_\PP^{\times \mid B \mid})$.
Since $f$ is an isomorphism, by definition of $i_*$ and $i_!$ in Proposition  \ref{DK-explicit} we have the following commutative diagram:
$$\xymatrix{ 
cr(M)(B) \ar[d]^{\tilde{cr}(M^*)(f)} \ar@{^{(}->}[r] & \bigoplus_{X \subset B} cr(M)(X)=i_*(\tilde{cr}M^*)(j(B)) \ar[d]^{i_*(\tilde{cr}M^*)(j(f))} \ar[r]^-{\simeq} & M(j(B)) \ar[d]^{M^*(j(f))}\\
    cr(M)(f(B))=cr(M)(A)  \ar[d]^{cr(M_*)(f)}\ar@{^{(}->}[r]&   \bigoplus_{Y \subset A} cr(M)(Y)=i_*(\tilde{cr}M^*)(j(A))   \ar[d]^{i_!(crM_*)(j(f))} \ar[r]^-{\simeq}  &M(j(A))  \ar[d]^{M_*(j(f))} \\
    cr(M)(B)  \ar@{^{(}->}[r]&  \bigoplus_{Z \subset B} cr(M)(Z)=i_!(crM)(j(B))  \ar[r] \ar[r]^-{\simeq}&  M(j(B))\\
} .$$ 
So $cr(M_*)(f) \tilde{cr}(M^*)(f)$ is the restriction of $M_*(j(f)) M^*(j(f))$ on $cr(M)(B)$, but
\begin{eqnarray*}
M_*(j(f)) M^*(j(f))
&=&M_*(j(f)) M_*(j(f^{-1}))\\
&=& M_*(j(f) j(f^{-1}))\\
&=&Id
\end{eqnarray*}
whence $cr(M_*)(f) \tilde{cr}(M^*)(f)=Id$. 

Let $$\mathcal{D}=\quad \vcenter{
\xymatrix{ A \ar@{-->}[d]_ -{(\phi, \omega^{\phi})} \ar[r]^-f &B \ar@{-->}[d]^-{(\psi, \omega^{\psi})}\\
C \ar[r]_-g &D
} }$$  be a square in $\Omega(\PP)_2$. Applying to $\mathcal{D}$ the functor $j: \Omega(\PP)_2 \to \Gamma(\PP)_2$ we obtain a square in $\Gamma(\PP)_2$ such that $j(g)^{-1}(0)=0$. Since $M$ is a $\mathcal{M}$-functor we have the following commutative diagram:

$$\xymatrix{ 
 & M(j(A)) \ar[r]^ -{M_*(j(f))}  & M(j(B))=\bigoplus_{Z \subset B} crM(Z) \ar@{->>}[r]^-{p} & crM(B)\\
 crM(C) \ar@{^{(}->}[r]_-{i} & M(j(C))\simeq \bigoplus_{X \subset C} crM(X)   \ar[r]_-{M_*(j(g))}  \ar[u]^-{M^*(j(\phi, \omega^{\phi}))}&  M(j(D)). \ar[u]_-{M^*(j(\psi, \omega^{\psi}))}
} $$ 
On the one hand we compute the composition $p \circ M^*(j(\psi, \omega^{\psi})) \circ M_*(j(g)) \circ i$. Since $g$ is surjective, by definition of $i_*$ and $i_!$ in Theorem \ref{DK} we have the following commutative diagram:
$$\xymatrix{ 
cr(M)(C) \ar[d]^{cr(M_*)(g)} \ar@{^{(}->}[r]^-{i} & \bigoplus_{X \subset C} cr(M)(X)=i_!(crM_*)(j(C)) \ar[d]^{i_!(crM_*)(j(g))} \ar[r]^-{\simeq} & M(j(C)) \ar[d]^{M_*(j(g))}\\
    cr(M)(g(C))=cr(M)(D)  \ar[d]^{\tilde{cr}(M^*)(\psi, \omega^{\psi})}\ar@{^{(}->}[r]&   \bigoplus_{Y \subset D} cr(M)(Y)=i_!(crM_*)(j(D))   \ar[d]^{i_*(\tilde{cr}M^*)(j(\psi, \omega^{\psi}))} \ar[r]^-{\simeq}  &M(j(D))  \ar[d]^{M^*(j(\psi, \omega^{\psi}))} \\
    cr(M)(B) &  \bigoplus_{Z \subset B} cr(M)(Z)=i_*(crM)(j(B))  \ar@{->>}[l]_-p \ar[r]^-{\simeq}&  M(j(B))\\
} .$$ 
So
\begin{equation} \label{eq1}
p \circ M^*(j(\psi, \omega^{\psi})) \circ M_*(j(g)) \circ i=\tilde{cr}(M^*)(\psi, \omega^{\psi}) \circ cr(M_*)(g).
\end{equation}

On the other hand we compute the composition $p \circ M_*(j(f)) \circ M^*(j(\phi, \omega^{\phi})) \circ i$.
By definition of $i_*$ the restriction of $M^*(j(\phi, \omega^{\phi}))$ to the factor $cr(M)(C)$  is $0$ to the factors $cr(M)(L)$ such that $L \subset A$ and $\phi(L) \neq C$  and is the morphism
$$cr(M^*)(\phi, \omega^{\phi})_L: cr(M)(C) \to cr(M)(L)$$ 
to the factors $cr(M)(L)$ where $L \subset A$ and $\phi(L) =C$. So, we have the following commutative diagram:
$$\xymatrix{ 
 & cr(M)(C) \ar[d] \ar@{^{(}->}[r]^-{i} & i_*(\tilde{cr}M^*)(j(C)) \ar[d]^{i_*(\tilde{cr}M^*)(j((\phi, \omega^{\phi}))) \simeq M^*(j((\phi, \omega^{\phi})))} \simeq M(j(C)) \\
 \underset{\tiny{ \begin{array}{c} L \subset A \\ \phi(L)=C; f(L)=B \end{array}}}{\bigoplus}  crM(L)         \ar[d]_-{k} &     
 \underset{ \tiny{ \begin{array}{c} L \subset A \\ \phi(L)=C \end{array}}}{\bigoplus} crM(L)  \ar[d] \ar@{^{(}->}[r]  \ar@{->>}[l]&   i_*(\tilde{cr}M^*)(j(D))   \ar[d]^{i_!(crM_*)(j(f)) \simeq M_*(j(f))} \simeq M(j(A))   \\
 cr(M)(B)  &   \underset{ \tiny{ \begin{array}{c} L \subset A \\ \phi(L)=C \end{array}}}{\bigoplus} crM(f(L)) \ar@{->>}[l] &  i_!(crM_*)(j(B))  \ar@{->>}[l] \simeq M(j(B))\\
} $$ 
where we have $k=\underset{\tiny{ \begin{array}{c} L \subset A \\ \phi(L)=C \\ f(L)=B \end{array}}}{\bigoplus} cr(M_*)(f_{\mid L}) $. We deduce that:
\begin{eqnarray*}
p \circ M_*(j(f)) \circ M^*(j(\phi, \omega^{\phi})) \circ i&=& \underset{\tiny{ \begin{array}{c} L \subset A \\ \phi(L)=C \\ f(L)=B \end{array}}}{\sum} cr(M_*)(f_{|L}) \circ cr(M^*)(\phi, \omega^{\phi})_L\\
&= &\sum_{L \in Adm(A)} cr(M_*)(f_{|L}) \circ cr(M^*)(\phi, \omega^{\phi})_L.
\end{eqnarray*}
Combining this last equality and (\ref{eq1}) we obtain:
$$ cr(M^*)(\psi, \omega^{\psi}) \circ cr(M_*)(g)=\sum_{L \in Adm(A)} cr(M_*)(f_{|L}) \circ cr(M^*)(\phi, \omega^{\phi})_L.$$

\item Let $M=(M_*,M^*) \in PMack (\Omega(\PP), Ab)$, we prove that $(i_!M_*, i_*M^*)$ is a $\mathcal{M}$-functor. 

First, we check that $(i_!(M_*), i_*(M^*))$ is a Janus functor. For $[n] \in \Gamma(\PP)_2$, 
$$i_!(M_*)([n])=\bigoplus_{\mu \subset \underline{n}}M_*(\mu)=\bigoplus_{\mu \subset \underline{n}}M^*(\mu)=i_*(M_*)([n]).$$

Let $i: [n] \to [n] \vee [m]$ be a standard inclusion and $r:  [n] \vee [m] \to [n]$ be the associated standard retraction. On the one hand, we compute $i_*(M^*)(r, \omega^r)$.

$$\xymatrix{ 
M^*(\underline{\mid \mu \mid}) \ar[d]_{M^*((r, \omega^r)_{ i(\mu)})} \ar@{^{(}->}[r] & \bigoplus_{\mu \subset \underline{n}} M^*(\underline{\mid \mu \mid})= i_*(M^*)([n]) \ar[d]^{i_*(M^*)(r, \omega^r)}\\
    M^*(\underline{\mid i(\mu) \mid}) \ar@{^{(}->}[r]&  \bigoplus_{\nu \subset \underline{n+m}} M^*(\underline{\mid \nu \mid}) =i_*(M^*)( [n] \vee [m] ) } $$ 

since, the only subset $\nu \subset \underline{n+m}$ such that $r(\nu)=\mu$ is $i(\mu)$.

On the other hand, we compute $i_!(M_*)(i)$.
$$\xymatrix{ 
M_*(\underline{\mid \mu \mid}) \ar[d]_{M_*(i_{\mid \mu})} \ar@{^{(}->}[r] & \bigoplus_{\mu \subset \underline{n}} M_*(\underline{\mid \mu \mid})= i_!(M_*)([n]) \ar[d]^{i_!(M_*)(i)}\\
    M_*(\underline{\mid i(\mu) \mid}) \ar@{^{(}->}[r]&  \bigoplus_{\nu \subset \underline{n+m}} M_*(\underline{\mid \nu \mid}) =i_!(M_*)( [n] \vee [m] ) }. $$ 

Since $i_{\mid \mu}: \mu \to i(\mu)$ and $r_{\mid i(\mu)}: i(\mu) \to \mu$ are isomorphisms inverse to each other and $M$ is in $PMack (\Omega(\PP), Ab)$ we have: $M^*((r, \omega^r)_{ i(\mu)})=M_*(i_{\mid \mu})$. So 
$$i_*(M^*)(r, \omega^r)=i_!(M_*)(i).$$
We prove, in a similar way, that: $i_*(M^*)(i)=i_!(M_*)(r, \omega^r).$

Let $$ \mathcal{D}=\quad \vcenter{
\xymatrix{ [k] \ar@{-->}[d]_ -{(\phi, \omega^{\phi})} \ar[r]^-{f_1} &[l] \ar@{-->}[d]^-{(\psi, \omega^{\psi})}\\
[m] \ar[r]_-f &[n]
} }$$
be a square in $\Gamma(\PP)_2$
such that $f^{-1}(0)=0$. 
On the one hand, we compute $i^*(M^*)(\psi, \omega^{\psi})i_!(M_*)(f)$:
$$\xymatrix{ 
M_*(\underline{\mid \mu \mid}) \ar[d]_{M_*(f_{\mid \mu})} \ar@{^{(}->}[r] & \bigoplus_{\mu \subset \underline{m}} M_*(\underline{\mid \mu \mid})= i_!(M_*)([m]) \ar[d]^{i_!(M_*)(f)}\\
   M_*(\underline{\mid f(\mu) \mid}) \ar[d]_{u} \ar@{^{(}->}[r]& i_!(M_*)( [n] )  \ar[d]^{i_*(M^*)(\psi, \omega^\psi) }\\
     \bigoplus_{K \subset \underline{l}, \psi(K)=f(\mu))} M(K)\ar@{^{(}->}[r]& i_!(M_*)( [l] )   }  $$ 
(since $f^{-1}(0)=0$, $0 \not\in f(\mu)$) where $u$ is the sum of the  maps
$$M^*((\psi, \omega^\psi)_K): M_*(\underline{\mid f(\mu) \mid}) \to M_*(K)$$ to the factors $M_*(K)$ such that $K \subset \underline{l}$ and $\psi(K)=f(\mu)$. 
For a fixed $K \subset \underline{l}$ such that $\psi(K)=f(\mu)$ we have the following square in $\Omega(\PP)_2$:
$$\vcenter{
\xymatrix{
\phi^{-1}(\mu) \cap f_1^{-1}(K) \ar@{-->}[d]_ -{(\tilde{\phi}, \omega^{\tilde{\phi}})} \ar[r]^-{\tilde{f_1}} &K \ar@{-->}[d]^-{(\psi, \omega^{\psi})_K}\\
\mu \ar[r]_-{f_{\mid \mu}} &f(\mu)=\psi(K).
} }$$
Since $(M_*, M^*)$ is a pseudo-Mackey functor we have:
$$M^*((\psi, \omega^{\psi})_K) M_*(f_{\mid \mu})=\sum_{K' \in Adm(\phi^{-1}(\mu) \cap f_1^{-1}(K))}M_*(\tilde{f_1}_{\mid K'})M^*((\tilde{\phi}, \omega^{\tilde{\phi}})_{K'}).$$

On the other hand, we compute $i_!(M_*)(f_1)i^*(M^*)(\phi, \omega^{\phi})$. We have:
$$\xymatrix{ 
M_*(\underline{\mid \mu \mid}) \ar[d]_{v } \ar@{^{(}->}[r] & \bigoplus_{\mu \subset \underline{n}} M_*(\underline{\mid \mu \mid})= i^*(M_*)([m]) \ar[d]^{i^*(M^*)(\phi, \omega^{\phi})}\\
 \bigoplus_{L \subset  [k], \phi(L)=\mu} M(L) \ar[d]_{  \bigoplus_{L \subset  [k], \phi(L)=\mu}M_*({f_1}_{\mid L} )       } \ar@{^{(}->}[r]& i_!(M_*)([k])  \ar[d]^{i_!(M_*)(f_1) }\\
     \bigoplus_{L \subset  [k], \phi(L)=\mu}  M(f_1(L))\ar@{^{(}->}[r]& i_!(M_*)( [l] )  }  $$ 
     where $v$ is the sum of the morphisms $M^*({\phi}_{\mid L}): M_*(\underline{\mid \mu \mid}) \to M(L)$ to the factors $M(L)$ such that $L \subset  [k]$ and $\phi(L)=\mu$.
For a fixed $K \subset \underline{l}$ such that $\psi(K)=f(\mu)$, the value of the composition $\bigoplus_{L \subset  [k], \phi(L)=\mu}M_*({f_1}_{\mid L} )  v$ to $M(K)$ is equal to:
$$\sum_{L \subset  [k], \phi(L)=\mu, f_1(L)=K}M_*({f_1}_{\mid L} )  M^*(({\phi},\omega^{\phi})_L)=\sum_{L \in Adm(\phi^{-1}(\mu) \cap f_1^{-1}(K))} M_*({f_1}_{\mid L} )  M^*(({\phi},\omega^{\phi})_L).$$

We deduce that:
$$i^*(M^*)(\psi, \omega^{\psi})i_!(M_*(f))=i_!(M_*(f_1))i^*(M^*)(\phi, \omega^{\phi}).$$

\end{enumerate}
\end{proof}
As a corollary of Theorem \ref{thm-pal} we obtain the following result:
\begin{cor} \label{equ-pol-cor}
There is a natural equivalence of categories between the category of reduced polynomial functors of degree $\leq n$ from $Free(\PP)$ to $Ab$ and the category of pseudo-Mackey functors from $\Omega(\PP)$ to $Ab$ which have zero values on $\underline{0}$ and on sets of cardinality $> n$.
\end{cor}

\begin{rem}
Applying Theorem \ref{thm-pal} for $\PP=\mathcal{I}$ (the initial unitary set-operad, see Example \ref{ex-1.1} (2)), we recover the Dold-Kan type theorem of \cite{Pira-Dold}: $Func(\Gamma, Ab) \simeq Func (\Omega, Ab)$. In fact, the vertical maps of $\Omega(\PP)_2$ are just the bijections, so by Remark \ref{rem-iso} (1) the functor $M^*$ is determined by $M_*$.
\end{rem}

\section{Polynomial maps and polynomial functors} \label{section5}
The principal aim of this section is to prove, in Corollary \ref{pol-gr-mon}, that polynomial functors of degree $n$ from finitely generated free groups to abelian groups coincide with polynomial functors of degree $n$ from finitely generated free monoids to abelian groups. In order to prove this crucial result for this paper, we introduce intermediate material interesting on its own. In fact, after recalling the definition of polynomial maps from monoids in the sense of Passi, we extend this definition to algebraic theories and, more generally to sets of morphisms of a suitable general category. This leads us to introduce a category $T_n(\overline{\Z}[\C])$ generalizing the categories $P_n(\C)$ for $\C$ an additive category considered in \cite{Pira}. These categories have the important property that polynomial functors of degree $n$ from $\C$ to $Ab$ are equivalent to additive functors from $T_n(\overline{\Z}[\C])$ to $Ab$ (see Corollary \ref{Poln(C)=Add(PnC)}). 

The proof of Corollary \ref{pol-gr-mon} relies on three main ingredients: the isomorphism between polynomial maps from monoids and from their ``groupification" in Proposition \ref{poly-map}, the link between polynomial functors and polynomial maps leading to Corollary \ref{Poln(C)=Add(PnC)} and the canonical isomorphism $T_n \overline{\Z}[gr] \simeq T_n \overline{\Z}[mon]$  obtained in Theorem \ref{Pnmon=Pngr} from the computation of these categories given in Corollary  \ref{Pn(Free(P))}.

\subsection{Passi polynomial maps}\label{sec:Pnclass}
In this section, we begin by recalling the classical definition of polynomial maps from monoids to abelian groups due to Passi and we prove that polynomial maps on a monoid and on its groupification coincide (see Proposition \ref{poly-map}). Then we extend the definition of Passi to maps from algebraic objects of any type, i.e. from models of any algebraic theory.

\subsubsection{The classical case for monoids}\label{recall:PassisPn}

Let $M$ be a monoid with a unit element $e$ and $I(M)$ be the augmentation ideal of the monoid algebra $\mathbb{Z}[M]$ (i.e. the kernel of the map $\epsilon: \Z[M] \to \Z$ such that $\epsilon (\sum_{m \in M} \alpha_m m)=\sum_{m \in M} \alpha_m$). We know that $I(M)$ is a free abelian group having as free basis the elements $m-1$ where $m \in M$ and $m \neq e$. For $k\geq 0$, we denote by $I^k(M)$ the $k-$th power of $I(M)$ with the convention $I^0(M)=\Z[M]$. Let $A$ be an abelian group, for a map $f: M \to A$ let $\bar{f}: \Z[M] \to A$ denote the extension of $f$ to a $\Z$-linear homomorphism.

\begin{defi}\cite{Passi}
Let $M$ be a monoid and $A$ be an abelian group. A map $f: M \to A$ such that $f(e)=0$ is polynomial of degree lower or equal to $n$ if $\bar{f}$ vanishes on $I^{n+1}(M)$. 
\end{defi}

The abelian group $I(M)/I^{n+1}(M)$ is denoted by $P_n(M)$ and is called the Passi group of $M$ of order $n$. The canonical quotient map $I(M) \to P_n(M)$ is denoted by $\rho_n$.  The following Proposition explains the importance of the Passi group.
\begin{prop}\cite{Passi} \label{Passi}
Let $M$ be a monoid and $f_n: M \to P_n(M)$ given by $f_n(x)=\rho_n(x-1)$, then $f_n$ is the universal polynomial map of degree $n$ from $M$ to abelian groups.
\end{prop}

Let $Map_n(M,A)$ be the abelian group of polynomial maps $f:M \to A$ of degree $\leq n$. By the previous Proposition we have:
\begin{equation} \label{eq}
Map_n(M,A) \simeq Hom_{Ab}(P_n(M),A).
\end{equation}

In the sequel, we need the following elementary result which is wellknown.

\begin{prop} \label{prop-cog-2}
For $G \in Gr$ we have a natural isomorphism of abelian groups:
$$\theta: P_1(G)=I(G) / I^2(G)\xrightarrow{\simeq} G^{ab}:=G/[G,G]$$
such that for $g \in G$, $\theta(\overline{g-1})=\overline{g}$.
\end{prop}

Let $Gr$ be the category of all groups and $Mon$ the category of all monoids.
The forgetful functor $U: Gr \to Mon$ admits a left adjoint denoted by ${\bf g}$ and named groupification functor.
The unit of this adjunction $\eta_M: M \to U({\bf g}(M))$ for $M \in Mon$ induces a homomorphism:
$$\eta_M^*: Map_n(U{\bf g}(M), A) \to Map_n(M,A).$$
In the sequel, we prove the following result:

\begin{prop} \label{poly-map}
Let $M$ be a monoid and $A$ be an abelian group, the homomorphism:
$$\eta_M^*\colon Map_n(U{\bf g}(M), A) \to Map_n(M,A)$$
is an isomorphism.
\end{prop}
By (\ref{eq}) it is sufficient to prove the following lemma:

\begin{lm} \label{iso-ab-gr}
Let $M$ be a monoid. Then the map
$$P_n(\eta)\colon P_n(M) \to P_n(U{\bf g}(M))$$
 is an isomorphism of abelian groups.
\end{lm}

\begin{proof}     
Recall that for a monoid (resp. group) $N$ the natural injection $\alpha_N\colon N\to \Z[N]$  is the unit of the adjunction between monoids (resp. groups) and rings, sending $N$ to $\Z[N]$ and a ring $R$ to itself viewed as a multiplicative monoid (resp. to its group of units $R^{\times}$). Moreover, the augmentation map $\epsilon\colon \Z[N]\to \Z$ induces an augmentation map $\overline{\epsilon}$ on the quotient ring $\Z[N]/I^{n+1}(N)$ such that $\overline{\epsilon} q_n^N = \epsilon$, where $q_n^N\colon \Z[N]\to \Z[N]/I^{n+1}(N) $ is the canonical projection. 

It follows that the map $\eta_M\colon M \to U({\bf g}(M))$ successively induces  unique homomorphism of augmented rings $\Z[M] \to \Z[U({\bf g}(M)) ]$ and $$\overline{\eta_M}\colon \Z[M]/ I^{n+1}(M) \to \Z[U({\bf g}(M)) ]/ I^{n+1}(U({\bf g}(M))) =  \Z[ {\bf g}(M) ]/ I^{n+1}( {\bf g}(M)) $$ such that $
\overline{\eta_M} q_n^M \alpha_M = q_n^{ U({\bf g}(M)) }  \alpha_{U({\bf g}(M))} \eta_M$. So it suffices to construct an augmented ring homomorphism $\overline{\gamma_M} \colon \Z[ {\bf g}(M) ]/ I^{n+1}( {\bf g}(M)) \to  \Z[M]/ I^{n+1}(M)$ inverse to $\overline{\eta_M}$, since by restriction to augmentation ideals we then obtain the desired isomorphism $P_n(\eta_M)$ and its inverse.

The key point is that the multiplicative submonoid $q_n^M(1+I(M) )$ of $\Z[M]/ I^{n+1}(M)$ actually is a subgroup of its group of units since for $x\in  I(M)$, $q_n(\sum_{k=0}^n (-1)^k x^k)$ is a multiplicative inverse of $q_n(1+x)$. Thus the monoid homomorphism $\beta_n\colon M \to U(q_n^M(1+I(M) ))$, $\beta_n(m) = q_n(m) = q_n(1+(m-1))$, extends to a group homomorphism $\beta_n'\colon {\bf g}(M) \to q_n^M(1+I(M)) $ such that $\beta_n'\eta_M = \beta_n$ since $\eta_M$ is the unit of the adjunction between monoids and groups. In turn, $\beta_n'$ induces a unique ring homomorphism $\gamma_M\colon \Z[ {\bf g}(M)] \to \Z[M]/ I^{n+1}(M)$ such that $\gamma_M \alpha_{ {\bf g}(M)}$ is the composite map $\xymatrix{ {\bf g}(M) \ar[r]^-{\beta_n'} & q_n^M(1+I(M)) \hookrightarrow  \Z[M]/ I^{n+1}(M)}$. Moreover, $\gamma_M$ commutes with the augmentation maps, hence induces the desired augmented ring homomorphism 
$\overline{\gamma_M} \colon \Z[ {\bf g}(M)]/ I^{n+1}({\bf g}(M)) \to \Z[M]/ I^{n+1}(M)$ such that $\overline{\gamma_M} q_n^{{\bf g}(M)} = \gamma_M$. To check that $\overline{\eta_M}$ and $\overline{\gamma_M}$ indeed are mutually inverse ring isomorphisms, it suffices to precompose the composite maps $\overline{\gamma_M} \,\overline{\eta_M}$ and $\overline{\eta_M} \, \overline{\gamma_M}$
with the maps $\xymatrix{ M \ar[r]^-{\alpha_M}&\Z[M] }$ $\xymatrix{  \ar@{->>}[r]^-{q_n^M} &  \Z[M]/ I^{n+1}(M) }$ and $$\xymatrix{M \ar[r]^-{\eta_M} &U({\bf g}(M))={\bf g}(M) \ar[r]^-{\alpha_{{\bf g}(M)}} & \Z[{\bf g}(M)]  \ar@{->>}[r]^-{q_n^{{\bf g}(M)}} &  \Z[{\bf g}(M)]/ I^{n+1}({\bf g}(M)) }$$
and to successively use their respective universal properties.
\end{proof}
\bigskip

\begin{rem}
The isomorphism $Map_n(U{\bf g}(M), A) \simeq Map_n(M,A)$ is the generalization to the non-abelian setting of Lemma $8.1$ in \cite{BDFP}.
\end{rem}

\subsubsection{Generalization to algebraic theories}
		
	We introduce a generalized Passi functor $P_n^{\mathcal{P}}$ 
	from algebras over any unitary set-operad 	$\mathcal{P}$ to abelian groups, and give a conceptual interpretation which is new even in the classical case:
	 it turns out that $P_n^{\PP}$ is the polynomialization of degree $\le n$ of the canonical reduced functor from $\mathcal{P}$-algebras to abelian groups. This is used in the next section to describe the category $T_n\bar{\Z}[\mathcal{P}\ti alg]$ in terms of the Passi functor, at least for ``good'' operads $\mathcal{P}$.\bigskip
	
	For sets $A_1,\ldots, A_m$ recall the canonical isomorphism
	\[ \Z[A_1]\otimes \ldots\otimes \Z[A_m] \to  \Z[A_1\times \ldots \times A_m] \,,\quad a_1\otimes \ldots\otimes a_m\mapsto (a_1,\ldots,a_m)\]
	for $a_k\in A_k$.
	
	\begin{defi}Let $\mathcal{P}$ be a unitary set-operad with composition operations $\gamma_{m;i_1,\ldots,i_m}$ as before. Then an operad $\Z[\mathcal{P}]$ in $Ab$ is defined by $\Z[\mathcal{P}](m) = \Z[\mathcal{P}(m)]$ and composition operations
	\[\xymatrix{\overline{\gamma}_{m;i_1,\ldots,i_m}\colon\,\Z[\mathcal{P}(m)] \otimes \Z[\mathcal{P}(i_1)] \otimes \ldots \otimes \Z[\mathcal{P}(i_m)]\ar[r]^-{\cong}&\Z[\mathcal{P}(m)\times \mathcal{P}(i_1) \times \ldots \times \mathcal{P}(i_m)] \ar[d]^{\Z[\gamma_{m;i_1,\ldots,i_m}]}\\	& \Z[\mathcal{P}(i_1+\ldots+i_m)]}\]
Similarly, if $A$ is a $\mathcal{P}$-algebra with operations $\mu_m\colon \mathcal{P}(m)\times_{\mathfrak{S}_m} A^m\to A$ then $\Z[A]$ is a 
	$\Z[\mathcal{P}]$-algebra with operations
	\[\xymatrix{\overline{\mu}_m\colon \Z[\mathcal{P}(m)] \otimes_{\mathfrak{S}_m} \Z[A]^{\otimes m} \ar[r]^-{\cong} & \Z[\mathcal{P}(m)\times_{\mathfrak{S}_m} A^m]
	\ar[r]^-{\Z[\mu_m]} & \Z[A]
	}\]
For $\theta\in\mathcal{P}(m)$ and $x_k\in\Z[A]$ we write $\overline{\theta}(x_1\otimes\ldots\otimes x_m):=\overline{\mu}_m(\theta\otimes
x_1\otimes\ldots\otimes x_m)$. The constant map $0\colon A\to *$ into the trivial $\mathcal{P}$-algebra gives rise to a $\Z[\mathcal{P}]$-algebra map $\epsilon=\Z[0]\colon \Z[A]\to \Z[*]=\Z$; its kernel is an ideal denoted by $I(A)$ and has the canonical basis $a-1_A,a\in A\backslash\{1_A\}$. Let
\[I^n_{\mathcal{P}}(A) := \sum_{m\ge n}\overline{\mu}_m(\Z[\mathcal{P}(m)]\otimes_{\mathfrak{S}_m} I(A)^{\otimes m}) = \sum_{m\ge n}\sum_{\theta\in\mathcal{P}(m)}\overline{\theta}(I(A)^{\otimes m})\,\]
\[\xymatrix{P_n^{\mathcal{P}}(A) = I(A)/I^{n+1}_{\mathcal{P}}(A)\quad \mbox{and}\quad q_n\colon I(A)\ar@{->>}[r]&P_n^{\mathcal{P}}(A)}.\]
	Finally, $P_n^{\mathcal{P}}$ is a functor from $\mathcal{P}$-algebras to $Ab$ such that the map $p_n\colon A \to P^{\mathcal{P}}_n(A)$ is natural.	
		\end{defi}
	
	Note that we do not decorate $I(A)$ by an index $\mathcal{P}$ since, in contrast with $I^{n+1}_{\mathcal{P}}(A)$, it does not depend on $\mathcal{P}$.
	
	\begin{exple}  \label{exple:PnAs}
	 For $\mathcal{P}=\C om$ and $\A s$, $P_n^{\mathcal{P}}$ coincides with the Passi functor recalled in section \ref{sec:Pnclass}: we have 
	$I^n_{{\mathcal{C}om}}(A)=\sum_{m\ge n}\overline{\mu}_m(\Z[\ul{1}]\otimes_{\mathfrak{S}_m} I(A)^{\otimes m}) = \sum_{m\ge n} I^m(A) = I^n(A)$ since $I^m(A) \subset I^n(A)$ for $m\ge n$, and similarly $I^n_{{\mathcal{A}s}}(A)=\sum_{m\ge n}\overline{\mu}_m(\Z[\mathfrak{S}_m]\otimes_{\mathfrak{S}_m} I(A)^{\otimes m}) = \sum_{m\ge n} I^m(A) = I^n(A)$.
	
	\end{exple}
	
	\newcommand{\palg}{$\mathcal{P}$-algebra}
	
 In order to extend the above definitions to pointed algebraic categories we recall some basic notions of universal algebra from the categorical viewpoint, cf.\ \cite{Bor2}.
	 
	 \begin{defi}\label{def:theory}
	 An algebraic theory is a category $\mathbb{T}$ with objects $\ul{m}$, $m\ge 0$, such that $\ul{m}=\ul{1}^m$, i.e.\ $\ul{m}$ is the $m$-th cartesian power of $\ul{1}$. A $\mathbb{T}$-model is a functor $M\colon \mathbb{T}\to Set$ which preserves finite products. The models of $\mathbb{T}$ form a category $\mathbb{T}\ti Mod$ with morphisms being the natural transformations.
	 
	 Moreover, we say that $\mathbb{T}$ is pointed if the object $\ul{0}$ is a zero-object in $\mathbb{T}$.
	 \end{defi}
	 
	 For a set $X$  let  $\mathcal{F}_{\mathbb{T}}(X)$ denote the free $\mathbb{T}$-model generated by $X$, and $\mathcal{F}(\mathbb{T})$ be the category with objects the sets $\underline{m}$ for $m\geq 0$, and with morphisms  $\mathcal{F}(\mathbb{T})(\underline{k},\underline{l}) = \mathbb{T} \ti Mod(\mathcal{F}_{\mathbb{T}}(\underline{k}),\mathcal{F}_{\mathbb{T}}(\underline{l}))$.
	 
	 \begin{rem}\label{rem:theories} It is often useful to 
	 think of the product decompositions $\ul{m}=\ul{1}^m$ as being specified, i.e.\
	 of the corresponding projection maps $p_k\colon \ul{m} \to \ul{1}$, $1\le k\le m$, being fixed.
	 
	 Any model  $M\colon \mathbb{T}\to Set$ takes the empty product $\ul{0}=\ul{1}^0$ to an empty product in $Set$, i.e.\ a one-point set. So if $\mathbb{T}$ is pointed the functor $M$ naturally takes values in the category of pointed sets $Set_*$: the basepoint of $M(\ul{m})$ is the element $M(0)(*)$ where $0\colon \ul{0}\to \ul{m}$ is the zero morphism and $*$ is the unique element of $M(\ul{0})$.
	 \end{rem}
	 
	 To any algebraic theory $\mathbb{T}$ one can associate a set-operad $\mathcal{P}_{\mathbb{T}}$, as follows. Let $\mathcal{P}_{\mathbb{T}}(m) = {\rm Hom}_{\mathbb{T}}(\ul{m},\ul{1})$ with the right action of $\mathfrak{S}_m$ induced by its natural left action 
	 on $\ul{m}$. To define the composition operation $\gamma$ let $m\ge 1$, $i_1,\ldots,i_m\ge 0$, $\theta\in {\rm Hom}_{\mathbb{T}}(\ul{m},\ul{1})$ and 
	 $\theta_k\in {\rm Hom}_{\mathbb{T}}(\ul{i_k},\ul{1})$ for $1\le k\le m$. Then first let $\ul{pr}_k\colon \ul{i_1+\ldots+i_m} \to \ul{i_k}$ be the unique map such that $
	 p_l\circ \ul{pr}_k = p_{i_1+\ldots+i_{k-1}+l}$ for $1\le l\le k$. Next let $\theta_1\times \ldots\times \theta_m\colon \ul{i_1+\ldots+i_m} \to \ul{m}$ be the unique map such that $
	p_k (\theta_1\times \ldots\times \theta_m) = \theta_k \ul{pr}_k$ for  $1\le k\le m$.
	 So we can define $\gamma(\theta;\theta_1,\ldots,\theta_m) = \theta\circ (\theta_1\times \ldots\times \theta_m)$.
	 
	 Moreover, we need the \textit{forgetful functor} $V_{\mathbb{T}\ti Mod}^{\mathcal{P}_{\mathbb{T}}\ti Alg}\colon \mathbb{T}\ti Mod \to \mathcal{P}_{\mathbb{T}}\ti Alg$ which is defined by $V_{\mathbb{T}\ti Mod}^{\mathcal{P}_{\mathbb{T}}\ti Alg}(M) =M(\ul{1})$ as a set, endowed with the $\mathcal{P}_{\mathbb{T}}$-algebra structure given by the maps
	 \[\xymatrix{
	 \mu_m\colon \mathcal{P}_{\mathbb{T}}(m)\times_{\mathfrak{S}_m} M(\ul{1})^m \ar[r]^-{\cong}& {\rm Hom}_{\mathbb{T}}(\ul{m},\ul{1}) \times_{\mathfrak{S}_m} M(\ul{m}) \ar[r] &  M(\ul{1})}\]
	 given by the isomorphism $M(\ul{1})^m \cong M(\ul{m})$ coming from the fact that $M$ preserves products, and by the evaluation map $(f,x)\mapsto M(f)(x)$ for $(f,x)\in
	 {\rm Hom}_{\mathbb{T}}(\ul{m},\ul{1}) \times M(\ul{m}) $.

	 So we can finally define the functor \[ P_n^{\mathbb{T}} := P_n^{\mathcal{P}_{\mathbb{T}}} \circ V_{\mathbb{T}\ti Mod}^{\mathcal{P}_{\mathbb{T}}\ti Alg}\colon \mathbb{T}\ti Mod\to Ab.\]

	 In the sequel, \textbf{$\mathbb{T}$ always denotes a pointed algebraic theory}; in this case the operad $\mathcal{P}_{\mathbb{T}}$ is \textbf{unitary}.

It is less obvious than in the example \ref{exple:PnAs} to see that for the theories of groups $\mathbb{G}r$ and of abelian groups $\mathbb{A}b$ the functor $P_n^{\mathbb{G}r}$ also coincides with Passi's functor $P_n$ in section \ref{recall:PassisPn}: in fact, $\mathcal{P}_{\mathbb{G}r}(m)=\mathcal{F}_{\mathbb{G}r}(\ul{m})$ since a theory is always isomorphic with the opposite category of the category of its finitely generated free models. For a solution of this problem see Proposition \ref{calculPnGr} below.

	For some letter $a$ let $a^{\sqcup k}=a,\ldots,a$ denote its $k$-fold repetition.
	
	\begin{lm}\label{ApIq} Let $A$ be a \palg. For $p\ge 0$, $q\ge 1$ let $a_1,\ldots,a_p\in A$, $x_1,\ldots,x_q\in I(A)$ and $\theta\in \mathcal{P}(p+q)$. Then 
	\[\overline{\theta}(a_1\otimes\ldots\otimes a_p\otimes x_1 \otimes \ldots\otimes x_q) \equiv 
	\overline{\gamma(\theta; 0_{\mathcal{P}}^{\sqcup p},1_{\mathcal{P}}^{\sqcup q})}(x_1 \otimes \ldots\otimes x_q)\quad\mbox{mod $I^{q+1}_{\PP} (A)$.}\]
	\end{lm}
	
	\begin{proof} 
	By induction on $p$, the case $p=0$ being clear. For $p\ge 1$,
	\begin{eqnarray*}
	\overline{\theta}(a_1\otimes\ldots\otimes a_p\otimes x_1 \otimes \ldots\otimes x_q) 
&=& 
\overline{\theta}(a_1\otimes \ldots\otimes a_{p-1}\otimes (a_p-1_A)\otimes  x_1 \otimes \ldots\otimes x_q)\\
& &{}+\overline{\theta}(a_1\otimes \ldots\otimes a_{p-1}\otimes 1_A\otimes x_1 \otimes \ldots\otimes x_q) \\ 
&\equiv&  \overline{\gamma(\theta; 0_{\mathcal{P}}^{\sqcup p-1},1_{\mathcal{P}}^{\sqcup q+1})}((a_p-1_A)\otimes x_1\otimes \ldots\otimes x_q) \\
 & &{}+\overline{\gamma(\theta;1_{\mathcal{P}}^{\sqcup p-1}, 0_{\mathcal{P}},1_{\mathcal{P}}^{\sqcup q})}(a_1\otimes \ldots \otimes a_{p-1}\otimes x_1 \otimes \ldots\otimes x_q) \\ 
 & & \mbox{mod $I^{q+2}(A)$}\\
&\equiv& \overline{\gamma(\gamma(\theta;1_{\mathcal{P}}^{\sqcup p-1}, 0_{\mathcal{P}},1_{\mathcal{P}}^{\sqcup q}),0_{\mathcal{P}}^{\sqcup p-1},1_{\mathcal{P}}^{\sqcup q})}( x_1 \otimes \ldots\otimes x_q) \\ 
& & \mbox{mod $I^{q+1}(A)$}\\
&\equiv& \overline{\gamma(\theta; 0_{\mathcal{P}}^{\sqcup p},1_{\mathcal{P}}^{\sqcup q})}(x_1 \otimes \ldots\otimes x_q)
\end{eqnarray*}
by associativity of $\gamma$.
\end{proof}
	
	\begin{cor}\label{modI2}
	Let $A$ be a \palg. For $n\ge 1$ the subgroup $I^n_{\PP}(A)$ is an ideal of $\Z[A]$. Moreover, for $\theta\in \mathcal{P}(p)$ and $a_1,\ldots,a_p \in A$ we have
	\[  \theta(a_1,\ldots,a_p) - 1_A \equiv (\theta_1(a_1) - 1_A)  + \ldots +  (\theta_p(a_p) - 1_A)\quad\mbox{{\rm mod} $I^{2}_{\PP}(A)$}\]
where $\theta_k = \gamma(\theta; 0_{\mathcal{P}}^{\sqcup k-1},1_{\mathcal{P}},0_{\mathcal{P}}^{\sqcup p-k})$.
\end{cor}
	
\begin{proof} 	
For $m\ge n$, $a_1,\ldots,a_p,b_1,\ldots,b_q \in A$, $x_1,\ldots,x_m\in I(A)$, $\vartheta\in \mathcal{P}(m)$  and $\theta\in \mathcal{P}(p+1+q)$, there is an obvious permutation $\sigma$ such that
\[  \overline{\theta}(a_1\otimes \ldots\otimes a_p \otimes \vartheta (x_1\otimes \ldots\otimes x_m) \otimes b_1 \otimes \ldots\otimes b_q )\hspace{2cm} \]
\begin{eqnarray*}
&=& \overline{(\theta.\sigma)} (a_1\otimes \ldots\otimes a_p \otimes b_1 \otimes \ldots\otimes b_q \otimes \vartheta(x_1\otimes \ldots\otimes x_m ))\\
&=&\overline{ \gamma(\theta.\sigma ; 1_{\mathcal{P}}^{\sqcup p+q},\vartheta)}(a_1\otimes \ldots\otimes a_p \otimes b_1 \otimes \ldots\otimes b_q \otimes x_1\otimes \ldots\otimes x_m )
\end{eqnarray*}
So the first assertion is immediate from Lemma \ref{ApIq}.
The second assertion is proved by induction on $p$, the case $p=1$ being clear. For $p\ge 1$, we use Proposition \ref{ApIq} to see that modulo $I^2_{\PP}(A)$,
\begin{eqnarray*}
\theta(a_1,\ldots,a_p) - 1_A 
&=& \overline{\theta}(a_1\otimes\ldots \otimes a_{p-1}\otimes (a_p - 1_A))+
\theta(a_1,\ldots,a_{p-1},1_A) - 1_A\\
&\equiv&  ({\gamma(\theta;0_{\mathcal{P}}^{\sqcup p-1},1_{\mathcal{P}})}(a_p) - 1_A ) +\gamma(\theta;1_{\mathcal{P}}^{\sqcup p-1},0_{\mathcal{P}})
(a_1,\ldots,a_{p-1})- 1_A\\
&\equiv&  (\theta_p(a_p) - 1_A) + \sum_{k=1}^{p-1}  \left( \gamma(\gamma(\theta;1_{\mathcal{P}}^{\sqcup p-1},0_{\mathcal{P}});0_{\mathcal{P}}^{\sqcup k-1},1_{\mathcal{P}},0_{\mathcal{P}}^{\sqcup p-1-k})(a_k)- 1_A \right)\\
&\equiv&  (\theta_p(a_p) - 1_A) + \sum_{k=1}^{p-1}   (\theta_k(a_k) - 1_A)
\end{eqnarray*}
\end{proof}

	\begin{lm}\label{P1linear}
	The functors $P_1^{\mathcal{P}}\colon \mathcal{P}\ti Alg \to Ab$ and
	$P_1^{\mathbb{T}}\colon \mathbb{T}\ti Mod \to Ab$
	are linear.
	\end{lm}
\begin{proof} 	
	Let $A,B$ be $\mathcal{P}$-algebras. Then $A\vee B$ consists of the elements $\theta(a_1,\ldots,a_p,b_1,\ldots,b_q)$ with $p,q\ge 0$, $\theta\in  \mathcal{P}(p+q)$, $a_1,\ldots,a_p\in A$, $b_1,\ldots,b_q\in B$. Hence by Corollary \ref{modI2} we have $I(A\vee B)= I(A)+I(B)+I^2(A\vee B)$. It follows that the section $s_{A,B}\colon P_1^{\mathcal{P}}(A)\oplus P_1^{\mathcal{P}}(B) \to P_1^{\mathcal{P}}(A\vee B)$, defined by $s_{A,B}=(P_1^{\mathcal{P}}(i_1),P_1^{\mathcal{P}}(i_2))$ of $\hat{r}^{P_1^{\mathcal{P}}}$ is surjective, whence an isomorphism. Hence $\hat{r}^{P_1^{\mathcal{P}}}$ is an isomorphism, as desired. The same argument works for $P_1^{\mathbb{T}}$.
	\end{proof}
	
	\begin{lm}\label{Pnexsequ}
	There is a natural exact sequence
	\begin{equation} \label{ses-Pnexsequ}
\xymatrix{
	\Z[\mathcal{P}(n)]\otimes  P_1^{\mathcal{P}}(A)^{\otimes n} \ar[r]^-{\nu_n} & P_n^{\mathcal{P}}(A)\ar[r]^-{\overline{q_{n-1}}} & P_{n-1}^{\mathcal{P}}(A) \ar[r] &0}
\end{equation}

	of funtors on $\mathcal{P}$-$Alg$ where $\nu_n(\theta\otimes \overline{a_1-1} \otimes \cdots\otimes \overline{a_n-1}) = q_n(\overline{\theta}((a_1-1)\otimes \cdots\otimes(a_n-1)))$ for $\theta\in \mathcal{P}(n)$ and $a_1,\ldots,a_n\in A$.
	\end{lm}	
	
\begin{proof} 	
	By definition of $P_n^{\mathcal{P}}$ we have an exact sequence
	\[ \xymatrix{
	\Z[\mathcal{P}(n)]\otimes I(A)^{\otimes n} \ar[r]^-{\overline{\mu}_n} & P_n^{\mathcal{P}}(A)\ar[r]^-{\overline{q_{n-1}}} & P_{n-1}^{\mathcal{P}}(A) \ar[r] &0}.\]
So it suffices to show that $\overline{\mu}_n$ factors through $id\otimes q_1^{\otimes n} \colon \Z[\mathcal{P}(n)]\otimes I(A)^{\otimes n} \to \Z[\mathcal{P}(n)]\otimes P_1^{\mathcal{P}}(A)^{\otimes n}$. But by right-exactness of the tensor product,
\begin{eqnarray*}
{\rm Ker}(id\otimes q_1^{\otimes n} ) &=& \sum_{i=1}^n\sum_{j\ge 2} \overline{\mu}_n\left(
\Z[\mathcal{P}(n)]\otimes I(A)^{\otimes i-1} \otimes \overline{\mu}_j(\Z[\mathcal{P}(n)]\otimes I(A)^{\otimes j} ) \otimes I(A)^{\otimes n-i} \right)\\
&\subset&  \sum_{i=1}^n\sum_{j\ge 2} \overline{\mu}_{n-1+j}\left(
\Z[\mathcal{P}(n-1+j)]\otimes I(A)^{\otimes n-1+j}\right) 
\end{eqnarray*}
by associativity of $\overline{\mu}$, so $\overline{\mu}_n {\rm Ker}(id\otimes q_1^{\otimes n} ) =0$ since $n-1+j>n$ for $j\ge 2$.
\end{proof}

\begin{prop}\label{Pnpoldegn}
	The functors $P_n^{\mathcal{P}}\colon \mathcal{P}\ti Alg \to Ab$ and
	$P_n^{\mathbb{T}}\colon \mathbb{T}\ti Mod \to Ab$
	are polynomial of degree $\le n$.
	\end{prop}	
	
\begin{proof} 
By induction on $n$:
By the Lemmas \ref{P1linear} and Proposition \ref{polopol} the functor $\Z[\mathcal{P}(n)]\otimes P_1^{\mathcal{P}}(-)^{\otimes n}\colon 
\mathcal{P}$-$Alg \to Ab$ is polynomial of degree $\le n$, hence so is $P_n^{\mathcal{P}}$ since $P_{n-1}^{\mathcal{P}}$ is polynomial of degree $\le n-1$ by induction and the category of polynomial  functors of degree $\le n$ is thick. The same argument works for $P_n^{\mathbb{T}}$.
\end{proof}	

Now we are ready to give a conceptual interpretation of the functors $P_n^{\mathcal{P}}$ and
	$P_n^{\mathbb{T}}$. 
Recall that a category $\mathcal{C}$ is concrete if there exists a faithful forgetful functor $V_{\mathcal{C}}\colon \C\to Set$. Morally a concrete category is a category whose objects are sets endowed with some structure and the morphisms are set maps compatible with this structure in some suitable sense.
If in addition $\C$ is pointed (with the one-point set being its null-object) then  $V_{\mathcal{C}}$ naturally takes values in $Set_*$ (where $Set_*$ is the category of pointed sets). Thus any pointed concrete category $\C$ admits a \textit{canonical reduced functor} $\bar{\Z}[-]\circ V_{\mathcal{C}}\colon \C \to Ab$. Note that $\mathbb{T}\ti Mod$ is a pointed concrete category where $V_{\mathbb{T}\ti Mod}(M) = M(\ul{1})$ with basepoint $1_{M(\ul{1})} = M(0_{\ul{0},\ul{1}})(*)$, see Remark \ref{rem:theories}.

\begin{thm}\label{Pn=Tn} 
There are canonical isomorphisms 
\[ \xymatrix{
\Xi_n^{\mathcal{P}} \colon T_n(\bar{\Z}[-] \circ V_{\mathcal{P}\ti Alg}) \ar[r]^-{\cong} & P_n^{\mathcal{P}}  \quad\mbox{and}\quad
\Xi_n^{\mathbb{T}} \colon T_n(\bar{\Z}[-] \circ V_{\mathbb{T}\ti Mod})\ar[r]^-{\cong} & P_n^{\mathbb{T}}\circ V_{\mathbb{T}\ti Mod}^{\mathcal{P}_{\mathbb{T}}\ti Alg}
}\]
of functors on $\mathcal{P}$-$Alg$ and on $\mathbb{T}$-$Mod$, resp., sending $\bar{a}$ and $\bar{x}$ to  $q_n(a-1_A)$ and $q_n(x-1_{M(\ul{1})})$ where $a\in A$ for $A \in \mathcal{P}$-$Alg$ and $x\in M(\ul{1})$ for $M\in \mathbb{T}$-$Mod$, resp.
\end{thm}

\begin{proof}
The natural isomorphisms $\xymatrix{\Xi^{\mathcal{P}} \colon \bar{\Z}[A] \ar[r]^-{\cong} & I(A)}$ and $\xymatrix{\Xi^{\mathbb{T}} \colon \bar{\Z}[M(\ul{1})] \ar[r]^-{\cong} & I(M(\ul{1}))}$ sending $\bar{a}$ to $a-1_A$ and $\bar{x}$ to $x-1_{M(\ul{1})}$ allow to replace the functors $\bar{\Z}[-]\circ V_{\mathcal{P}\ti Alg}$ and $\bar{\Z}[-]\circ V_{\mathbb{T}\ti Mod}$ with $I$ and $I\circ V_{\mathbb{T}\ti Mod}^{\mathcal{P}_{\mathbb{T}}\ti Alg}$, resp. By Corollary \ref{Pnpoldegn} the functor $P_n^{\mathcal{P}}$ is polynomial of degree $\le n$, hence $\xymatrix{q_n\colon I\ar@{-{>>}}[r]& P_n^{\mathcal{P}}}$ factors through 
$\xymatrix{t_n^I\colon I \ar@{-{>>}}[r]& T_nI}$, whence $\Xi_n^{\mathcal{P}}$ is well defined. To show that it has an inverse it suffices to show that the map $(t_n^I)_A$ annihilates $I^{n+1}_{\mathcal{P}}(A) = \sum_{m>n}\overline{\mu}_m ( \Z[\mathcal{P}(m)]\otimes I(A)^{\otimes m})$. For $m>n$ consider the multifunctor $M_m\colon 
(\mathcal{P}$-$Alg)^m\to Ab$ defined by $(A_1,\ldots,A_m) \mapsto 
\Z[\mathcal{P}(m)]\otimes I(A_1)\otimes\cdots \otimes I(A_m)$, which is multi-reduced. The morphism $q_n\overline{\mu}_m\colon M_m\Delta^m \to P_n^{\mathcal{P}}$ factors through $\xymatrix{t_n\colon M_m\Delta^m\ar@{-{>>}}[r]& T_n(M_m\Delta^m)}$ since $P_n^{\mathcal{P}}$ is polynomial of degree $\le n$, but $T_n(M_m\Delta^m)=0$ by Lemma \ref{TnMDelm}, as desired. For $\Xi_n^{\mathbb{T}}$ the same argument works.
\end{proof}

We now are ready to compare the functors $P_n^{\mathbb{T}}$ for $\mathbb{T}$ being the theory of groups $\mathbb{G}r$ or of abelian groups $\mathbb{A}b$ with the classical Passi functor of section \ref{recall:PassisPn}.

\begin{prop}\label{calculPnGr} There are natural isomorphisms 
\[ P_n\hspace{1mm}\cong\hspace{1mm} T_n(\bar{\Z}[-]\circ V_{Gr})
\hspace{1mm}\cong\hspace{1mm} P_n^{\mathbb{G}r} \quad\mbox{and}\quad
P_n\hspace{1mm}\cong\hspace{1mm} T_n(\bar{\Z}[-]\circ V_{Ab})
\hspace{1mm}\cong\hspace{1mm} P_n^{\mathbb{A}b}
\]
 of functors from $Gr$ to $Ab$ and from $Ab$ to $Ab$, respectively.
\end{prop}

\begin{proof} The right-hand isomorphisms (invoking $P_n^{\mathbb{G}r} $ and $P_n^{\mathbb{A}b}$) were already established in Theorem \ref{Pn=Tn}. Now the left-hand isomorphisms are obtained in exactly same way, based on the well known exact sequence $(G^{ab})^{\otimes n} \to P_n(G) \to P_{n-1}(G) \to 0$ (proved exactly in the same way as sequence (\ref{ses-Pnexsequ})) and 
the fact that $P_1(G)\cong G^{ab}$ for any group $G$ (see Proposition \ref{prop-cog-2}), whence the functor $P_1$ from $Gr$ and from $Ab$ to $Ab$ is linear in each case.
\end{proof}

\subsection{Polynomial maps on the morphism sets of a category}
Let $\A$ be a small additive category. In \cite{Pira} the author defines a preadditive category $P_n\A$ in the following way: the objects of $P_n\A$ are those of $\A$ and $Hom_{P_n\A}(A,B)=P_n(Hom_{\A}(A,B))$. The importance of this category lies on the following Proposition.
\begin{prop} \cite{Pira} \label{Pira-Pn}
Let $\A$ be a small additive category. Then:
$$Pol_{n}(\A, Ab) \simeq Lin(P_n\A, Ab).$$
\end{prop}
In the next section, we extend this result from $\A$ to any category $\C$ \textbf{having a null object and finite sums} (and from $Ab$ to any abelian category).

In order to define a suitable analogue of the category $P_n\C$ in this context we need to define polynomial maps on the set of morphisms of $\C$. We prove in Proposition \ref{Pn-add} that for $\A$ a small additive category this definition is equivalent  with the definition of Passi polynomial maps.

Let $\C$ be a category having a null object and finite sums, and let $n\ge 1$. 
	
	\begin{defi} \label{polmap-setmor}
	Let $X,Y$ be objects in $\C$, $A$ be an abelian group and $\varphi\colon {\rm Hom}_{\mathcal{C}}(X,Y)\to A$ be a normalized function, i.e. $\varphi(0)=0$. We say that $\varphi$ is polynomial of degree $\le n$ if its $\Z$-linear extension $\hat{\varphi}\colon U^{\mathcal{C}}_X(Y) \to A$ factors through the quotient map $\xymatrix{t_n\colon U^{\mathcal{C}}_X(Y) \ar@{->>}[r] & T_nU^{\mathcal{C}}_X(Y)\,.}$ 
	\end{defi}
	
It is clear from the definition that the map $t_n'\colon  {\rm Hom}_{\mathcal{C}}(X,Y) \to T_nU^{\mathcal{C}}_X(Y)$, $t_n'(f) =t_n(f)$, is the \textit{universal	polynomial map of degree $\le n$} on ${\rm Hom}_{\mathcal{C}}(X,Y)$.
	
	\begin{rem}
 The previous definition generalizes the definition of quadratic maps from morphism sets of $\C$ to abelian groups given in \cite{HV} section $2.3$.
 \end{rem}

 	\begin{prop} \label{TnZ}
	There exists a preadditive category $T_n\bar{\Z}[\C]$ with the same objects as $\C$ whose morphism sets are given by ${\rm Hom}_{T_n\bar{\Z}[\mathcal{C}]}(X,Y) =T_nU^{\mathcal{C}}_X(Y)$ and whose composition satisfies $t_ng\circ t_nf = t_n(g\circ f)$.
	\end{prop}
	
\begin{proof} Consider the homomorphism $\gamma\colon U^{\mathcal{C}}_Y(Z)\otimes U^{\mathcal{C}}_X(Y) \to U^{\mathcal{C}}_X(Z)$ defined by $\gamma(g\otimes f) = g\circ f$ for $g\in  {\rm Hom}_{\mathcal{P}}(Y,Z),$ $f\in  {\rm Hom}_{\mathcal{P}}(X,Y)$. We must show that $t_n\circ \gamma$ factors through $\xymatrix{t_n\otimes t_n\colon U^{\mathcal{C}}_Y(Z)\otimes U^{\mathcal{C}}_X(Y) \ar@{->>}[r] & T_nU^{\mathcal{C}}_Y(Z)\otimes T_nU^{\mathcal{C}}_X(Y)\,.}$ First note that a map $f\in  {\rm Hom}_{\mathcal{P}}(X,Y)$ induces a natural transformation $f^*\colon U^{\mathcal{C}}_Y\to U^{\mathcal{C}}_X$ and hence another natural transformation $T_n(f^*)\colon T_nU^{\mathcal{C}}_Y\to T_nU^{\mathcal{C}}_X$. Now consider the following diagram
\[\xymatrix{
U^{\mathcal{C}}_Y(Z)\otimes T_nU^{\mathcal{C}}_X(Y) \ar[r]^-{a} & T_nU^{\mathcal{C}}_X(Z) \\
U^{\mathcal{C}}_Y(Z)\otimes U^{\mathcal{C}}_X(Y) \ar[r]^-{\gamma} \ar[u]^{1\otimes t_n} \ar[d]_{t_n\otimes 1} & U^{\mathcal{C}}_X(Z) \ar@{->>}[u]^{t_n}\ar@{->>}[d]_{t_n}\\
T_nU^{\mathcal{C}}_Y(Z)\otimes U^{\mathcal{C}}_X(Y) \ar[r]^-{b} & T_nU^{\mathcal{C}}_X(Z)
}\]
where $a(g\otimes t_nf) = T_nU^{\mathcal{C}}_X(g)(t_nf) = t_n(gf)$ and $b(t_ng \otimes f) = T_n(f^*)(t_ng) = t_n(gf)$.
As the diagram commutes we have ${\rm Ker}(1\otimes t_n) \subset {\rm Ker}(t_n\gamma)\supset {\rm Ker}(t_n\otimes 1)$; but ${\rm Ker}(t_n\otimes t_n) = {\rm Ker}(1\otimes t_n) + {\rm Ker}(t_n\otimes 1)$ by right-exactness of the tensor product, so $t_n\gamma$ indeed factors through $t_n\otimes t_n$, as desired.
\end{proof}	

We now exhibit a first case where the category $T_n\bar{\Z}[\C]$ can be described in terms of the Passi functor $P_n$; \textit{we  generalize this to \palg s in the next section.}
	
\begin{prop}\label{Pn-add} 
Suppose that $\C$ is an additive category. Then for $X,Y\in \C$ there is an  isomorphism $T_nU_X^{\mathcal{C}}(Y) \,\cong P_n( {\rm Hom}_{\mathcal{C}}(X,Y))$ natural in $Y$.
\end{prop}	
	
	\begin{proof}
	The functor $U_X^{\mathcal{C}}$ factors as $\xymatrix{\C \ar[r]^-{ {\rm Hom}_{\mathcal{C}}(X,-)} & Ab \ar[r]^-{V_{Ab}} & Set_* \ar[r]^-{\bar{\Z}[-]} & Ab}$. Thus
	\begin{eqnarray*}
	T_nU_X^{\mathcal{C}} &=&  T_n((\bar{\Z}[-]\circ V_{Ab}) \circ  {\rm Hom}_{\mathcal{C}}(X,-))\\
	&\cong& T_n(T_n(\bar{\Z}[-]\circ V_{Ab}) \circ  {\rm Hom}_{\mathcal{C}}(X,-))\quad \mbox{by Lemma \ref{Tn(GoF)}}\\
	&\cong&T_n(\bar{\Z}[-]\circ V_{Ab}) \circ {\rm Hom}_{\mathcal{C}}(X,-)\quad \mbox{by Proposition \ref{polopol}}\\
	&\cong& P_n\circ {\rm Hom}_{\mathcal{C}}(X,-) \quad \mbox{by Proposition \ref{calculPnGr}}
\end{eqnarray*}
\end{proof}

\subsection{Characterization of polynomial functors}
Polynomial maps on the set of morphisms allow to characterize polynomial functors on pointed categories with finite sums, in terms of their effect on morphisms instead of their effect on objects which is used to define polynomiality of functors. Actually Theorem \ref{thm:polfunc-polmap} generalizes the corresponding one for polynomial functors between abelian categories due to Eilenberg and MacLane \cite{EML}. As a consequence of this Theorem we obtain a description of polynomial functors on $\C$ in terms of additive functors on $T_n\bar{\Z}[\C]$ extending Proposition  \ref{Pira-Pn}.

	\begin{thm}\label{thm:polfunc-polmap} Let $\mathcal{A}$ be an abelian category and $F\colon \C \to \A$ be a reduced functor. Then the following are equivalent:\
	
	\begin{enumerate}
	
	\item $F$ is polynomial of degree $\le n$.
	
	\item For all $X,Y\in \C$ the map $F_{XY}\colon {\rm Hom}_{\mathcal{C}}(X,Y)\to {\rm Hom}_{\mathcal{A}}(FX,FY)$, $F_{XY}(f)=F(f)$,  is polynomial of degree $\le n$.
	
	\end{enumerate}
	
	If $\A=Ab$, (1) and (2) are equivalent with

	\begin{enumerate}
	\setcounter{enumi}{2}

	\item Any natural transformation $U_X^{\mathcal{C}} \to F$ factors through $\xymatrix{t_n\colon U_X^{\mathcal{C}}  \ar@{->>}[r] & T_nU_X^{\mathcal{C}}}$, so if $\C$ is small  the map $t_n^*\colon {\rm Hom}(T_nU_X^{\mathcal{C}},F) \to {\rm Hom}(U_X^{\mathcal{C}},F)$ is an isomorphism.
	\end{enumerate}
	
	\end{thm}
	
	The proof of this theorem relies on the following lemma.
	
	\begin{lm}\label{addopol} 
	
	Let $\xymatrix{\C\ar[r]^F & \A\ar[r]^G& \mathcal{B}}$ be functors where $\A$ and $\mathcal{B}$ are abelian categories and $F$ is reduced. Then:
	
	\begin{enumerate}
	
	\item If $G$ is additive there is a natural isomorphism $G(F(X_1|\ldots|X_m)) \,\cong\, (G\circ F)(X_1|\ldots|X_m)$ for $m\ge 2$ and $X_1,\ldots,X_m \in \C$.
	
	\item Let $\A'$ be a full subcategory of $\A$ containing all objects $F(X_1|\ldots|X_m)$, $m\ge 1$ and $X_1,\ldots,X_m \in \C$. If $G$ is additive and faithful on $\A'$ then $F$ is polynomial of degree $\le n$ iff $G\circ F$ is.
	
		\end{enumerate}

\end{lm}
\begin{proof}
Let $X,Y\in C$. As $F$ is reduced the exact sequence $$\xymatrix{0\ar[r] &F(X|Y) \ar[r]^{\iota^F_{X,Y}} & F(X\vee Y)\ar[r]^-{\hat{r}^F} & F(X) \oplus F(Y) \ar[r] & 0}$$ is split by the section $(F(i^2_{\hat{2}}),F(i^2_{\hat{1}})) \colon F(X)\oplus F(Y) \to F(X\vee Y)$. Whence the sequence
\[ \xymatrix{0\ar[r] &G(F(X|Y)) \ar[r]^{G(\iota^F_{X,Y})} &(GF)(X\vee Y)\ar[r]^-{G(\hat{r}^F)} & G(F(X) \oplus F(Y)) \ar[r] & 0}\] 
is short exat since $G$ is additive. But the composite map
\[\xymatrix{GF(X\vee Y) \ar[r]_-{G(\hat{r}^F)} & G(F(X) \oplus F(Y)) \ar[rr]_-{(G(r^2_{\hat{2}}),G(r^2_{\hat{1}}))^t}^{\cong} && GF(X) \oplus GF(Y)}\]
equals $\hat{r}^{GF}$, whence $G(F(X|Y)) \cong (GF)(X|Y)$. 

So we can assume that (1) holds by induction for $m\ge 2$. Then for $X_1,\ldots,X_{m+1}\in \C $ we have natural isomorphisms
\begin{eqnarray*}
G(F(X_1|\ldots|X_{m+1})) &\cong& G(F(X_1|\ldots|X_{m-1}|-)(X_m|X_{m+1}))\quad \mbox{by Definition \ref{cross-eff}}\\
&\cong& (G\circ F(X_1|\ldots|X_{m-1}|-))(X_m|X_{m+1})\quad\mbox{by (1) in the case $m=2$}\\
&\cong& (G\circ F)(X_1|\ldots|X_{m-1}|-)(X_m|X_{m+1})\quad\mbox{by induction hypothesis}\\
&\cong& (G\circ F)(X_1|\ldots|X_{m+1})\quad
\end{eqnarray*}
Now (2) follows from (1) and the fact that $\forall A \in \mathcal{A}'$, $G(A)=0$ iff $G(1_A)=G(0_A)$ iff $1_A=0_A$ since $G$ is faithful. But the last equality is equivalent to $A=0$. 
\end{proof}

	\begin{proof}[Proof of Theorem \ref{thm:polfunc-polmap}]
	
	Both (1) and (2) hold iff they hold for the restriction of $F$ to any small full subcategory $\C'$ closed under finite sums. So we can suppose that $\C$ is small.
	
	In a  first step we suppose that $\A=Ab$. We start by proving that here the statements (1) and (2) are both equivalent with (3). Clearly (1) implies (3) since $T_n$ is the left adjoint of the forgetful functor by Proposition \ref{Tn-prop}, so let us prove the converse. The Yoneda lemma provides an epimorphism of functors $\xymatrix{\alpha\colon \bigoplus_{X\in \mathcal{C}}\bigoplus_{x\in FX} U_X^{\mathcal{C}} }$ $\xymatrix{\ar@{->>}[r] &F\,.}$ But the restriction of $\alpha$ to any component $U_X^{\mathcal{C}}$ factors through $T_nU_X^{\mathcal{C}}$ by hypothesis, hence $\alpha$ factors through $G=\bigoplus_{X\in \mathcal{C}}\bigoplus_{x\in FX} T_nU_X^{\mathcal{C}}$. Noting that $cr_{n+1}G=\bigoplus_{X\in \mathcal{C}}\bigoplus_{x\in FX} cr_{n+1}T_nU_X^{\mathcal{C}} = 0$, we conclude that $G$ is polynomial of degree $\le n$, and hence so is its quotient $F$.
	
	Now let us prove the equivalence of (2) and (3). Recall that the Yoneda isomorphism $\mathcal{Y}_X\colon FX \stackrel{\cong}{\longrightarrow} {\rm Hom}(U_X^{\mathcal{C}},F)$ is given by the composite map
		\[ \xymatrix{
		\mathcal{Y}_X(x)_Y\colon U_X^{\mathcal{C}}(Y) \ar[r]^{\widehat{F_{XY}}}& {\rm Hom}_{Ab}(FX,FY) \ar[r]^-{ev_x} & FY		}\]
		for $x\in FX$ and $Y\in \C$, where $ev_x$ is the evaluation in $x$. Now the map $ev\colon 
		 {\rm Hom}_{Ab}(FX,FY)$ $\to \prod_{x\in FX} FY$ such that $pr_x ev = ev_x$ is injective, whence assertion (2) is equivalent with saying that for all $X,Y\in \C$ and $x\in FX$ the map $\mathcal{Y}_X(x)_Y=ev_x\widehat{F_{XY}}$ factors through $\xymatrix{ (t_n)_Y\colon U_X^{\mathcal{C}}(Y)  \ar@{->>}[r] &{}}$ $T_nU_X^{\mathcal{C}}(Y)$. But this means that for all  $X\in \C$ and $x\in FX$ the natural transformation $\mathcal{Y}_X(x)$ factors through $t_n$, or equivalently, that every natural transformation from $U_X^{\mathcal{C}}$ to $F$ factors through $t_n$, as claimed.
		 
		 Now suppose that $\A$ is an arbitrary abelian category. In order to reduce $\A$ to a small subcategory, we consider the set of objects:
		 $$\mathcal{O} = \bigcup_{m\ge 1} \{F(X_1|\ldots|X_m)|X_1,\ldots,X_m\in C\}$$ and the full subcategory $\A'$  of $\A$ with  objects in $\mathcal{O}$. Consider the additive functor $G\colon \A \to \prod_{A\in \mathcal{O}} Ab$ such that $pr_AA \circ G = {\rm Hom}_{\mathcal{A}}(A,-)$. Then $G$ is faithful on $\A'$ since a non-trivial map $\alpha\colon A \to B$ in $\A'$ is detected by the functor 
		 ${\rm Hom}_{\mathcal{A}}(A,-)$. Thus by Lemma \ref{addopol} (2) $F$ is polynomial of degree $\le n$ iff $G\circ F$ is, which means that $F^A={\rm Hom}_{\mathcal{A}}(A,-)\circ F\colon \C \to Ab$ is polynomial of degree $\le n$ for all $A\in \mathcal{O}$ since cross-effects of $G\circ F$ are formed factorwise. 
		 
		 On the other hand, let $X,Y\in C$. Then the map $F_{X,Y}$ is polynomial of degree $\le n$ iff $\widehat{F_{X,Y}}({\rm Ker}(t_n)) =0$ iff $G_{FX,FY}\widehat{F_{X,Y}}({\rm Ker}(t_n)) =0$ iff for all $A\in \mathcal{O}$, $\widehat{F^A_{X,Y}}({\rm Ker}(t_n)) =0$ iff for all $A\in \mathcal{O}$, the map $F^A_{X,Y}$ is polynomial of degree $\le n$. But by the case $\A=Ab$, the functor $F^A$ is polynomial of degree $\le n$ iff the map  $F^A_{X,Y}$ is polynomial of degree $\le n$ for all $X,Y\in C$, whence (1) $\Leftrightarrow$ (2), as desired.
		\end{proof}

	Note that there exists a functor $$\underline{t_n} \colon \C \to T_n\bar{\Z}[\C]$$ such that $\underline{t_n}(X)=X$ and $\underline{t_n}(f)=t_nf$.  Theorem \ref{thm:polfunc-polmap} provides the following:

\begin{cor}\label{Poln(C)=Add(PnC)}
	
	For any abelian category $\A$ the functor $\underline{t_n} $ induces an isomorphism of categories
	\[ \underline{t_n} ^*\colon Add(T_n\bar{\Z}[\C] , \A) \stackrel{\sim}{\longrightarrow} Pol_{n}(\C,\A).\]
	
	\end{cor}

	\begin{rem}

	Remark that \textit{strict} polynomial functors of degree 
	$\le n$ are \textit{defined} in the same spirit, namely to be additive functors on a category $\Gamma_n\C$, formed by replacing morphism sets in an additive category $\C$ enriched in $\mathbb{K}$-modules by their images under the functor  
	$\Gamma_n$  (see \cite{pira-PS}).
	\end{rem}


 \subsection{Description of $T_n\bar{\Z}[\mathcal{P}\ti alg]$ in terms of the generalized Passi functor}
	
	We now show that for \textit{``good''} operads $\mathcal{P}$ and algebraic theories $\mathbb{T}$, the category $T_n\bar{\Z}[\mathcal{C}]$ for $\C$ being the category of finitely generated free $\mathcal{P}$-algebras or $\mathbb{T}$-models can be described in terms of the functors $P_n^{\mathcal{P}}$ or $P_n^{\mathbb{T}}$,  respectively. But this description is only available if  cartesian products of $\PP$-algebras can be decomposed in terms of   $\mathcal{P}$-algebra operations, as follows.
	
	Let $m\ge 1$ and $A_1,\ldots,A_m$ be $\mathcal{P}$-algebras. For $1\le k\le m$ let $i_k^{\times}\colon A_k\to A_1\times\ldots \times A_m$ be the canonical injection, i.e. $i_1^{\times}(a)=(a,1_{A_2},\ldots,1_{A_m})$ etc. Note that $i_k^{\times}$ is a $\mathcal{P}$-algebra map since $\mathcal{P}(0)=\{0_{\mathcal{P}}\}$.
	
\begin{lm}\label{decomp-prod}
 Suppose that $\mathcal{P}(2)$ contains an element $\theta_2$ for which $0_{\mathcal{P}}$ is a unit, i.e. $\gamma_{2;1,0}(\theta_2;1_{\mathcal{P}},0_{\mathcal{P}})=1_\PP=\gamma_{2;0,1}(\theta_2;0_{\mathcal{P}},1_{\mathcal{P}})$. Then there exists an operation $\theta_m\in \mathcal{P}(m)$ such that $(a_1,\ldots,a_m) = \theta_m(i_1^{\times}a_1,\ldots,i_m^{\times}a_m)$ for $(a_1,\ldots,a_m)\in A_1\times\ldots \times A_m$.
\end{lm}

\begin{proof}
By induction on $m$; the case $m=1$ is clear by taking $\theta_1=1_\PP$, and for $m=2$, we indeed have $(a_1,a_2) = (\theta_2(a_1,1_{A_1}), \theta_2(1_{A_2},a_2)) = \theta_2((a_1,1_{A_2}),(1_{A_1},a_2)) = \theta_2(i_1^{\times}a_1,i_2^{\times}a_2)$. Suppose that the assertion holds for $m-1$, $m\ge 3$. Writing $I_1^{\times} \colon A_1\times\ldots \times A_{m-1} \to A_1\times\ldots \times A_{m}$ for the canonical injection we have
\begin{eqnarray*}
(a_1,\ldots,a_m) &=& \theta_2(I_1^{\times}(a_1,\ldots,a_{m-1}),i_m^{\times}a_m)\\
&=& \theta_2(I_1^{\times}\theta_{m-1}(i_1^{\times}a_1,\ldots,i_{m-1}^{\times}a_{m-1}),i_m^{\times}a_m)\quad\mbox{by induction hypothesis}\\
&=& \theta_2(\theta_{m-1}I_1(i_1^{\times}a_1,\ldots,i_{m-1}^{\times}a_{m-1}),i_m^{\times}a_m)\quad\mbox{since $I_1^{\times}$ is a $\mathcal{P}$-algebra map}\\
&=& \theta_2(\theta_{m-1}(i_1^{\times}a_1,\ldots,i_{m-1}^{\times}a_{m-1}),i_m^{\times}a_m)\\
&=& \gamma_{2;m-1,1}(\theta_2,\theta_{m-1},1_{\mathcal{P}})(i_1^{\times}a_1,\ldots,i_m^{\times}a_m).
\end{eqnarray*}
Hence the assertion also holds for $m$ with $\theta_m = \gamma_{2;m-1,1}(\theta_2,\theta_{m-1},1_{\mathcal{P}})$.
\end{proof}

In the sequel, it will be convenient to replace the category  $Free(\PP)$ by $\mathcal{F}(\PP)$, see Remark
\ref{skelett-Free(P)}. 
The coproducts in $Free(\PP)$ and in $\mathbb{T} \ti Mod$ are given by: $\underline{k} \amalg \underline{l} =\underline{k+l}$. 

We also need the following elementary observation.

\begin{prop}\label{P-Alg(Free,A)} 
For $m\ge 1$ the functor ${\rm Hom}_{\mathcal{P}\ti Alg}(\mathcal{F}_{\mathcal{P}}(\ul{m}),-)$ naturally takes values in the category of $\mathcal{P}$-algebras such that for a $\PP$-algebra $A$ the natural bijection $ev\colon {\rm Hom}_{\mathcal{P}\ti Alg}(\mathcal{F}_{\mathcal{P}\ti Alg}(\ul{m}),A) \to A^m$, $ev(f)=(f(1),\ldots,f(m))$, is an isomorphism of $\mathcal{P}$-algebras.
\end{prop}

\begin{proof} 
Let $k\ge 0$, $\theta\in \mathcal{P}(k)$ and $f_1,\ldots,f_k\in {\rm Hom}_{\mathcal{P}\ti Alg}(\mathcal{F}_{\mathcal{P}\ti Alg}(\ul{m}),A)$.
As $\ul{m}$ is a basis of $\mathcal{F}_{\mathcal{P}\ti Alg}(\ul{m})$ we may define  $\theta(f_1,\ldots,f_k)$ to be the unique map such that for $i\in \ul{m}$, $\theta(f_1,\ldots,f_k)(i)=\theta(f_1(i),\ldots,f_k(i))$ .
\end{proof}

\begin{thm}\label{Pn(P-Alg)} Suppose that $\mathcal{P}(2)$ contains an element $\theta_2$ for which $0_{\mathcal{P}}$ is a unit, and the same for $\mathcal{P}_{\mathbb{T}}$. Then for $m\ge 1$ there are isomorphisms of functors
\[ T_nU^{\mathcal{P}\ti Alg}_{\mathcal{F}_{\mathcal{P}}(\ul{m})} \hspace{2mm}\cong\hspace{2mm} P_n^{\mathcal{P}}\circ {\rm Hom}_{\mathcal{P}\ti Alg}(\mathcal{F}_{\mathcal{P}}(\ul{m}),-)
\]
on $\mathcal{P}\ti Alg$ and
\[ T_nU^{\mathbb{T}\ti Mod}_{\mathcal{F}_{\mathbb{T}}(\ul{m})} \hspace{2mm}\cong\hspace{2mm} 
P_n^{\mathbb{T}} \circ {\rm Hom}_{\mathbb{T}\ti Mod}(\mathcal{F}_{\mathbb{T}}(\ul{m}),-)
\]
on $\mathbb{T}\ti Mod$.
\end{thm}

\begin{proof} 
We have
\begin{eqnarray*}
T_nU^{\mathcal{P}}_{\mathcal{F}_{\mathcal{P}\ti Alg}(\ul{m})}
&=& T_n\left( (\bar{\Z}[-] \circ V_{\mathcal{P}\ti Alg}) \circ
{\rm Hom}_{\mathcal{P}\ti Alg}(\mathcal{F}_{\mathcal{P}}(\ul{m}),-) \right) \quad\mbox{by definition of $U$}\\
&=& T_n\left( T_n(\bar{\Z}[-]\circ V_{\mathcal{P}\ti Alg}) \circ {\rm Hom}_{\mathcal{P}\ti Alg}(\mathcal{F}_{\mathcal{P}}(\ul{m}),-) \right) \quad\mbox{by Lemma \ref{Tn(GoF)}}\\
&=& T_n\left( P_n^{\mathcal{P}} \circ {\rm Hom}_{\mathcal{P}\ti Alg}(\mathcal{F}_{\mathcal{P}}(\ul{m}),-) \right) \quad\mbox{by Theorem \ref{Pn=Tn}.}
\end{eqnarray*}
So it suffices to show that the composite functor $P_n ^{\mathcal{P}} \circ {\rm Hom}_{\mathcal{P}\ti Alg}(\mathcal{F}_{\mathcal{P}}(\ul{m}),-) $ is polynomial of degree $\le n$. 
By the same argument as in the proof of Proposition \ref{Pnpoldegn} it suffices to prove this in the case $n=1$. 
First observe that by combining Lemma \ref{decomp-prod} with Lemma \ref{modI2} we have $P_1^{\mathcal{P}} (A^m) = \sum_{k=1}^mP_1^{\mathcal{P}} (i_k^{\times})
P_1^{\mathcal{P}} (A)$. Moreover, for $A_1,A_2\in \mathcal{P}\ti Alg$ the following diagram commutes for $l=1,2$:
\[ \xymatrix{
A_l \ar[r]^-{i_l} \ar[d]^{i_k^{\times}} & A_1\vee A_2 \ar[d]^{i_k^{\times}} \\
(A_l)^m\ar[r]^-{(i_l)^m} & (A_1\vee A_2)^m
}\]
Combining these observations with Proposition \ref{P-Alg(Free,A)} and the fact that $P_1^{\mathcal{P}} $ is linear by Lemma \ref{P1linear}, we obtain
\begin{eqnarray*}
P_1^{\mathcal{P}} ((A\vee B)^m) &=& \sum_{k=1}^m P_1^{\mathcal{P}} (i_k^{\times})P_1^{\mathcal{P}} (A_1\vee A_2)\\
&=& \sum_{k=1}^m P_1^{\mathcal{P}} (i_k^{\times})(P_1^{\mathcal{P}} (i_1)P_1^{\mathcal{P}} (A_1) +P_1^{\mathcal{P}} (i_2)P_1^{\mathcal{P}} (A_2))\quad\mbox{since $P_1^{\mathcal{P}} $ is linear}\\
&=& P_1^{\mathcal{P}} ((i_1)^m) \left( \sum_{k=1}^m P_1^{\mathcal{P}} (i_k^{\times})P_1^{\mathcal{P}} (A_1) \right) +P_1^{\mathcal{P}} ((i_2)^m)\left( \sum_{k=1}^m P_1^{\mathcal{P}} (i_k^{\times})P_1^{\mathcal{P}} (A_2)\right)\\
&=& P_1^{\mathcal{P}} ((i_1)^m) P_1^{\mathcal{P}} ((A_1)^m)  +P_1^{\mathcal{P}} ((i_2)^m)P_1^{\mathcal{P}} ((A_1)^m).
\end{eqnarray*}
Thus the  section $(P_1^{\mathcal{P}} ((i_1)^m) , P_1^{\mathcal{P}} ((i_2)^m) )\colon P_1^{\mathcal{P}} ((A_1)^m) \oplus P_1^{\mathcal{P}} ((A_2)^m) \to P_1^{\mathcal{P}} ((A_1\vee A_2)^m)$ of $\hat{r}^{P_1^{\mathcal{P}} }$
is surjective, hence an isomorphism. Consequently the functor $A\mapsto P_1^{\mathcal{P}} (A^m)$ is linear, as desired. The same proof works in the case of $\mathbb{T}$-algebras.
\end{proof}
In order to compute $T_n\bar{\Z}[\mathcal{F}(\mathcal{P})]$ and $T_n\bar{\Z}[\mathcal{F}(\mathbb{T})]$ we need the following definition.
\begin{defi}
The category $P_n^{\mathcal{P}} \mathcal{F}(\mathcal{P})$ is defined as follows: the objects are the sets $\ul{m}$, $m\ge 0$, the morphism sets are ${\rm Hom}_{P^{\mathcal{P}}_n\mathcal{F}(\mathcal{P})}(\ul{k},\ul{l}) = P_n^{\mathcal{P}}\left( 
{\rm Hom}_{\mathcal{F}(\mathcal{P})}(\mathcal{F}_{\mathcal{P}}(\ul{k}),\mathcal{F}_{\mathcal{P}}(\ul{l})
\right))$, and the composition satisfies $p_n(g)\circ p_n(f) = p_n(g\circ f)$. The latter identity means that there is a functor $\ul{p}_n^{\PP}\colon \mathcal{F}(\mathcal{P}) \to P_n^{\mathcal{P}}\mathcal{F}(\mathcal{P})$ which is the identity on objects and sends a morphism $f$ in $\mathcal{F}(\mathcal{P})$ to $p_n(f)$.

The category $P_n^{\mathbb{T}}\mathcal{F}(\mathbb{T})$ is defined in an analogous way.

\end{defi}

\begin{cor}\label{Pn(Free(P))} For $\mathcal{P}$  as in Theorem \ref{Pn(P-Alg)} the category $T_n\bar{\Z}[\mathcal{F}(\mathcal{P})]$ is isomorphic to the category $P_n^{\mathcal{P}}\mathcal{F}(\mathcal{P})$. 

Similarly, for $\mathbb{T}$  as in Theorem \ref{Pn(P-Alg)} the category $T_n\bar{\Z}[\mathcal{F}(\mathbb{T})]$ is isomorphic to the category $P_n^{\mathbb{T}}\mathcal{F}(\mathbb{T})$.
\hfill$\Box$
\end{cor}

Finally we wish to compare the categories $P_n^{\mathcal{A}s}mon$ and $P_n^{\mathbb{G}r}gr$.
For this we still need the well known fact that the functor ${\bf g}$ preserves finite products. To make this precise let $M_1,\ldots,M_m$ be monoids. As ${\bf g}$ is left adjoint to $U$ the monoid morphism $\eta_{\ul{M}} := \eta_{M_1} \times \ldots \times \eta_{M_m}\colon M_1\times \ldots \times  M_m \to {\bf g}(M_1)\times \ldots \times  {\bf g}(M_m) $ induces a unique group homomorphism $\widehat{\eta_{\ul{M}} } \colon {\bf g}(M_1\times \ldots \times  M_m ) \to {\bf g}(M_1)\times \ldots \times  {\bf g}(M_m)$. Then:

\begin{prop}\label{g-pres-prods} The map $\widehat{\eta_{\ul{M}} }$ is an isomorphism.
\end{prop}

\begin{thm}\label{Pnmon=Pngr} The functor ${\bf g}\colon mon \to gr$ induces an isomorphism of preadditive categories $P_n{\bf g} \colon P_n^{\mathcal{A}s}mon \to P_n^{\mathbb{G}r}gr$ such that $P_n{\bf g}\circ \ul{p}{}_n^{\mathcal{A}s} =\ul{p}{}_n^{\mathbb{G}r}\circ {\bf g}$.

\end{thm}

\begin{proof}
First of all, we may replace both of the functors  $P_n^{\mathcal{A}s}$ and $P_n^{\mathbb{G}r}$ in the definition of the categories $P_n^{\mathcal{A}s}mon$ and $ P_n^{\mathbb{G}r}gr$ by Passi's functor $P_n$, according to Example \ref{exple:PnAs} and Proposition \ref{calculPnGr}; we denote the resulting categories by $P_nmon$ and $P_ngr$. In the sequel we freely use the basic fact that by definition $P_n(G) = P_n(U(G))$ for any group $G$.

Now let $k,l\ge 1$ and consider the following diagram of plain arrows where the map $\widehat{\eta^k}$ renders the left-hand triangle commutative and is an isomorphism by Proposition \ref{g-pres-prods}.
\[\xymatrix{
& \mathcal{F}_{Mon}(\ul{l})^k \ar[dl]^-{\eta} \ar[dd]^{\eta^k} & mon(\ul{k},\ul{l})
\ar[l]^-{\cong}_{ev}
\ar[dd]^{{\bf g}_{\ul{k},\ul{l}}} \ar[r]^-{p_n} & P_n(mon(\ul{k},\ul{l})) \ar@{.>}[dd]^{P_n({\bf g}_{\ul{k},\ul{l}})} \\
{\bf g}(\mathcal{F}_{Mon}(\ul{l})^k) \ar[dr]^{\widehat{\eta^k}}& & & \\
 &{\bf g}(\mathcal{F}_{Mon}(\ul{l}))^k &gr(\ul{k},\ul{l}) \ar[l]^-{\cong}_-{ev}\ar[r]^-{p_n} &P_n(gr(\ul{k},\ul{l}))
}\]
The left-hand square commutes  since for $f\in mon(\ul{k},\ul{l})$ and $i\in \ul{k}$,
$pr_i\circ \eta^k\circ ev(f) = \eta f(i)$ while \begin{eqnarray*}
pr_i\circ ev\circ {\bf g}_{\ul{k},\ul{l}}(f) &=& {\bf g}(f)(i) \\
&=&  {\bf g}(f)\eta(i)\\ &=& \eta(f(i)) \quad\mbox{by naturality of $\eta$.}
\end{eqnarray*}
It follows that the map ${\bf g}_{\ul{k},\ul{l}}$ is a morphism of monoids since $ev$ and $\eta^k$ are. Applying the functor $P_n$ we deduce  that the dotted arrow $P_n({\bf g}_{\ul{k},\ul{l}})$ exists, is a homomorphism of abelian groups and renders the right-hand square commutative. We claim that $P_n{\bf g}\colon P_nmon\to P_ngr$ defined by $P_n{\bf g}(\ul{m}) =\ul{m}$ and $(P_n{\bf g})_{\ul{k},\ul{l}} = P_n({\bf g}_{\ul{k},\ul{l}})$ is an isomorphism of additive categories.
To show that $P_n{\bf g}$ is a functor let $\xymatrix{\ul{k}\ar[r]^-f & \ul{l}\ar[r]^-g& \ul{m}}$. Then 
\begin{eqnarray*}
P_n({\bf g}_{\ul{k},\ul{m}})(p_n(g) \circ p_n(f)) &=& P_n({\bf g}_{\ul{k},\ul{m}})(p_n(gf))\quad\mbox{by definition of $P_nmon$}\\
&=&p_n({\bf g}_{\ul{k},\ul{m}}(gf)) \quad\mbox{by definition of $
P_n({\bf g}_{\ul{k},\ul{m}})$}\\
&=&p_n({\bf g}_{\ul{l},\ul{m}}(g) \circ {\bf g}_{\ul{k},\ul{l}}(f))\quad\mbox{since ${\bf g}$ is a functor}\\
&=&p_n({\bf g}_{\ul{l},\ul{m}}(g)) \circ p_n({\bf g}_{\ul{k},\ul{l}}(f))\quad\mbox{by definition of $P_ngr$}\\
&=&P_n({\bf g}_{\ul{l},\ul{m}})(p_n(g)) \circ P_n({\bf g}_{\ul{k},\ul{l}})(p_n(f))
\end{eqnarray*}
It follows that $P_n{\bf g}$ is a functor since the composition in $P_nmon$ and $P_ngr$ is bilinear and $P_n(mon(\ul{k},\ul{l}))$ is generated by the elements $p_n(f)$, $f\in mon(\ul{k},\ul{l})$.

It remains to prove that $P_n({\bf g}_{\ul{k},\ul{l}})$ is an isomorphism for all $k,l\ge 1$. To see this apply the functor $P_n$ to the left-hand triangle and square of the above diagram: as $P_n(\eta)$ is an isomorphism by Lemma \ref{iso-ab-gr} so is $P_n(\eta^k)$ and hence $P_n({\bf g}_{\ul{k},\ul{l}})$.
\end{proof}

		     We now can state the main application of the results of this section:

\begin{cor} \label{pol-gr-mon}
There is an isomorphism of categories $Pol_n(gr,Ab) \,\cong\,Pol_n(mon,Ab)$.
\end{cor}	
	
	\begin{proof}
	
	We have:
	$$Pol_n(gr,Ab) \,\cong\, Add(T_n \overline{\Z}[gr],Ab)\,\cong\,Add(P_n^{\mathbb{G}r}gr,Ab)\,\cong\,Add(P_n^{\mathcal{A}s}mon,Ab)$$
	$$\cong\,Add(T_n \overline{\Z}[mon],Ab) \,\cong\, Pol_n(mon,Ab)$$
	where the first and fifth equivalences follow from Corollary \ref{Poln(C)=Add(PnC)}, the second and fourth equivalences follow from Corollary \ref{Pn(Free(P))} and the third equivalence is given by Theorem \ref{Pnmon=Pngr}.
	 \end{proof}


\section{Application: polynomial functors from abelian groups and from groups} \label{section-5.5}
In this section we apply Theorem \ref{thm-pal} and results of section \ref{section5} to the cases $\PP=\mathcal{C} om$ and $\PP= \mathcal{A} s$. In the first case we recover the results of \cite{BDFP}.

\subsection{Polynomial functors from abelian groups}
Applying Theorem \ref{thm-pal} for $\PP=\mathcal{C} om$ we obtain a natural equivalence of categories
$$Func(Free(\mathcal{C} om), Ab) \simeq PMack(\Omega(\mathcal{C} om), Ab).$$
Since $\Omega(\mathcal{C} om)=\Omega$ by Example \ref{Omega} and $Free(\mathcal{C} om)=common$ by \ref{ex-alg}, we obtain:
\begin{equation} \label{eq-ab}
Func(common, Ab) \simeq PMack(\Omega, Ab)
\end{equation}
which is the Theorem $0.2$ of \cite{BDFP}. Furthermore, we have:
\begin{thm}[\cite{BDFP} section $8$] \label{BDFP-8}
Let $n \in \N$, there is an equivalence of categories
$$Pol_n(common,Ab) \simeq Pol_n(ab,Ab).$$
\end{thm}
This theorem is generalized to the non-commutative setting in Corollary \ref{pol-gr-mon}. In the sequel, we briefly recall  the proof of Theorem \ref{BDFP-8} sketched in \cite{BDFP} in order to explain clearly why it can't be adapted to the non-commutative setting.

The proof sketched in \cite{BDFP} consists in constructing a functor:
	$$\tilde{-}: Pol_n(common, Ab) \to Pol_n(ab,Ab)$$
	which induces an equivalence of categories. A crucial tool in defining $\tilde{-}$ is Lemma $8.1$ in \cite{BDFP} asserting the equivalence between polynomial maps of degree $\leq n$ from a commutative monoid $M$ to an abelian group $A$ and polynomial maps of degree $\leq n$ from the commutative group $gr(M)$ to $A$ where $gr$ is the group completion. (We generalize this result to the non-commutative setting in Proposition \ref{poly-map}).
	
	For $F \in Pol_n(common, Ab)	$ the functor $\tilde{F}: ab \to Ab$ is obtained in the following way:
	\begin{itemize}
\item For $\underline{n} \in ab$, $\tilde{F}(\underline{n})=F(\underline{n})$;
\item for the morphisms: the map $\psi: Hom_{ab}(\underline{l}, \underline{m}) \to Hom_{Ab}(F(\underline{l}), F(\underline{m}))$ is the unique polynomial map of degree $n$ corresponding, by Lemma $8.1$ in \cite{BDFP} to the polynomial map of degree $n$: $Hom_{common}(\underline{l}, \underline{m}) \to Hom_{Ab}(F(\underline{l}), F(\underline{m}))$ given by the polynomiality of the functor $F$;
\item to prove that $\tilde{F}$ is a functor we have to prove the commutativity of the following diagram:

$$\xymatrix{
Hom_{ab}(\underline{l}, \underline{m}) \times Hom_{ab}(\underline{m}, \underline{k}) \ar[d]_{\phi} \ar[r]^-{\psi \times \psi} & Hom_{Ab}(F(\underline{l}), F(\underline{m})) \times Hom_{Ab}(F(\underline{m}), F(\underline{k})) \ar[d]^{\phi}
 \\
Hom_{ab}(\underline{l}, \underline{k}) \ar[r]_-{\psi}& Hom_{Ab}(F(\underline{l}), F(\underline{k}))
}$$
In this diagram the composition map $\phi$ is bilinear since we are in the commutative case.
By composition, $\psi \circ \phi$ is bipolynomial of degree $n$ in each variable. This map corresponds to the unique bipolynomial map of degree $n$ in each variable associated  to the map sending $(f,g)\in Hom_{common}(\underline{l}, \underline{m}) \times Hom_{common}(\underline{m}, \underline{k})$ to
$F(g \circ f)$, by Lemma $8.1$ in \cite{BDFP}. The other composition is also bipolynomial of degree $n$ in each variable and corresponds, by Lemma $8.1$ in \cite{BDFP}  to the unique bipolynomial map of degree $n$ in each variable associated to the map sending $(f,g)$ to $F(g)\circ F( f)$. Since $F$ is a functor we deduce that $\tilde{F}$ is a functor.
\end{itemize}
      In the case of non-commutative monoids and non-commutative groups, the map corresponding to $\phi$ in the previous diagram is no more bilinear (in fact we have $f(g+g')=fg+fg'$ but in general $(f+f') g\neq fg+f'g$). Consequently the composition $\psi \phi$, a priori, has no reason to be bipolynomial of degree $n$ in each variable and we can't prove as above that a similar extension to $gr$ of a functor from  $mon$ indeed defines a functor from $gr$ to $Ab$.

We deduce from (\ref{eq-ab}) and Theorem \ref{BDFP-8}:

\begin{thm}[\cite{BDFP} Theorem $6.1$]
Let $n \in \N$, there is an equivalence of categories
$$Pol_n(ab,Ab) \simeq (PMack(\Omega, Ab))_{\leq n}$$
where $(PMack(\Omega, Ab))_{\leq n}$ is the full subcategory of $PMack(\Omega, Ab)$ having as objects the functors which vanish on the sets $X$ such that $\mid X \mid >n$ and on $\underline{0}$.
\end{thm}

\subsection{Polynomial functors from groups}
Applying Theorem \ref{thm-pal} for $\PP=\mathcal{A} s$ we obtain a natural equivalence of categories
$$Func(Free(\mathcal{A} s), Ab) \simeq PMack(\Omega(\mathcal{A} s), Ab).$$
Since  $Free(\mathcal{A} s)=mon$ by \ref{ex-alg}, we obtain:
$$Func(mon, Ab) \simeq PMack(\Omega(\mathcal{A} s), Ab).$$

By Corollary \ref{pol-gr-mon} we obtain.
\begin{thm}
Let $n \in \N$, there is an equivalence of categories
$$Pol_n(gr, Ab) \simeq (PMack(\Omega(\mathcal{A} s), Ab))_{\leq n}$$
where $(PMack(\Omega(\mathcal{A} s), Ab))_{\leq n}$ is the full subcategory of $PMack(\Omega(\mathcal{A} s), Ab)$ having as objects the functors which vanish on  $\underline{0}$ and on the sets $\underline{m}$ such that $m >n$.
\end{thm}

\section{Presentation of polynomial functors} \label{section6}
The aim of this section is to give a presentation of polynomial functors from $Free(\PP)$ to $Ab$ in terms of minimal data.

First we give the generators of the category $\Omega(\PP)_2$.
\subsection{Generators of $\Omega(\PP)_2$}
\subsubsection{Horizontal maps}
Since $(\Omega(\PP)_2)^h=\Omega$, the horizontal generators of $\Omega(\PP)_2$ are given in \cite{BDFP}.

For all $i\in \N$ such that $1 \leq i <n$ we consider:
\begin{itemize}
\item $s_i^n: \underline{n} \to \underline{n-1}$ the unique surjection preserving the natural order and such that:
$$s^n_i(i)=s_i^n(i+1)=i;$$
\item $\tau^n_i: \underline{n} \to \underline{n}$ the transposition of $\mathfrak{S}_n$ which exchanges $i$ and $i+1$.
\end{itemize}

\begin{lm}[Lemme 9.5 in \cite{BDFP}] \label{relations-hor}
The category $\Omega$ is generated by $\{ s_i^n, \tau^n_i \ \mid \ 1 \leq i <n \}$ submitted to the relations of the symmetric group $\mathfrak{S}_n$ and to the following conditions:
\begin{enumerate}
\item
$s_i^{n-1} s_j^n= s^{n-1}_{j-1} s^n_i \mathrm{\quad  for\ } i<j$
\item 
$s_i^n \tau_i^n=s_i^n$
\item
$
\tau^{n-1}_k s^n_j= 
\left \lbrace
\begin{array}{lll}
s^n_j \tau^n_k \mathrm{\quad  for\ } k<j-1\\
s^n_{j-1} \tau^n_j  \tau^n_{j-1} \mathrm{\quad  for\ } k=j-1\\
s^n_{j+1} \tau^n_j  \tau^n_{j+1} \mathrm{\quad  for\ } k=j\\
s^n_j \tau^n_{k+1} \mathrm{\quad  for\ } k>j.
\end{array}
\right.
$
\end{enumerate}

\end{lm}
 \subsubsection{Vertical maps}
We have $(\Omega(\PP)_2)^v=\Omega(\PP)$.

For all $k\in \N$ such that $1 \leq k <n$  and for all $i \in \N$ such that $1 \leq i \leq n-k$ we consider:
\begin{itemize}
\item $s_i^{n, n-k}: \underline{n} \to \underline{n-k}$ the unique surjection preserving the natural order and such that:
$$s^{n, n-k}_i(i)=s_i^{n, n-k}(i+1)=\ldots= s_i^{n, n-k}(i+k)=i$$
$\forall \omega \in \PP(k+1)$ we obtain a morphism in $\Omega(\PP)$:
$$(s^{n, n-k}_i, (1_{\PP}^{\times (i-1)},\omega, 1_{\PP}^{\times (n-k-i)})): \underline{n} \to \underline{n-k};$$ 
in the sequel, this morphism will be denoted by $(s^{n, n-k}_i,\omega)$ for simplicity;
\item $(\tau^n_i, 1_{\PP}^{\times n} \in \PP(1)^{\times n}): \underline{n} \to \underline{n}$ where $\tau^n_i$ is  the transposition of $\mathfrak{S}_n$ which exchanges $i$ and $i+1$;
\item
$(Id_n, (\alpha_1, \ldots, \alpha_n)): \underline{n} \to \underline{n}$ where $\forall i \in \{1, \ldots, n\}$, $\alpha_i \in \PP(1)$.
\end{itemize}

\begin{lm} \label{relations-vert}
The category $\Omega(\PP)$ is generated by the morphisms $(s^{n, n-k}_i, \omega \in \PP(k+1))$, $(\tau^n_k, 1_{\PP}^{\times n} \in \PP(1)^{\times n})$ and  $(Id_n, (\alpha_1, \ldots, \alpha_n) \in \PP(1)^{\times n})$ submitted to the relations of the symmetric group $\mathfrak{S}_n$ and to the following conditions:
\begin{enumerate}
\item \label{rel-vert-1}
$(Id_n, (\omega_1, \ldots, \omega_n))(Id_n, (\alpha_1, \ldots, \alpha_n))=(Id_n, (\gamma(\omega_1, \alpha_1), \ldots, \gamma(\omega_n, \alpha_n)))$
\item \label{rel-vert-2}
$(Id_n, (\omega_1, \ldots, \omega_n))(\tau^n_k, 1_{\PP}^{\times n})=
(\tau^n_k, 1_{\PP}^{\times n})(Id_n, (\omega_1, \ldots, \omega_{k-1}, \omega_{k+1}, \omega_k, \omega_{k+2}, \ldots, \omega_n));
$
\item \label{rel-vert-3}
$(Id_{n-k}, \omega_1, \ldots, \omega_{n-k})(s_{i}^{n, n-k},\omega)=(s_{i}^{n, n-k},\gamma(\omega_i, \omega))(Id_n, \omega_1, \ldots, \omega_{i-1}, 1_{\PP}^{\times (k+1)}, \omega_{i+1}, \ldots, \omega_{n-k})$
\item \label{rel-vert-4}
$(s_{i}^{n, n-k},\omega)(Id_{n}, \omega_1, \ldots, \omega_{n})$\\
$=(Id_{n-k}, \omega_1, \ldots, \omega_{i-1}, 1_{\PP}, \omega_{i+k+1}, \ldots,\omega_{n})(s_{i}^{n, n-k},\gamma(\omega; \omega_i, \omega_{i+1}, \ldots, \omega_{i+k}))$
\item \label{rel-vert-5}
$(s_i^{n-k,n-k-p}, \omega \in \PP(p+1))(s_j^{n, n-k}, \alpha \in \PP(k+1))$
$$=\left \lbrace
\begin{array}{lll}
(s_i^{n, n-k-p}, \gamma(\omega;1_{\PP}^{\times (j-i)},  \alpha, 1_{\PP}^{\times(p-j+i)}) \in \PP(k+p+1))  \mathrm{\quad for\ } i \leq j\leq i+p\\
 (s_j^{n-p, n-k-p}, \alpha )(s_{i+k}^{n,n-p}, \omega)  \mathrm{\quad for\ } j<i\\
 (s_{j-p}^{n-p, n-k-p}, \alpha)(s_{i}^{n,n-p}, \omega )  \mathrm{\quad for\ } j>i+p\\
\end{array}
\right.
$$
\item \label{rel-vert-6}
$(\tau^{n-p}_k, 1_{\PP}^{\times (n-p)} \in \PP(1)^{\times n-p})(s_j^{n, n-p}, \alpha \in \PP(p+1))$
$$=\left \lbrace
\begin{array}{lll}
(s_j^{n, n-p}, \alpha)(\tau^n_k, 1_{\PP}^{\times n}) \mathrm{\quad for\ } k<j-1\\
(s_{j-1}^{n, n-p}, \alpha)(\tau^n_{j+p-1}, 1_{\PP}^{\times n})(\tau^n_{j+p-2}, 1_{\PP}^{\times n})  \ldots    (\tau^n_j, 1_{\PP}^{\times n})(\tau^n_{j-1}, 1_{\PP}^{\times n}) \mathrm{\quad for\ } k=j-1\\
(s_{j+1}^{n, n-p}, \alpha)(\tau^n_j, 1_{\PP}^{\times n}) (\tau^n_{j+1}, 1_{\PP}^{\times n}) \ldots (\tau^n_{j+p-1}, 1_{\PP}^{\times n})(\tau^n_{j+p}, 1_{\PP}^{\times n}) \mathrm{\quad for\ } k=j\\
(s_j^{n, n-p}, \alpha)(\tau^n_{k+p}, 1_{\PP}^{\times n}) \mathrm{\quad for\ } k>j.
\end{array}
\right.
$$
\item \label{rel-vert-7}
$(s_j^{n, n-p}, \alpha \in \PP(p+1))(\tau^n_k, 1_{\PP}^{\times n} \in \PP(1)^{\times n})$
$$=\left \lbrace
\begin{array}{lll}
(s_{j}^{n, n-p}, \alpha. \tau_{k-j+1})\mathrm{\quad for\ }  j \leq k \leq j+p-1\\
(\tau^n_{k-p}, 1_{\PP}^{\times n}) (s_j^{n, n-p}, \alpha)\mathrm{\quad for\ } k>j+p \mathrm{\quad or\ }  k<j-1.
\end{array}
\right.
$$
Remark: For $k=j+p$ and $k=j-1$ there are no relation.

\end{enumerate}
\end{lm}

 \subsubsection{The role of the generators}
 The following proposition allows us to reduce the number of conditions to be verified to prove that a Janus functor is a pseudo-Mackey functor.
 \begin{prop} \label{pour-gene},
 Let $M$ be a Janus functor; $M$ is a pseudo-Mackey functor iff conditions $(1)$ and $(2)$ in Definition \ref{PMack} are satisfied for horizontal (resp. vertical) generators $h$ (resp. $(v, \omega^v)$) of $\Omega(\PP)_2$.
 \end{prop} 
 The proof of this proposition relies on the following lemma which corresponds to Lemma $9.2$ in \cite{BDFP}.
 
 \begin{lm}
 Let $M=(M_*, M^*): \Omega(\PP)_2 \to Ab$ be a Janus functor. For two squares in $\Omega(\PP)_2$
 $$\vcenter{
\xymatrix{ A \ar@{-->}[d]_ -{(\phi_3, \omega^{\phi_3})} \ar[r]^-{g_1} &B \ar@{-->}[d]^-{(\phi_2, \omega^{\phi_2})} \ar[r]^-{g_2} & C \ar@{-->}[d]^-{(\phi_1, \omega^{\phi_1})} \\
D \ar[r]_-{f_1} &E \ar[r]_-{f_2}& F
} }$$
if
$$ M^*(\phi_1, \omega^{\phi_1}) M_*(f_2)=\underset{B' \in Adm(B)}{\sum} M_*({g_2}_{B'}) M^*((\phi_2, \omega^{\phi_2})_{B'})$$
and
$$M^*(\phi_2, \omega^{\phi_2}) M_*(f_1)=\underset{A' \in Adm(A)}{\sum} M_*({g_1}_{A'}) M^*((\phi_3, \omega^{\phi_3})_{A'}) $$
then
$$M^*(\phi_1, \omega^{\phi_1}) M_*(f_2 \circ f_1)=\underset{A' \in Adm(A)}{\sum} M_*({(g_2 \circ g_1)}_{A'}) M^*((\phi_3, \omega^{\phi_3})_{A'}). $$
For two squares  in $\Omega(\PP)_2$
 $$\vcenter{
\xymatrix{ A \ar@{-->}[d]_ -{(\psi_1, \omega^{\psi_1})} \ar[r]^-{g} &B \ar@{-->}[d]^-{(\phi_1, \omega^{\phi_1})} \\
D \ar[r]_-{f} \ar@{-->}[d]_ -{(\psi_2, \omega^{\psi_2})}&E  \ar@{-->}[d]^ -{(\phi_2, \omega^{\phi_2})}\\
G \ar[r]_-{e} &H
} }$$
if
$$ M^*(\phi_1, \omega^{\phi_1}) M_*(f)=\underset{A' \in Adm(A)}{\sum} M_*({g}_{A'}) M^*((\psi_1, \omega^{\psi_1})_{A'})$$
and
$$M^*(\phi_2, \omega^{\phi_2}) M_*(e)=\underset{D' \in Adm(D)}{\sum} M_*({f}_{D'}) M^*((\psi_2, \omega^{\psi_2})_{D'}) $$
then
$$M^*((\phi_2, \omega^{\phi_2})(\phi_1, \omega^{\phi_1})) M_*(e)=\underset{A' \in Adm(A)}{\sum} M_*({g}_{A'}) M^*(((\psi_2, \omega^{\psi_2})(\psi_1, \omega^{\psi_1}))_{A'}). $$

 \end{lm} 
 \begin{proof}\cite{BDFP}
 We have:
 
$$M^*(\phi_1, \omega^{\phi_1}) M_*(f_2 \circ f_1) = \big(M^*(\phi_1, \omega^{\phi_1}) M_*(f_2)\big) M_*( f_1) 
= \underset{B' \in Adm(B)}{\sum}M_*({g_2}_{B'}) M^*((\phi_2, \omega^{\phi_2})_{B'})M_*( f_1).$$
Let 
$$\vcenter{
\xymatrix{ A' \ar@{-->}[d]_ -{(\phi_3, \omega^{\phi_3})_{A'}} \ar[r]^-{{g_1}_{A'}} &B' \ar@{-->}[d]^-{(\phi_2, \omega^{\phi_2})_{B'}} \\
D \ar[r]_-{f_1} &E } }$$
be the unique square of $\Omega(\PP)_2$ associated to $(f_1, (\phi_2, \omega^{\phi_2})_{B'})$
we have:
\begin{eqnarray*}
 \underset{B' \in Adm(B)}{\sum}M_*({g_2}_{B'}) M^*((\phi_2, \omega^{\phi_2})_{B'})M_*( f_1)&=& \underset{B' \in Adm(B)}{\sum}M_*({g_2}_{B'})  \underset{A'' \in Adm(A')}{\sum} M_*({g_1}_{A''}) M^*(\phi_3, \omega^{\phi_3})_{A''}\\
 &=& \underset{B' \in Adm(B)}{\sum}\underset{A'' \in Adm(A')}{\sum} M_*({g_2}_{B'} {g_1}_{A''}) M^*((\phi_3, \omega^{\phi_3})_{A''})\\
 &=&\underset{B' \in Adm(B)}{\sum}\underset{A'' \in Adm(A')}{\sum} M_*((g_2 \circ {g_1})_{A''}) M^*((\phi_3, \omega^{\phi_3})_{A''})\\
 &=&\underset{A'' \in Adm(A)}{\sum} M_*((g_2 \circ {g_1})_{A''}) M^*((\phi_3, \omega^{\phi_3})_{A''})\\
\end{eqnarray*}
The proof of the second part of the lemma is similar.
 \end{proof}
 
  \begin{proof}[Proof of Proposition \ref{pour-gene}]
The proposition is a direct consequence of the previous lemma applied to a decomposition of horizontal and vertical maps into a product of the generators.
 \end{proof}

\subsection{Presentation of polynomial functors}
Let $(M_*, M^*)$ be a Janus functor $\Omega(\PP)_2 \to Ab$. We put:
$$I_n^{\alpha_1, \ldots, \alpha_n}:=M^*(Id_n, (\alpha_1, \ldots, \alpha_n)): M(\underline{n}) \to M(\underline{n})$$
$$T_k^n:=M_*(\tau_k^n): M(\underline{n}) \to M(\underline{n})$$
$$P_i^n:= M_*(s_i^n): M(\underline{n}) \to M(\underline{n-1})$$
$$H_i^{(n,n-p)\  \omega}:=M^*(s_i^{n, n-p}, \omega): M(\underline{n-p}) \to M(\underline{n})\quad \mathrm{for}\ \omega \in \PP(p+1).$$

\begin{prop} \label{presentation}
Let $M=(M_*, M^*)$ be a Janus functor $\Omega(\PP)_2 \to Ab$, $M$ is a pseudo-Mackey functor iff the following relations are satisfied:
\begin{align} \label{presentation-1}
T_i^n&=(M^*(\tau_i^n, 1_{\PP}^{\times n}))^{-1} \mathrm{\quad  for\ } 1 \leq i \leq n-1;
\end{align}
\begin{align}  \label{presentation-2}
I_{n-1}^{\alpha_1,\ldots \alpha_{n-1}}P_i^n&=P_i^n I_n^{\alpha_1, \ldots, \alpha_{i-1}, \alpha_i , \alpha_i, \ldots, \alpha_n}\mathrm{\quad  for\ } 1 \leq i \leq n;
\end{align}

\begin{align}
H_j^{(n+k-1,n-1)\  \alpha} P_i^{n}&=P_i^{n+k} H_{j+1}^{(n+k,n)\  \alpha}
\end{align}
\begin{equation*}
 \mathrm{\quad  for\ } 1 \leq i<j<n, \ 1 \leq k < n \mathrm{\quad  and\ } \alpha \in \PP(k+1);
\end{equation*}

\begin{align}
H_j^{(n+k-1,n-1)\  \alpha} P_i^{n}&=P_{k+i}^{n+k} H_{j}^{(n+k,n)\  \alpha}
\end{align}
\begin{equation*}
 \mathrm{\quad  for\ } 1 \leq j<i<n, \  1 \leq k < n \mathrm{\quad  and\ } \alpha \in \PP(k+1);
\end{equation*}
\begin{equation}
H_i^{(n+k-1,n-1)\  \alpha} P_i^{n}= P_{i+k}^{n+k}P_{i+k-1}^{n+k+1} \ldots P_i^{n+2k} M_*(\sigma) H_i^{(n+2k, n+k)\ \alpha}H_{i+1}^{(n+k,n)\ \alpha}
\end{equation}
\begin{equation*}
+\sum_{\beta=1}^{k+1} \sum_{i\leq u_1<\ldots<u_\beta \leq i+k} P_{i+k}^{n+k} \ldots \hat{P_{u_{\beta}}} \ldots \hat{P_{u_1}} \ldots P_i^{n+2k-\beta} \sum_{\epsilon_1, \ldots, \epsilon_\beta \in J } M_*(\sigma_{u_1+\epsilon_1(k+1), \ldots, u_{\beta}+\epsilon_{\beta}(k+1)})
\end{equation*}
\begin{equation*} H_i^{(n+2k-\beta, n+k-\epsilon_1-\ldots-\epsilon_{\beta})\ \gamma(\alpha; \omega_1, \ldots, \omega_{k+1})}H_{i+1}^{(n+k-\epsilon_1-\ldots-\epsilon_{\beta},n)\ \gamma(\alpha; \omega'_1, \ldots, \omega'_{k+1})}  
\end{equation*}
\begin{equation*}
\mathrm{\quad  for\ } 1 \leq i<n, \  1 \leq k \leq n \mathrm{\quad  and\ } \alpha \in \PP(k+1);
\end{equation*}
where  $J=\{(\epsilon_1, \ldots, \epsilon_\beta) \in \{0,1\}^\beta \mid \beta-k\leq \epsilon_1+ \ldots + \epsilon_{\beta} \leq k \} $, for $l \in \{1, \ldots, k+1\}$:
$$
\omega_l=\left \lbrace
\begin{array}{lll}
0_\PP \mathrm{\quad for\ }  l \in \{\overline{(1+\epsilon_1)}(u_1-i+1), \ldots, \overline{(1+\epsilon_\beta)}(u_\beta-i+1) \} \setminus \{0\} \\
1_\PP \mathrm{\quad otherwise}
\end{array}
\right.
$$
where $\overline{x}$ is the class modulo $2$ of $x$
and 
$$
\omega'_l=\left \lbrace
\begin{array}{lll}
0_\PP \mathrm{\quad for\ }  l \in \{ \epsilon_1(u_1-i+1), \ldots, \epsilon_\beta(u_\beta-i+1) \} \setminus \{0\} \\
1_\PP \mathrm{\quad otherwise}
\end{array}
\right.
$$
$\sigma$  is\  the\ permutation\ of\ ${n+2k}$\ letters\ such\ that\
$Supp(\sigma)=\{ i+1, \ldots, i+2k\}$ and for $I=\{i, i+1, \ldots, i+2k, i+2k+1\}$ we have:
$$\sigma_{\mid I}=\begin{pmatrix} i&i+1&i+2& \ldots &i+k& i+k+1& i+k+2& \ldots& i+2k+1\\ i& i+2& i+4&  \ldots & i+2k & i+1&i+3& \ldots& i+2k+1\end{pmatrix}.$$ 
and
$\sigma_{u_1+\epsilon_1(k+1), \ldots, u_{\beta}+\epsilon_{\beta}(k+1)} $ is \ the \ permutation\ obtained\ from\ $\sigma$
removing the columns 
$
 \left(
  \begin{array}{ c}
    u_1+\epsilon_1(k+1)\\
    \ldots
  \end{array} \right)
$
$
 \left(
  \begin{array}{ c}
    u_2+\epsilon_2(k+1)\\
    \ldots
  \end{array} \right)
$
\ldots 
$
 \left(
  \begin{array}{ c}
    u_{\beta}+\epsilon_{\beta}(k+1)\\
    \ldots
  \end{array} \right)
$
and reindexing in order to have a permutation of $\underline{n+2k-\beta}$.  
\end{prop}
\begin{rem}
In the last family of equations we take the convention that 
$$H_p^{(m,m)\  \omega}=I_{m}^{1_\PP^{\times p-1},\  \omega,\  1_\PP^{ \times m-p}}\mathrm{\  for\ } 1 \leq p \leq m \mathrm{\ and\ } \omega \in \PP(1).$$
\end{rem}
\begin{rem}
The last family of equations which decomposes $H_i^{(n+k-1,n-1)\  \alpha} P_i^{n}$ as a sum of $(3^{k+1}-2)$-terms is the generalization of the equation $(M4)$ in Proposition $9.6$ in \cite{BDFP}, which decomposes $H_i^n P_i^n$ as a sum of $7$-terms. In fact, in  \cite{BDFP}, we have $k=1$. It is a consequence of the fact that $\mathcal{C}om$ is an operad generated by $\mathcal{C}om(2)$ which is a situation studied in the next section.
\end{rem}

\begin{proof}[Proof of Proposition \ref{presentation}]
According Proposition \ref{pour-gene}, it is sufficient to verify conditions $(1)$ and  $(2)$ in Definition \ref{PMack} on horizontal and vertical generators. 

According Lemma \ref{dble-iso-MT}, double isomorphisms in $\Omega(\PP)_2$ are pairs $(f, (f, 1_\PP^{\times n}))$ where $f$ is an isomorphism of $\underline{n}$.
Condition $(1)$ in Definition \ref{PMack} applied to the double isomorphism  $(\tau_i^n, (\tau_i^n, 1_{\PP}^{\times n}))$ becomes
\begin{equation} \label{aaaa}
T_i^n=(M^*(\tau_i^n, 1_\PP^{\times n}))^{-1}.
\end{equation}
Condition $(1)$ in Definition \ref{PMack} applied to the double isomorphism  $(Id_{\underline{n}}, (Id_{\underline{n}}, 1_{\PP}^{\times n}))$ gives the condition:
$$I_n^{1_\PP, \ldots, 1_\PP}=Id_{M(\underline{n})}$$
which is a consequence of $(\ref{aaaa})$ since:
$$I_n^{1_\PP, \ldots, 1_\PP}=M^*(Id_n,1_\PP^{\times n})=M^*(\tau_i^n, 1_{\PP}^{\times n}) M^*(\tau_i^n, 1_{\PP}^{\times n})=M_*(\tau_i^n)^{-1}M_*(\tau_i^n)^{-1}=(M_*(Id_n))^{-1}=Id_{M(\underline{n})}.$$

For condition $(2)$ in Definition \ref{PMack}, we consider the pairs $(h,v)$ of horizontal and vertical generators of $\Omega(\PP)_2$ having the same target and the corresponding unique square in $\Omega(\PP)_2$.
\begin{itemize}
\item Let $(f,\omega^f): \underline{n+p} \to \underline{n}$ (for $p \in \mathbb{N}$) be a vertical generator of $\Omega(\PP)_2$. The pair $(h,v)=(\tau_i^n, (f, \omega^f))$ corresponds to the unique square in $\Omega(\PP)_2$:
$$\vcenter{
\xymatrix{ \underline{n+p} \ar@{-->}[d]_ -{(\tilde{f}, \omega^{\tilde{f}})} \ar[r]^-{g}& \underline{n+p}\ar@{-->}[d]^-{(f, \omega^f)}\\
\underline{n} \ar[r]_-{\tau_i^n} &\underline{n}
} }$$ 
where $g$ is a bijection.
This square is equivalent, up to double isomorphisms, to the square:
$$\mathcal{D}=\quad \vcenter{
\xymatrix{ \underline{n+p} \ar@{-->}[d]_ -{(\tilde{f}, \omega^{\tilde{f}})\circ (g^{-1}, 1_\PP^{\times (n+p)})} \ar[r]^-{Id}& \underline{n+p}\ar@{-->}[d]^-{(\tau_i^n, 1_\PP^{\times n}) \circ (f, \omega^f)}\\
\underline{n} \ar[r]_-{Id} &\underline{n}.
} }$$ 
We have $Adm(\mathcal{D})=\{ \underline{n+p} \}$. So, in this case condition $(2)$ in Definition \ref{PMack} becomes:
\begin{equation} \label{aaa}
M^*\big((\tau_i^n, 1_\PP^{\times n}) \circ (f, \omega^f)\big) = M^*\big((\tilde{f}, \omega^{\tilde{f}})\circ (g^{-1}, 1_\PP^{\times (n+p)})\big).
\end{equation}

By Lemma \ref{uniq-square} the square associated to the pair $(Id, (\tau_i^n, 1_\PP^{\times n}) \circ (f, \omega^f))$ is unique up to double isomorphism so we have the following equality of vertical maps in $\Omega(\PP)_2$:
$$(\tilde{f}, \omega^{\tilde{f}})\circ (g^{-1}, 1_\PP^{\times (n+p)})=(\tau_i^n, 1_\PP^{\times n}) \circ (f, \omega^f).$$
We deduce that  relation $(\ref{aaa})$ is a consequence of the functoriality of $M^*$.

\item The pair $(h,v)=\big(s_i^{n}, (Id_{n-1},(\alpha_1, \ldots, \alpha_{n-1}))\big)$ corresponds to the unique square in $(\Omega(\PP))_2$:
$$\mathcal{D}=\quad \vcenter{
\xymatrix{ \underline{n} \ar@{-->}[d]_ -{\big(Id_n,(\alpha_1, \ldots, \alpha_{i-1}, \alpha_i, \alpha_i, \alpha_{i+1}, \ldots, \alpha_{n-1})\big)} \ar[r]^-{s_i^{n}}& \underline{n-1}\ar@{-->}[d]^-{\big(Id_{n-1},(\alpha_1, \ldots, \alpha_{i-1}, \alpha_i, \alpha_{i+1}, \ldots, \alpha_{n-1})\big)}\\
\underline{n} \ar[r]_-{s_i^{n}} &\underline{n-1}
} }$$ 
We have $Adm(\mathcal{D})=\{ \underline{n} \}$. So, in this case condition $(2)$ in Definition \ref{PMack} becomes:
$$I_{n-1}^{\alpha_1,\ldots \alpha_{n-1}}P_i^n=P_i^n I_n^{\alpha_1, \ldots, \alpha_{i-1}, \alpha_i , \alpha_i, \ldots, \alpha_n}$$

\item The pair $(h,v)=\big(s_i^{n}, (\tau_j^{n-1},{1_\PP}^n)\big)$ corresponds to the unique square in $(\Omega(\PP))_2$:
$$\quad \vcenter{
\xymatrix{ \underline{n} \ar@{-->}[d]_ -{(g,1_\PP^{\times n})} \ar[r]^-{k}& \underline{n-1}\ar@{-->}[d]^-{(\tau_j^{n-1},{1_\PP}^n)}\\
\underline{n} \ar[r]_-{s_i^{n}} &\underline{n-1}
} }$$ 
where $g$ is a bijection. This square is equivalent, up to double isomorphisms, to the square: 
$$\mathcal{D}=\quad \quad \vcenter{
\xymatrix{ \underline{n} \ar@{-->}[d]_ -{(Id,1_\PP^{\times n})} \ar[r]^-{k \circ g^{-1}}& \underline{n-1}\ar@{-->}[d]^-{(Id,{1_\PP}^n)}\\
\underline{n} \ar[r]_-{\tau_j^{n-1} \circ s_i^{n}} &\underline{n-1}.
} }$$ 
Using a similar argument that for the pair $(\tau^i_n, (f,\omega^f))$ we deduce that the relation obtained from this diagram is a consequence of the functoriality of $M_*$.

\item For $i<j$, the pair $(h,v)=(s_i^{n}, (s_j^{n+k-1,n-1},\alpha))$ corresponds to the unique square in $(\Omega(\PP))_2$:
$$\mathcal{D}=\quad \vcenter{
\xymatrix{ \underline{n+k} \ar@{-->}[d]_ -{(s_{j+1}^{n+k,n},\alpha)} \ar[r]^-{s_i^{n+k}}& \underline{n+k-1}\ar@{-->}[d]^-{(s_{j}^{n+k-1,n-1},\alpha)}\\
\underline{n} \ar[r]_-{s_i^{n}} &\underline{n-1}
} }$$ 
We have $Adm(\mathcal{D})=\{ \underline{n+k} \}$. So, in this case condition $(2)$ in Definition \ref{PMack} becomes:
$$ H_j^{(n+k-1,n-1)\  \alpha} P_i^{n}=P_i^{n+k} H_{j+1}^{(n+k,n)\  \alpha} $$
\item For $i>j$, the pair $(h,v)=(s_i^{n}, (s_j^{n+k-1,n-1},\alpha))$ corresponds to the unique square in $(\Omega(\PP))_2$:
$$\mathcal{D}=\quad \vcenter{
\xymatrix{ \underline{n+k} \ar@{-->}[d]_ -{(s_j^{n+k,n},\alpha)} \ar[r]^-{s_{k+i}^{n+k}}& \underline{n+k-1}\ar@{-->}[d]^-{(s_{j}^{n+k-1,n-1},\alpha)}\\
\underline{n} \ar[r]_-{s_i^{n}} &\underline{n-1}
}}$$ 
We have $Adm(\mathcal{D})=\{ \underline{n+k} \}$. So, in this case condition $(2)$ in Definition \ref{PMack} becomes:
$$ H_j^{(n+k-1,n-1)\  \alpha} P_i^{n}=P_{k+i}^{n+k} H_{j}^{(n+k,n)\  \alpha} $$

\item For $i=j$, the pair $(h,v)=(s_i^{n}, (s_i^{n+k-1,n-1},\alpha))$ corresponds to the unique square in $(\Omega(\PP))_2$:
$$\mathcal{D}=\quad \vcenter{
\xymatrix{ \underline{n+2k} \ar@{-->}[d]_ -{p_1} \ar[r]^-{p_2}& \underline{n+k-1}\ar@{-->}[d]^-{(s_{i}^{n+k-1,n-1},\alpha)}\\
\underline{n} \ar[r]_-{s_i^{n}} &\underline{n-1}
} }$$ 
where 
$$p_1=(s_{i+1}^{n+k,n}, \alpha)(s_{i}^{n+2k,n+k}, \alpha)$$
and
$$p_2=s_{i+k}^{n+k,n+k-1}s_{i+k-1}^{n+k-1,n+k-2} \ldots s_{i+1}^{n+2k-1, n+2k-2}s_{i}^{n+2k,n+2k-1} \sigma$$
where $\sigma$ is the permutation of $n+2k$ letters such that $Supp(\sigma)=\{ i+1, \ldots, i+2k\}$ and for $I=\{i, i+1, \ldots, i+2k, i+2k+1\}$ we have:
$$\sigma_{\mid I}=\begin{pmatrix} i&i+1&i+2& \ldots &i+k& i+k+1& i+k+2& \ldots& i+2k+1\\ i& i+2& i+4&  \ldots & i+2k & i+1&i+3& \ldots& i+2k+1\end{pmatrix}.$$
In the square $\D$ we have:
$$\underline{n+2k}\simeq \{(1,1), \ldots, (i-1,i-1), (i,i), (i,i+1), \ldots, (i,i+k),  (i+1,i), (i+1,i+1), \ldots, (i+1,i+k),$$
$$ (i+2, i+k+1), (i+3, i+k+2), \ldots, (n,n+k-1) \}$$
so we have $Card(Adm(\mathcal{D}))=\Sigma_{l=0}^{k+1} \binom{k+1}{l}  2^{k+1-l}-2=3^{k+1}-2$.

There are $2(k+1)$ $\D-$admissible subsets of $\underline{n+2k}$ having $n+2k-1$ elements which are:
$$A^{\{i+p\}}=\underline{n+2k}-\{ i+p\} \mathrm{\quad  for\ } p \in \{0, \ldots, 2k+1\}.$$

For $p \in \{0, \ldots, k\}$, we have:
$${p_1}_{\mid A^{\{i+p\}}}=(s_{i+1}^{n+k,n},\alpha) (s_{i}^{n+2k-1,n+k}, \gamma(\alpha; 1_\PP^{\times p}, 0_\PP, 1_\PP^{\times k-p}))$$
with the convention that $(s_{i}^{m,m}, \omega)=(Id_{\underline{m}}, (1_\PP^{\times i-1}, \omega, 1_\PP^{\times m-i}))$ for $1\leq i < m$ and  $\omega \in \PP(1)$ and
$${p_2}_{\mid A^{\{i+p\}}}=s_{i+k}^{n+k,n+k-1} \ldots s_{i+p+1}^{n+2k-(p+1), n+2k-(p+2)} s_{i+p-1}^{n+2k-p, n+2k-(p+1)}\ldots s_{i}^{n+2k,n+2k-1} \sigma_p$$
where $\sigma_p$ is the permutation of $n+2k-1$ letters
obtained\ from\ $\sigma$
removing the columns 
$
 \left(
  \begin{array}{ c}
   i+p\\
    \ldots
  \end{array} \right)
$
 in  $\sigma_{\mid I}$ and reindexing.

For $p \in \{k+1, \ldots, 2k+1\}$, we have:
$${p_1}_{\mid A^{\{i+p\}}}=(s_{i+1}^{n+k-1,n}, \gamma(\alpha; 1_\PP^{\times p-(k+1)}, 0_\PP, 1_\PP^{\times 2k+1-p})) (s_{i}^{n+2k-1,n+k-1}, \alpha)$$
with the convention that $(s_{i+1}^{m,m}, \omega)=(Id_{\underline{m}}, (1_\PP^{\times i}, \omega, 1_\PP^{\times m-i-1}))$ for $1\leq i < m$ and  $\omega \in \PP(1)$ and

$${p_2}_{\mid A^{\{i+p\}}1}=s_{i+k}^{n+k,n+k-1} \ldots s_{i+p+1}^{n+2k-(p+1), n+2k-(p+2)} s_{i+p-1}^{n+2k-p, n+2k-(p+1)}\ldots s_{i}^{n+2k,n+2k-1} \sigma_p$$
where $\sigma_p$ is the permutation of $n+2k-1$ letters obtained  from        \ $\sigma$
removing the column 
$
 \left(
  \begin{array}{ c}
    i+p\\
    \ldots
  \end{array} \right)
$
 in  $\sigma_{\mid I}$ and reindexing.

More generally, for $\beta \in \{1, \ldots, k+1\}$, $\{u_1, \ldots, u_\beta\} \subset \{i, \ldots, i+k \}$ such that  $i\leq u_1< \ldots< u_\beta \leq i+k$ and $\iota \in J=\{ (\epsilon_1, \ldots, \epsilon_\beta)  \in \{0,1\}^{\times \beta}\ \mid \beta-k \leq \epsilon_1+ \ldots+ \epsilon_\beta \leq k \}$ we consider the set:
$$U^\iota=\{(i+\epsilon_1, u_1), (i+\epsilon_2, u_2), \ldots, (i+\epsilon_\beta, u_\beta) \}.$$
Therefore, for $\beta\leq k$ (resp. $\beta=k+1$) there are $\binom{k+1}{\beta} 2^\beta$ (resp. $2^{k+1}-2$) $\D$-admissible subsets of $\underline{n+2k}$ having $n+2k-\beta$ elements which are 
$$A^{U^\iota}=\underline{n+2k}-U^\iota.$$
We have:
$${p_1}_{\mid A^{U^\iota}}=(s_{i+1}^{n+k-\Sigma_{i=1}^\beta\epsilon_i,n}, \gamma(\alpha; \omega'_1, \ldots, \omega'_{k+1})) (s_{i}^{n+2k-\beta,n+k-\Sigma_{i=1}^\beta \epsilon_i},  \gamma(\alpha; \omega_1, \ldots, \omega_{k+1}))$$
with the convention that 
$(s_{j}^{m,m}, \omega)=(Id_{\underline{m}}, (1_\PP^{\times j-1}, \omega, 1_\PP^{\times m-j}))$ for $1\leq j \leq m$ and  $\omega \in \PP(1)$ and
$${p_2}_{\mid A^{U^\iota}}=s_{i+k}^{n+k,n+k-1} \ldots \hat{s_{u_\beta}} \ldots \hat{s_{u_1}} \ldots s_{i}^{n+2k,n+2k-1} \sigma_{u_1+\epsilon_1(k+1), \ldots, u_{\beta}+\epsilon_{\beta}(k+1)}$$
where, for $l \in \{1, \ldots, k+1\}$,
$$
\omega_l=\left \lbrace
\begin{array}{lll}
0_\PP \mathrm{\quad for\ }  l \in \{\overline{(1+\epsilon_1)}(u_1-i+1), \ldots, \overline{(1+\epsilon_\beta)}(u_\beta-i+1) \} \setminus \{0\} \\
1_\PP \mathrm{\quad otherwise}
\end{array}
\right.
$$
where $\overline{x}$ is the class modulo $2$ of $x$
and 
$$
\omega'_l=\left \lbrace
\begin{array}{lll}
0_\PP \mathrm{\quad for\ }  l \in \{ \epsilon_1(u_1-i+1), \ldots, \epsilon_\beta(u_\beta-i+1) \} \setminus \{0\} \\
1_\PP \mathrm{\quad otherwise}
\end{array}
\right.
$$
and
$\sigma_{u_1+\epsilon_1(k+1), \ldots, u_{\beta}+\epsilon_{\beta}(k+1)} $ is \ the \ permutation\ of\ $n+2k-\beta$\ letters\ obtained\ from\ $\sigma$
removing the columns 
$
 \left(
  \begin{array}{ c}
    u_1+\epsilon_1(k+1)\\
    \ldots
  \end{array} \right)
$
$
 \left(
  \begin{array}{ c}
    u_2+\epsilon_2(k+1)\\
    \ldots
  \end{array} \right)
$
\ldots 
$
 \left(
  \begin{array}{ c}
    u_{\beta}+\epsilon_{\beta}(k+1)\\
    \ldots
  \end{array} \right)
$
and reindexing in order to have a permutation of $\underline{n+2k-\beta}$.  
So, in this case condition $(2)$ in Definition \ref{PMack} becomes:
\begin{equation}
H_i^{(n+k-1,n-1)\  \alpha} P_i^{n}= P_{i+k}^{n+k}P_{i+k-1}^{n+k+1} \ldots P_i^{n+2k} M_*(\sigma) H_i^{(n+2k, n+k)\ \alpha}H_{i+1}^{(n+k,n)\ \alpha}
\end{equation}
\begin{equation*}
+\sum_{\beta=1}^{k+1} \sum_{i\leq u_1<\ldots<u_\beta \leq i+k} P_{i+k}^{n+k} \ldots \hat{P_{u_{\beta}}} \ldots \hat{P_{u_1}} \ldots P_i^{n+2k-\beta} \sum_{\epsilon_1, \ldots, \epsilon_\beta \in J } M_*(\sigma_{u_1+\epsilon_1(k+1), \ldots, u_{\beta}+\epsilon_{\beta}(k+1)})
\end{equation*}
\begin{equation*} H_i^{(n+2k-\beta, n+k-\epsilon_1-\ldots-\epsilon_{\beta})\ \gamma(\alpha; \omega_1, \ldots, \omega_{k+1})}H_{i+1}^{(n+k-\epsilon_1-\ldots-\epsilon_{\beta},n)\ \gamma(\alpha; \omega'_1, \ldots, \omega'_{k+1})}  
\end{equation*}

\end{itemize}

\end{proof}


\subsection{Examples}
In this section we present several applications of Proposition \ref{presentation}. We give the presentation of polynomial functors from $\PP-alg$ for an operad $\PP$ generated by binary operations and the presentation of linear and quadratic functors.
\subsubsection{Presentation of polynomial functors from $\PP-alg$ for an operad $\PP$ generated by binary operations}
We suppose, in this section, that $\PP$ is generated by binary operations. For example, $\C$om and $\A$s are such operads. In this case, we have to consider less generators for $\Omega(\PP)$.
For all $i \in \N$ such that $1 \leq i \leq n-1$ we consider:
\begin{itemize}
\item
$(Id_n, \alpha_1, \ldots, \alpha_n): \underline{n} \to \underline{n}$ where $\forall i \in \{1, \ldots, n\}$, $\alpha_i \in \PP(1)$.
\item $s_i^{n, n-1}: \underline{n} \to \underline{n-1}$ the unique surjection preserving the natural order and such that:
$$s^{n, n-1}_i(i)=s_i^{n, n-1}(i+1)=i$$
$\forall \omega \in \PP(2)$ we obtain a morphism in $\Omega(\PP)$:
$$(s^{n, n-1}_i, \omega): \underline{n} \to \underline{n-1};$$ 
\item 
$(\tau^n_i, 1_{\PP}^{\times n} \in \PP(1)^{\times n}): \underline{n} \to \underline{n}$ where $\tau^n_i$ is  the transposition of $\mathfrak{S}_n$ which exchanges $i$ and $i+1$.
\end{itemize}
By an easy computation we obtain the following simplification of Proposition \ref{presentation}:

\begin{prop} \label{presentation-gal}
Let $\PP$ be an operad generated by binary operations and $M=(M_*, M^*)$ be a Janus functor $\Omega(\PP)_2 \to Ab$, $M$ is a pseudo-Mackey functor iff the following relations are satisfied:

\begin{align}
T_i^n&=(M^*(\tau_i^n, 1_{\PP}^{\times n}))^{-1} \mathrm{\quad  for\ } 1 \leq i \leq n;
\end{align}
\begin{align}
I_{n-1}^{\alpha_1,\ldots \alpha_{n-1}}P_i^n&=P_i^n I_n^{\alpha_1, \ldots, \alpha_{i-1}, \alpha_i, \alpha_i, \ldots, \alpha_n}\mathrm{\quad  for\ } 1 \leq i \leq n;
\end{align}
\begin{align}
H_j^{(n,n-1)\  \alpha} P_i^{n}&=P_i^{n+1} H_{j+1}^{(n+1,n)\  \alpha} \mathrm{\quad  for\ } 1 \leq i<j<n \mathrm{\quad  and\ } \alpha \in \PP(2);
\end{align}
\begin{align}
H_j^{(n,n-1)\  \alpha} P_i^{n}&=P_{1+i}^{n+1} H_{j}^{(n+1,n)\  \alpha} \mathrm{\quad  for\ } 1 \leq j<i<n \mathrm{\quad  and\ }  \alpha \in \PP(2);
\end{align}
\begin{equation}
H_i^{(n,n-1)\  \alpha} P_i^{n}= P_{i+1}^{n+1}P_{i}^{n+2} T_{i+1}^{n+2} H_{i}^{(n+2,n+1)\  \alpha} H_{i+1}^{(n+1,n)\  \alpha}+P_{i+1}^{n+1}T_{i}^{n+1} H_{i+1}^{(n+1,n)\  \alpha}
\end{equation}
\begin{equation*}
+P_{i+1}^{n+1}H_{i}^{(n+1,n)\  \alpha}+P_{i}^{n+1}H_{i+1}^{(n+1,n)\  \alpha}+P_{i}^{n+1}T_{i+1}^{n+1}H_{i}^{(n+1,n)\  \alpha}
\end{equation*}
\begin{equation*}
+I_n^{1_{\PP}^{\times (i-1)}, \gamma(\alpha; 1_\PP, 0_\PP), \gamma(\alpha; 0_\PP, 1_\PP), 1_{\PP}^{\times (n-i-1)}}+T_i^n
I_n^{1_{\PP}^{\times (i-1)}, \gamma(\alpha; 0_\PP, 1_\PP), \gamma(\alpha; 1_\PP, 0_\PP), 1_{\PP}^{\times (n-i-1)}}
\end{equation*}
\begin{equation*}
 \mathrm{\quad  for\ } 1 \leq i<n \mathrm{\quad  and\ } \alpha \in \PP(2).
\end{equation*}
\end{prop}

\begin{rem}
In this case, the last family of equations decomposes $H_i^{(n,n-1) \alpha} P_i^n$ as a sum of $7$ terms. For $\PP=\C$om, since $\C$om$(1)=\C$om$(2)=\underline{1}$ we recover the relation given in \cite{BDFP}. For $\PP=\A s$, since $\A$s$(1)=\underline{1}$ the last family of equations becomes:
\begin{equation}
H_i^{(n,n-1)\  \alpha} P_i^{n}= P_{i+1}^{n+1}P_{i}^{n+2} T_{i+1}^{n+2} H_{i}^{(n+2,n+1)\  \alpha} H_{i+1}^{(n+1,n)\  \alpha}+P_{i+1}^{n+1}T_{i}^{n+1} H_{i+1}^{(n+1,n)\  \alpha}
\end{equation}
\begin{equation*}
+P_{i+1}^{n+1}H_{i}^{(n+1,n)\  \alpha}+P_{i}^{n+1}H_{i+1}^{(n+1,n)\  \alpha}+P_{i}^{n+1}T_{i+1}^{n+1}H_{i}^{(n+1,n)\  \alpha} +Id_{M(\underline{n})}+T_i^n
\end{equation*}
\begin{equation*}
 \mathrm{\quad  for\ } 1 \leq i<n \mathrm{\quad  and\ } \alpha \in \A s(2)=\mathfrak{S}_2.
\end{equation*}

\end{rem}

\subsubsection{Presentation of linear functors}
\begin{prop}
There is an equivalence of  categories between $Lin(Free(\PP), Ab)$ and the category of diagrams:
$$\xymatrix{
\PP(1)\times M_e \ar[d]^{I_e}\\
M_e
}$$
where
\begin{itemize}
\item $M_e$ is an abelian group;
\item $\forall \omega \in \PP(1); \ I_e^{\omega}:=I_e(\omega,-):M_e \to M_{e}$ is a group morphism;
\end{itemize}
satisfying the following relations $\forall \omega \in \PP(1), \forall \omega' \in \PP(1) \quad$
\begin{enumerate}
\item $I_e^{1_{\PP}}=Id_{M_e}$;
\item $I_e^{\omega'} I_e^{\omega}=I_e^{\gamma(\omega; \omega')}$.
\end{enumerate}
\end{prop}

\begin{proof}
According Corollary \ref{equ-pol-cor}, $Lin(Free(\PP), Ab)$ is naturally equivalent to the category of pseudo-Mackey functors from $\Omega(\PP)$ to $Ab$ which have zero values on $\underline{0}$ and on sets of cardinality $>1$. So, it is sufficient to consider horizontal generator: $Id: \underline{1} \to \underline{1}$ and vertical generators $(Id, \omega \in \PP(1)): \underline{1} \to \underline{1}$ in $\Omega(\PP)_2 $. Let $M=(M_*, M^*): \Omega(\PP)_2 \to Ab$ be a Janus functor which has zero values  on $\underline{0}$ and on sets of cardinality $>1$. In the statement we replace $M(\underline{1})$ by $M_e$ and $I_1^\omega$ by $I_e^\omega$ for $\omega \in \PP(1)$. $M$ is a Janus functor iff we have the following relations:
$$M^*(Id,\omega') M^*(Id,\omega)=M^*((Id,\omega)(Id,\omega')) \quad i.e. \quad I_e^{\omega'} I_e^{\omega}=I^{\gamma(\omega',\omega)};$$
 $M$ is a pseudo-Mackey functor iff 
$$M^*(Id,1_\PP)=Id \quad i.e.\quad  I_e^{1_\PP}=Id_{M_e}.$$
(the other conditions obtained in Proposition \ref{presentation} are empty).
\end{proof}

\begin{rem}
Using the description of linear functors recalled in \cite{HV} we have: $Lin(Free(\PP), Ab) \simeq T_1U_{\mathcal{F}_\PP(1)}(\mathcal{F}_\PP(1)) \ti mod \simeq \Z[\PP(1)] \ti mod$ where $U_{\mathcal{F}_\PP(1)}$ is the reduced projective functor associated to $\mathcal{F}_\PP(1)$. We verify easily that this description coincide with that given in the last proposition.
\end{rem}

\subsubsection{Presentation of quadratic functors}
\begin{prop}
There is an equivalence of categories between $Quad(Free(\PP), Ab)$ and the category of diagrams:
$$\xymatrix{
\PP(1) \times \PP(1)\times M_{ee} \ar[d]_{I_{ee}} & \PP(1) \times M_e \ar[d]_{I_e}\\
M_{ee}  \ar@(ul,l)[]_T  \ar[r]_P& M_e\\
 & \PP(2) \times M_e \ar[ul]^H
}$$
where
\begin{itemize}
\item $M_e$ and $M_{ee}$ are abelian groups;
\item $\forall \omega \in \PP(1); \ I_e^{\omega}:=I_e(\omega,-):M_e \to M_{e}$ is a group morphism;
\item $\forall \omega \in \PP(1), \omega' \in \PP(1); \ I_{ee}^{\omega, \omega'}:=I_{ee}(\omega, \omega',-):M_{ee} \to M_{ee}$ is a group morphism;
\item $\forall \omega \in \PP(2); \ H^{\omega}:=H(\omega,-):M_e \to M_{ee}$ is a group morphism;
\item $P$ and $T$ are group morphisms;
\end{itemize}
satisfying the following relations $\forall (\omega, \omega', \omega_1, \omega'_1) \in \PP(1)^{\times 4}, \forall \alpha \in \PP(2), \forall \alpha' \in \PP(2)$
\begin{enumerate}
\item \label{quad-1} $T.T=Id_{M_{ee}}$;
\item \label{quad-2} $PT=P$;
\item \label{quad-3} $ I_e^{\omega'} I_e^{\omega}=I_e^{\gamma(\omega; \omega')}$;
\item \label{quad-4} $I_{ee}^{\omega, \omega'}I_{ee}^{\omega_1, \omega'_1}=I_{ee}^{\gamma(\omega_1, \omega), \gamma(\omega'_1, \omega')}$;
\item \label{quad-5} $I_{ee}^{\omega, \omega'}T=T I_{ee}^{\omega', \omega}$;
\item \label{quad-6} $H^{\alpha} I_e^{\omega}= H^{\gamma(\omega,\alpha)}$
\item \label{quad-7} $I_{ee}^{\omega, \omega'}H^{\alpha}=H^{\gamma(\alpha; \omega, \omega')}$;
\item \label{quad-8} $T H^{\alpha}=H^{\alpha. \tau}$ where $\tau$ is the transposition of $\mathfrak{S}_2$ exchanging $1$ and $2$;
\item \label{quad-9} $T=(M^*(\tau, (1_\PP, 1_\PP)))^{-1}$ 
\item \label{quad-10} $I_e^{\omega}P=P I_{ee}^{\omega, \omega}$ 
\item   \label{quad-11} $H^{\alpha}P=I_{ee}^{\gamma(\alpha; 1_{\PP}, 0_{\PP}),\gamma(\alpha; 0_{\PP}, 1_{\PP})}+TI_{ee}^{\gamma(\alpha; 0_{\PP}, 1_{\PP}),\gamma(\alpha; 1_{\PP}, 0_{\PP})}$.
\end{enumerate}
\end{prop}
\begin{proof}
According Corollary \ref{equ-pol-cor}, $Quad(Free(\PP), Ab)$ is naturally equivalent to the category of pseudo-Mackey functors from $\Omega(\PP)$ to $Ab$ which have zero values on $\underline{0}$ and on sets of cardinality $>2$. So, it is sufficient to consider horizontal generators: 
$$s_1^2: \underline{2} \to \underline{1}$$
$$\tau_1^2: \underline{2} \to \underline{2}$$
 and vertical generators 
 $$(Id, \omega \in \PP(1)): \underline{1} \to \underline{1}$$
 $$(Id, (\omega, \omega') \in \PP(1)^{\times 2}): \underline{2} \to \underline{2}$$
 $$(s_1^{2,1}, \omega \in \PP(2)): \underline{2} \to \underline{1}$$
 $$(\tau_1^2, (1_\PP, 1_\PP)): \underline{2} \to \underline{2}$$
  in $\Omega(\PP)_2 $. Let $M=(M_*, M^*): \Omega(\PP)_2 \to Ab$ be a Janus functor which has zero values  on $\underline{0}$ and on sets of cardinality $>2$. In the statement we replace $M(\underline{1})$ by $M_e$, $M(\underline{2})$ by $M_{ee}$, $I_1^\omega$ by $I_e^\omega$ for $\omega \in \PP(1)$, $I_2^{\omega, \omega'}$ by $I_{ee}^{\omega, \omega'}$ for $\omega \in \PP(1)$ and $\omega' \in \PP(1)$, $H_1^{(2,1) \omega}$ by $H^{\omega}$ for $\omega \in \PP(2)$, $P_1^2$ by $P$ and $T_1^2$ by $T$.
  
The conditions given in Lemma \ref{relations-hor} imply that $M_*$ is a covariant functor iff conditions (\ref{quad-1}) and (\ref{quad-2}) in the statement are satisfied. 

The conditions given in Lemma \ref{relations-vert} imply that $M^*$ is a contravariant functor iff conditions (\ref{quad-3}),(\ref{quad-4}), (\ref{quad-5}), (\ref{quad-6}), (\ref{quad-7}) and (\ref{quad-8}) in the statement are satisfied. 
  
 The conditions given in Proposition \ref{presentation} imply that
  $M$ is a pseudo-Mackey functor iff conditions (\ref{quad-9}),  (\ref{quad-10}) and  (\ref{quad-11}) are satisfied.
\end{proof}

In the particular case $\PP(1)=\{ 1_{\PP}\}$ we obtain the following simplification of the previous proposition: 

\begin{cor}
For $\PP$ an unitary set operad such that $\PP(1)=\{ 1_{\PP}\}$, there is an equivalence of categories between $Quad(Free(\PP), Ab)$ and the category of diagrams:
$$\xymatrix{
M_{ee}   \ar[r]_P& M_e\\
 & \PP(2) \times M_e \ar[ul]^H
}$$
where
\begin{itemize}
\item $M_e$ and $M_{ee}$ are abelian groups;
\item $\forall \omega \in \PP(2); \ H^{\omega}:=H(\omega,-):M_e \to M_{ee}$ is a group morphism;
\item $P$ is a group morphism;
\end{itemize}
satisfying the following relations $\forall \alpha \in \PP(2), \forall \alpha' \in \PP(2)$
\begin{enumerate}
\item $ PH^{\alpha}P=2P$;
\item $H^{\alpha}PH^{\alpha'}=H^{\alpha'}+H^{\alpha'.\tau}$ where $\tau$ is the transposition of $\mathfrak{S}_2$ exchanging $1$ and $2$.
\end{enumerate}
\end{cor}
\begin{proof}
If $\PP(1)=\{ 1_{\PP}\}$ the conditions (\ref{quad-3}), (\ref{quad-4}), (\ref{quad-5}), (\ref{quad-6}), (\ref{quad-7}) and (\ref{quad-10}) of the previous proposition are trivial.

Condition (\ref{quad-11}) becomes  $T=H^{\alpha}P-Id_{M_{ee}}$. So, we deduce that $T$ is determined by the other data.

Condition (\ref{quad-2}) becomes $ PH^{\alpha}P=2P$; condition (\ref{quad-8}) becomes $H^{\alpha}PH^{\alpha'}=H^{\alpha'}+H^{\alpha'.\tau}$ where $\tau$ is the transposition of $\mathfrak{S}_2$ exchanging $1$ and $2$ and condition (\ref{quad-1}) is a consequence of condition (\ref{quad-2}).
\end{proof}

For $\PP=\mathcal{C}om$, since $\mathcal{C}om(2)=\underline{1}$ we recover the equivalence between $Quad(ab,Ab)$ and quadratic $\Z$-module obtained by Baues in \cite{Baues}. Recall that a quadratic $\Z$-module is a diagram of abelian groups 
$$M_e \xrightarrow{H} M_{ee} \xrightarrow{P} M_e$$
satisfying $PHP=2P$ and $HPH=2H$.

For  $\PP=\mathcal{A}s$, since $\mathcal{A}s(2)=\mathfrak{S}_2=\{ 1, \tau \}$, we have two group morphisms from $M_e$ to $M_{ee}$: $H^1$ and $H^\tau$.
Relation $(1)$ in the corollary becomes:
$$(1 \ti 1)\quad PH^1P=2P$$
$$(1 \ti 2)\quad PH^\tau P=2P$$
and relation $(2)$ becomes:
$$(2 \ti 1)\quad H^1 P H^1=H^1+H^\tau$$
$$(2 \ti 2)\quad H^\tau P H^1=H^1+H^\tau$$
$$(2 \ti 3)\quad H^1 P H^\tau=H^1+H^\tau$$
$$(2 \ti 4)\quad H^\tau P H^\tau=H^1+H^\tau.$$
We deduce from relation $(2 \ti 1)$ that $H^\tau$ is determined by the other data. Then conditions $(1 \ti 2), (2 \ti 2), (2 \ti 3)$ and $(2 \ti 4)$ are consequences of relation $(1 \ti 1)$

So, we recover the equivalence between Quad(gr,Ab) and abelian square groups obtained by Baues and Pirashvili in \cite{Baues-Pira}. Recall that an abelian square group is a diagram of abelian groups 
$$M_e \xrightarrow{H} M_{ee} \xrightarrow{P} M_e$$
satisfying $PHP=2P$.


\bibliographystyle{amsplain}
\bibliography{biblio-poly}

\def\cprime{$'$}
\providecommand{\bysame}{\leavevmode\hbox to3em{\hrulefill}\thinspace}
\providecommand{\MR}{\relax\ifhmode\unskip\space\fi MR }
\providecommand{\MRhref}[2]{%
  \href{http://www.ams.org/mathscinet-getitem?mr=#1}{#2}
}
\providecommand{\href}[2]{#2}
\begin{thebibliography}{10}

\bibitem{Baues}
Hans-Joachim Baues, \emph{Quadratic functors and metastable homotopy}, J. Pure
  Appl. Algebra \textbf{91} (1994), no.~1-3, 49--107. \MR{MR1255923
  (94j:55022)}

\bibitem{BDFP}
Hans-Joachim Baues, Winfried Dreckmann, Vincent Franjou, and Teimuraz
  Pirashvili, \emph{Foncteurs polynomiaux et foncteurs de {M}ackey non
  lin\'eaires}, Bull. Soc. Math. France \textbf{129} (2001), no.~2, 237--257.
  \MR{MR1871297 (2002j:18004)}

\bibitem{Baues-Pira}
Hans-Joachim Baues and Teimuraz Pirashvili, \emph{Quadratic endofunctors of the
  category of groups}, Adv. Math. \textbf{141} (1999), no.~1, 167--206.
  \MR{MR1667150 (2000b:20073)}

\bibitem{Betley}
Stanislaw Betley, \emph{Stable {$K$}-theory of finite fields}, $K$-Theory
  \textbf{17} (1999), no.~2, 103--111. \MR{1696427 (2000d:18006)}

\bibitem{Bor2}
Francis Borceux, \emph{Handbook of categorical algebra. 2}, Encyclopedia of
  Mathematics and its Applications, vol.~51, Cambridge University Press,
  Cambridge, 1994, Categories and structures. \MR{MR1313497 (96g:18001b)}

\bibitem{BB}
Francis Borceux and Dominique Bourn, \emph{Mal'cev, protomodular, homological
  and semi-abelian categories}, Mathematics and its Applications, vol. 566,
  Kluwer Academic Publishers, Dordrecht, 2004. \MR{MR2044291 (2005e:18001)}

\bibitem{Dja-V}
Aur{\'e}lien Djament and Christine Vespa, \emph{Sur l'homologie des groupes
  orthogonaux et symplectiques \`a coefficients tordus}, Ann. Sci. \'Ec. Norm.
  Sup\'er. (4) \textbf{43} (2010), no.~3, 395--459. \MR{2667021 (2011h:20095)}

\bibitem{Ehres}
Charles Ehresmann, \emph{Cat\'egories structur\'ees}, Ann. Sci. \'Ecole Norm.
  Sup. (3) \textbf{80} (1963), 349--426. \MR{MR0197529 (33 \#5694)}

\bibitem{EML}
Samuel Eilenberg and Saunders Mac~Lane, \emph{On the groups {$H(\Pi,n)$}. {II}.
  {M}ethods of computation}, Ann. of Math. (2) \textbf{60} (1954), 49--139.
  \MR{MR0065162 (16,391a)}

\bibitem{Fied-Loday}
Zbigniew Fiedorowicz and Jean-Louis Loday, \emph{Crossed simplicial groups and
  their associated homology}, Trans. Amer. Math. Soc. \textbf{326} (1991),
  no.~1, 57--87. \MR{998125 (91j:18018)}

\bibitem{Pan-S}
Vincent Franjou, Eric~M. Friedlander, Teimuraz Pirashvili, and Lionel Schwartz,
  \emph{Rational representations, the {S}teenrod algebra and functor homology},
  Panoramas et Synth\`eses [Panoramas and Syntheses], vol.~16, Soci\'et\'e
  Math\'ematique de France, Paris, 2003. \MR{2117525 (2007m:55001)}

\bibitem{FFSS}
Vincent Franjou, Eric~M. Friedlander, Alexander Scorichenko, and Andrei Suslin,
  \emph{General linear and functor cohomology over finite fields}, Ann. of
  Math. (2) \textbf{150} (1999), no.~2, 663--728. \MR{1726705 (2001b:14076)}

\bibitem{Grandis-Pare}
Marco Grandis and Robert Pare, \emph{Limits in double categories}, Cahiers
  Topologie G\'eom. Diff\'erentielle Cat\'eg. \textbf{40} (1999), no.~3,
  162--220. \MR{1716779 (2000i:18007)}

\bibitem{HV}
M.~Hartl and C.~Vespa, \emph{Quadratic functors on pointed categories}, Adv.
  Math. \textbf{226} (2011), 3927--4010.

\bibitem{JMcotriples}
B.~Johnson and R.~McCarthy, \emph{Deriving calculus with cotriples}, Trans.
  Amer. Math. Soc. \textbf{356} (2004), no.~2, 757--803 (electronic).
  \MR{MR2022719 (2005m:18016)}

\bibitem{May-Kriz}
Igor K{\v{r}}{\'{\i}}{\v{z}} and J.~P. May, \emph{Operads, algebras, modules
  and motives}, Ast\'erisque (1995), no.~233, iv+145pp. \MR{1361938
  (96j:18006)}

\bibitem{Lindner}
Harald Lindner, \emph{A remark on {M}ackey-functors}, Manuscripta Math.
  \textbf{18} (1976), no.~3, 273--278. \MR{MR0401864 (53 \#5691)}

\bibitem{May-T}
J.~P. May and R.~Thomason, \emph{The uniqueness of infinite loop space
  machines}, Topology \textbf{17} (1978), no.~3, 205--224. \MR{MR508885
  (80g:55015)}

\bibitem{Passi}
Inder Bir~S. Passi, \emph{Group rings and their augmentation ideals}, Lecture
  Notes in Mathematics, vol. 715, Springer, Berlin, 1979. \MR{537126
  (80k:20009)}

\bibitem{Pira-Richter}
T.~Pirashvili and B.~Richter, \emph{Hochschild and cyclic homology via functor
  homology}, $K$-Theory \textbf{25} (2002), no.~1, 39--49. \MR{MR1899698
  (2003c:16011)}

\bibitem{PiraSS}
T.~I. Pirashvili, \emph{A spectral sequence of an epimorphism. {I}}, Trudy
  Tbiliss. Mat. Inst. Razmadze Akad. Nauk Gruzin. SSR \textbf{70} (1982),
  69--91. \MR{701219 (85f:18009)}

\bibitem{Pira}
Teimuraz Pirashvili, \emph{Polynomial approximation of {${\rm Ext}$} and {${\rm
  Tor}$} groups in functor categories}, Comm. Algebra \textbf{21} (1993),
  no.~5, 1705--1719. \MR{MR1213983 (94d:18020)}

\bibitem{Pira-Dold}
\bysame, \emph{Dold-{K}an type theorem for {$\Gamma$}-groups}, Math. Ann.
  \textbf{318} (2000), no.~2, 277--298. \MR{MR1795563 (2001i:20112)}

\bibitem{Pira-PROP}
\bysame, \emph{On the {PROP} corresponding to bialgebras}, Cah. Topol. G\'eom.
  Diff\'er. Cat\'eg. \textbf{43} (2002), no.~3, 221--239. \MR{MR1928233
  (2003i:18012)}

\bibitem{pira-PS}
\bysame, \emph{Introduction to functor homology}, Rational representations, the
  {S}teenrod algebra and functor homology, Panor. Synth\`eses, vol.~16, Soc.
  Math. France, Paris, 2003, pp.~1--26. \MR{2117526}

\bibitem{Popescu}
N.~Popescu, \emph{Abelian categories with applications to rings and modules},
  Academic Press, London, 1973, London Mathematical Society Monographs, No. 3.
  \MR{MR0340375 (49 \#5130)}

\bibitem{Schur}
I.~Schur, \emph{{\"Uber eine Klasse von Matrizen, die sich einer gegebenen
  Matrix zuordnen lassen.}}, Ph.D. thesis, {Diss. Berlin. 76 S }, 1901.

\bibitem{Slom}
Jolanta S{\l}omi{\'n}ska, \emph{Dold-{K}an type theorems and {M}orita
  equivalences of functor categories}, J. Algebra \textbf{274} (2004), no.~1,
  118--137. \MR{2040866 (2005c:18002)}

\bibitem{Stanley}
Richard~P. Stanley, \emph{Enumerative combinatorics. {V}ol. 1}, Cambridge
  Studies in Advanced Mathematics, vol.~49, Cambridge University Press,
  Cambridge, 1997, With a foreword by Gian-Carlo Rota, Corrected reprint of the
  1986 original. \MR{MR1442260 (98a:05001)}

\bibitem{Vespa1}
Christine Vespa, \emph{Generic representations of orthogonal groups: the
  functor category $\mathcal{ F}_{\rm quad}$}, J. Pure Appl. Algebra
  \textbf{212} (2008), no.~6, 1472--1499. \MR{MR2391661 (2008m:18001)}

\end{thebibliography}

\end{document}